\documentclass[11pt]{amsart}
\usepackage[mathscr]{euscript}
\usepackage[utf8]{inputenc}

\usepackage[margin=1in]{geometry}

\usepackage{enumitem} 
\usepackage{tikz} 
\usepackage{tikz-cd} 
\usetikzlibrary{arrows,calc,decorations.markings}
\usepackage{thmtools,thm-restate}
\usepackage{hyperref} 
\usepackage[capitalise,nosort]{cleveref} 
\usepackage{amsmath}
\usepackage{amssymb}
\usepackage{bbm}
\usepackage{mathtools}
\usepackage{stmaryrd} 
\usepackage{xparse}
\usepackage{blindtext}
\usepackage{cite}

\usepackage[T1]{fontenc}

\makeatletter
\newcommand{\@bbify}[1]{
  \ifcsname b#1\endcsname
  \message{WARNING: Overwriting b#1 with blackboard letter!}
  \fi
  \expandafter\edef\csname b#1\endcsname
  {\noexpand\ensuremath{\noexpand\mathbb #1}\noexpand\xspace}}
\newcommand{\@calify}[1]{
  \ifcsname c#1\endcsname
  \message{WARNING: Overwriting c#1 with calligraphic letter!}
  \fi 
  \expandafter\edef\csname c#1\endcsname
  {\noexpand\ensuremath{\noexpand\mathcal #1}\noexpand\xspace}}
\newcommand{\@bfify}[1]{
  \ifcsname bf#1\endcsname
  \message{WARNING: Overwriting c#1 with bold letter!}
  \fi
  \expandafter\edef\csname bf#1\endcsname
  {\noexpand\ensuremath{\noexpand\mathbf #1}\noexpand\xspace}}
\newcounter{@letter}\stepcounter{@letter}
\loop\@bbify{\Alph{@letter}}\@calify{\Alph{@letter}}\@bfify{\Alph{@letter}}
\ifnum\the@letter<26\stepcounter{@letter}\repeat
\makeatother

\newenvironment{tz}{\begin{center}\begin{tikzpicture}[scale=1]}{\end{tikzpicture}\end{center}}

\tikzstyle{d}=[double distance=.3ex]
\tikzstyle{w}=[preaction={draw=white, -,line width=4pt}]

\newcounter{diagram}

\tikzset{over/.style={auto=false,fill=white,inner sep=1.5pt, minimum size=0, outer sep=0}, 
pro/.style={postaction={decorate,decoration={
        markings,
        mark=at position .5 with {\node at (0,0) {$\bullet$};}
      }},
      inner sep=1ex,
      },n/.style={double equal sign distance, -implies},t/.style={double distance=2.5pt, -implies, postaction={draw,-}},
  }
\tikzset{%
node distance=1.5cm, la/.style={scale=0.8}, rr/.style={xshift=1.5cm},
space/.style={xshift=.5cm},
    symbol/.style={%
        draw=none,
        every to/.append style={%
            edge node={node [sloped, allow upside down, auto=false]{$#1$}}},
            
    }
}
  
\def\cellslide{0.5}
\def\celllength{.2cm}

\NewDocumentCommand{\cell}{ O{} O{m} O{\cellslide} O{\celllength} m m m }{
  \coordinate (mid) at ($({#5})!{#3}!({#6})$);
  \coordinate (start) at ($(mid)!{#4}!({#5})$);
  \coordinate (end) at ($(mid)!{#4}!({#6})$);
  \draw[#2] (start) to node
  [inner sep=4pt,outer sep=0,minimum size=0,#1]{{#7}} (end);
}

\newtheorem{thm}{Theorem}[subsection] 
\newtheorem*{thm*}{Theorem}
\newtheorem{cor}[thm]{Corollary}
\newtheorem{lemma}[thm]{Lemma}
\newtheorem{prop}[thm]{Proposition}

\newtheorem*{unnumberedtheorem}{Theorem}

\newtheorem*{unnumbereddefinition}{Definition}

\theoremstyle{definition}
\newtheorem{defn}[thm]{Definition}
\newtheorem{notation}[thm]{Notation}
\newtheorem{constr}[thm]{Construction}

\newtheorem{data}[thm]{Setup}

\theoremstyle{remark}
\newtheorem{rmk}[thm]{Remark}

\newtheorem{ex}[thm]{Example}

\crefname{lem}{Lemma}{Lemmas}
\crefname{thm}{Theorem}{Theorems}
\crefname{defn}{Definition}{Definitions}
\crefname{data}{Setup}{Setups}
\crefname{prop}{Proposition}{Propositions}
\crefname{rmk}{Remark}{Remarks}
\crefname{cor}{Corollary}{Corollaries}
\crefname{ex}{Example}{Examples}
\crefname{notation}{Notation}{Notations}
\crefname{constr}{Construction}{Constructions}
\crefname{recall}{Recall}{Recalls}
\crefname{descr}{Description}{Descriptions}
\crefname{para}{\textsection}{\textsection\textsection}

\newlist{rome}{enumerate}{7}
\setlist[rome]{label=(\roman*)}

\newlist{alphabet}{enumerate}{7}
\setlist[alphabet]{label=(\alph*)}

\newcommand{\pushout}[1]{\node at ($({#1})-(10pt,-10pt)$) {$\ulcorner$};}
\newcommand{\pullback}[1]{\node at ($({#1})+(10pt,-10pt)$) {$\lrcorner$};}
\NewDocumentCommand{\punctuation}{ m m O{5pt} }{\node at ($(#1.east)-(0,#3)$) {#2};}

\newcommand{\defThn}{\ensuremath{\theta}}
\newcommand{\defS}{k}

\NewDocumentCommand\Thn{O{n-1}}{\ensuremath{\Theta}_{{#1}}}
\NewDocumentCommand\Thnsset{O{n-1}}{\mathit{s}\set^{\ensuremath{\Theta_{\scriptscriptstyle{#1}}^{\op}}}}
\NewDocumentCommand\Thnssset{O{n-1}}{\mathit{ss}\set^{\ensuremath{\Theta_{\scriptscriptstyle{#1}}^{\op}}}}
\NewDocumentCommand\CSThn{O{n-1}}{{\ensuremath{\scriptscriptstyle(\infty,{#1})}}}
\NewDocumentCommand\AraThn{O{n-1}}{{\ensuremath{\scriptscriptstyle(\infty,{#1})}}}

\newcommand{\cat}{\cC\!\mathit{at}}

\newcommand{\set}{\cS\!\mathit{et}}
\newcommand{\sset}{\mathit{s}\set}
\newcommand{\Pcat}{\mathcal{P}\cat}
\newcommand{\DThn}{\Delta\times \Thn}
\newcommand{\DThnS}{\DThn\times \Delta}

\newcommand{\Dop}{\Delta^{\op}}

\newcommand{\MSspace}{\sset_{\Kan}}

\newcommand{\injThnspace}{(\sset_{\Kan})^{\Thnop}_\inj}

\newcommand{\MSThnsset}{\Thnsset_{\CSThn}}
\newcommand{\MSThncat}{\Thnsset_{\CSThn}\text{-}\cat}
\newcommand{\injsThnsset}{(\Thnsset_{\CSThn})^{\Delta^{\op}}_\inj}
\newcommand{\SegsThnsset}{\sThnsset_{\mathrm{dbl}\CSThn}}
\newcommand{\CSSsThnsset}{\sThnsset_{\CSThn[n]}}

\newcommand{\Dset}{\set^{\Dop}}
\newcommand{\Dsset}{\sset^{\Dop}}
\newcommand{\MSDsset}{\sset^{\Dop}_{\CSThn[1]}}

\newcommand{\diag}{\mathrm{diag}}

\newcommand{\SSset}{\mathit{s}\sset}

\newcommand{\sThnsset}{\mathit{s}\set^{\Thnop\times \Dop}}
\newcommand{\sThnssetslice}[1]{\sThnsset_{\;\;\;\;\;\;\;\;\;\;/{#1}}}

\newcommand{\Thncat}{\Thnsset\text{-}\cat}
\newcommand{\sThncat}{\Thnssset\text{-}\cat}
\newcommand{\Thnop}{{\Theta_{\scriptscriptstyle n-1}^{\op}}}
\newcommand{\DThnop}{\Dop\times \Thnop}

\newcommand{\MSrightfib}[1]{(\sThnssetslice{{#1}})_{\CSThn}^\mathrm{right}}
\newcommand{\MSleftfib}[1]{(\sThnssetslice{{#1}})_{\CSThn}^\mathrm{left}}

\newcommand{\pcatThn}{\Pcat(\Thnsset)}

\newcommand{\pcatinj}{\Pcat(\Thnsset_{\CSThn})}

\NewDocumentCommand\Sp{O{m}}{Sp[{#1}]}
\NewDocumentCommand\repD{O{m}}{F[{#1}]}
\NewDocumentCommand\repS{O{\defS}}{\ensuremath{\Delta}[{#1}]}

\newcommand{\slice}[2]{{#1}_{\sslash {#2}}}
\newcommand{\joinslice}[2]{{#1}_{\sslash^\star{#2}} }
\newcommand{\sliceunder}[2]{{}^{{#2} \sslash}{#1}}
\newcommand{\joinsliceunder}[2]{{}^{{#2} \sslash^\star}{#1}}

\newcommand{\pushprodstar}{\,\widehat{\star}\,}  
\newcommand{\pushprod}{\,\widehat{\times}\,}  
\NewDocumentCommand\repThn{O{\defThn}}%
  {\Thn{[{#1}]}}
\NewDocumentCommand\repDThn{O{m} O{\defThn}}%
  {F[{#1}]\times \Thn{[{#2}]}}
  \NewDocumentCommand\repDThnS{O{m} O{\defThn} O{\defS}}%
  {F[{#1},{#2},{#3}]}
\NewDocumentCommand\repThnS{O{\defThn} O{\defS}}%
  {\Thn[{#1}]\times \ensuremath{\Delta}[{#2}]}
  
\NewDocumentCommand\HomSh{O{m} O{X}}{{\Hom_{\Sh_{{#1}}{#2}}}}
\NewDocumentCommand\HomSigma{O{m} O{X}}{{\Hom_{\ensuremath{\Sigma}_{{#1}}{#2}}}}

\newcommand{\Ch}{\mathfrak{C}}
\newcommand{\Nh}{\mathfrak{N}}
\newcommand{\St}{\mathrm{St}}
\newcommand{\Un}{\mathrm{Un}}
\newcommand{\mapDelta}{\mu}

\DeclareMathOperator{\Ob}{Ob}
\DeclareMathOperator{\colim}{colim}

\DeclareMathOperator{\Ho}{Ho}

\DeclareMathOperator{\Hom}{Hom}
\DeclareMathOperator{\Map}{Map}
\DeclareMathOperator{\op}{op}

\DeclareMathOperator{\id}{id}

\newcommand{\Kan}{\CSThn[0]}
\newcommand{\inj}{\mathrm{inj}}
\newcommand{\proj}{\mathrm{proj}}
\newcommand{\sCat}{{\sset\text{-}\cat}}

\newcommand{\Sh}{\mathfrak{S}}

\newcommand{\NL}{N^h}
\newcommand{\CL}{c^h}

\NewDocumentCommand\PP{O{m} O{X} O{Y}}{{P_{#1}({#2}\hookrightarrow {#3})}}
\NewDocumentCommand\Cone{O{f} O{\ensuremath{\mapDelta}}}{{\mathrm{Cone}_{#2}({#1})}}
\NewDocumentCommand\newGbar{O{f} O{i} O{\ensuremath{\mapDelta}}}{{\overline{\cF}^{\,{#2}}_{#3}({#1})}}
\NewDocumentCommand\newG{O{f} O{i} O{\ensuremath{\mapDelta}}}{{\cF^{\,{#2}}_{#3}({#1})}}
\NewDocumentCommand\Hbar{O{i} O{m+1}}{{\overline{H}^{\,{#1}}_{#2}}}

\usepackage{blindtext}

\renewcommand\thepart{\Roman{part}}

\makeatletter
\renewcommand\part{%
  \par
  \addvspace{4ex}%
  \@afterindenttrue
  \secdef\@part\@spart
}

\def\@part[#1]#2{%
    \ifnum \c@secnumdepth >\m@ne
      \refstepcounter{part}%
      
      \addcontentsline{toc}{section}{\hspace{-.5cm} \bfseries\thepart.\hspace{1em}#1}%
    \else
      \addcontentsline{toc}{section}{#1}%
    \fi
    {\parindent \z@ \raggedright
     \interlinepenalty \@M
     \normalfont
     \thispagestyle{empty}
     \ifnum \c@secnumdepth >\m@ne
      \centering\large\textsc{\textbf{\thepart.}}\nobreakspace
     \fi
     \centering\large\textsc{\textbf{#2}}
     \par}%
    \nobreak
    \vskip .3cm
    \@afterheading}
\def\@spart#1{%
      \addcontentsline{toc}{part}{#1}%
    {\parindent \z@ \raggedright
     \interlinepenalty \@M
     \normalfont
     \thispagestyle{plain}
     \centering\large\textsc{\textbf{#1}}
     \par}%
    \nobreak
    \vskip .3cm
    \@afterheading}
\makeatother

\title{$(\infty,n)$-Limits II: Comparison across models}

\author{Lyne Moser}
\address{Fakultät für Mathematik, Universität Regensburg, Regensburg, Germany}
\email{lyne.moser@ur.de}

\author{Nima Rasekh}
\address{Max Planck Institute for Mathematics, Bonn, Germany}
\email{rasekh@mpim-bonn.mpg.de}

\author{Martina Rovelli}
\address{Department of Mathematics and Statistics, University of Massachusetts Amherst, Amherst, USA}
\email{mrovelli@umass.edu}

\begin{document}

\begin{abstract}
We show that the notion of $(\infty,n)$-limit defined using the enriched approach and the one defined using the internal approach coincide.
We also give explicit constructions of various double $(\infty,n-1)$-categories implementing various join constructions, slice constructions and cone constructions, and study their properties. We further prove that key examples of $(\infty,n)$-categories are (co)complete.
\end{abstract}

\maketitle

\setcounter{tocdepth}{1}
\tableofcontents

\section*{Introduction}

The role of $(\infty,n)$-categories in many branches of modern mathematics has become apparent in the past couple of decades, and the universal property of a number of important constructions is encoded as that of a -- possibly weighted -- (co)limit of an appropriate diagram between $(\infty,n)$-categories. Examples range from sheaf and descent properties in (derived) geometry via colimits to the category of monads, described as a weighted limit; see the introduction of \cite{MRR3} for further examples. Hence, it is crucial to establish a model independent, meaningful and usable theory of limits and colimits in this context.

In this paper we establish the comparison between the two main -- potentially competing -- approaches to the theory of limits in the context of $(\infty,n)$-categories one could envision. On the one hand, there is the ``enriched'' approach, which is available in the framework of \cite[\textsection5]{ShulmanHomotopyLimits} and essentially relies on having access to the hom $(\infty,n-1)$-categories of the $(\infty,n)$-categories of diagrams. On the other hand, there is the ``internal'' approach, which we introduced in \cite[Definition 1.4.24]{MRR3}, and relies on having access to slice and cone constructions in the world of double $(\infty,n-1)$-categories, i.e., internal categories to $(\infty,n-1)$-categories. The former approach is closer to the intuitive viewpoint adopted in the context of strict higher category theory, whereas the latter is an appropriate generalization of the practical viewpoint taken in the context of $(\infty,1)$-categories. 

Analog comparison results have been treated before under special specific assumptions (specific values of~$n$, particular weights, or further assumptions on the $(\infty,n)$-categories considered). These include: \cite{RVLimits} in the case of terminal weight---encoding ordinary (co)limits---when $n=1$, \cite{RovelliLimits} for a partial comparison in the case of general weights when $n=1$, \cite[\textsection5.3]{GHL2} in the case of $(\infty,2)$-categories arising from appropriately enriched model categories when $n=2$, \cite[\textsection4.2]{AGH} in the case of weights encoding partially (op)lax limits when $n=2$. We hereby give a completely general treatment, which accounts for an arbitrary weight and an arbitrary~$n\geq1$. 

\subsection*{The conical case}

Given a diagram $F\colon \cJ\to \cC$ of $(\infty,n)$-categories, let us recall the two types of universal property expected for an object $\ell$ of $\cC$ to be the $(\infty,n)$-limit of $F$.

Assuming one has access to the hom $(\infty,n-1)$-categories for both $\cC$ and $\cC^{\cJ}$, as a special case of the framework from \cite{ShulmanHomotopyLimits} one obtains a definition of limit which appropriately translates the idea that there is a correspondence between cones over $F$ and maps in $\cC$ to the object $\ell$:
\begin{unnumbereddefinition}[Enriched approach]
    An object $\ell\in\cC$ is a limit of $F$ if there is an equivalence of hom $(\infty,n-1)$-categories
\begin{equation}
\label{defEnriched}
\cC(c,\ell)\simeq\cC^{\cJ}(\Delta c,F) 
\end{equation}
natural in objects $c\in \cC$. 
\end{unnumbereddefinition}

While this definition is quite transparently capturing the correct idea, because of the naturality in $c$ it can only reasonably be exploited when working with $(\infty,n)$-categories presented by categories strictly enriched over $(\infty,n-1)$-categories, a quite limiting assumption in practice. Even then, computing the correct homotopy types of the hom $(\infty,n)$-categories featuring in \eqref{defEnriched} is notably challenging (typically at least as complicated as computing cofibrant replacements of diagrams), even for particularly simple values of $\cJ$.

Instead, we proposed in \cite{MRR3} an alternative approach to the notion of limit for $(\infty,n)$-categories, which encodes the idea that the limit comes with a cone that is appropriately terminal, and which is more in line with the literature for $n=1$ \cite{JoyalVolumeII,LurieHTT,GHNlax} and for $n=2$ in the strict case \cite{GPdoublelimits,cM1,cM2}; cf.~also \cite[\textsection4.2.3]{Loubaton4}. We defined the double $(\infty,n-1)$-category of cones $\slice{\cC}{F}$ and the double $(\infty,n-1)$-slice $\slice{(\slice{\cC}{F})}{(\ell,\lambda)}$ over an object $(\ell,\lambda)$ of $\slice{\cC}{F}$. We then set:

\begin{unnumbereddefinition}[Internal approach]
    An object $\ell\in\cC$ is a limit of $F$ if there exists an element $\lambda$ in the fiber of $\slice{\cC}{F}\to \cC$ at $\ell$ such that the canonical projection 
\begin{equation}
\label{defInternal}
\slice{(\slice{\cC}{F})}{(\ell,\lambda)}\xrightarrow{\simeq} \slice{\cC}{F}
\end{equation}
is an equivalence of double $(\infty,n-1)$-categories.
\end{unnumbereddefinition}

This definition is quite accessible in practice, and in forthcoming work \cite{MRR5} we will focus on developing the theory of $(\infty,n)$-limits, enhancing to the case of general $n$ many of the arguments that work for the cases $n=1,2$ from e.g.~\cite{LurieHTT,GHL2}.

We show as \cref{EquivalenceOfLimitsConical} that the two approaches agree:

\begin{unnumberedtheorem}
An object $\ell\in \cC$ is a limit of $F$ in the sense of \eqref{defEnriched} if and only if it is one in the sense~of~\eqref{defInternal}.
\end{unnumberedtheorem}

The result is in fact a special instance of a more general framework, which we now discuss.

\subsection*{The weighted case}

In higher category theory one is often interested in studying cones of a more exotic shape---specified by something called a \emph{weight}---and the summit of the universal cone of such shape would then be the corresponding weighted limit. When working with $(\infty,n)$-categories presented by categories strictly enriched over $(\infty,n-1)$-categories, the nature of the weight is expressed as a functor $W\colon \cJ\to \cat_{(\infty,n-1)}$ which assigns to each object of $\cJ$ the $(\infty,n-1)$-category of legs on that object for the type of cones that are being considered. The case of ordinary cones is recovered when considering the terminal weight and gives the notion of \emph{conical} limits. The case of lax cones when $n=2$ corresponds to the weight which assigns to each object of $\cC$ the corresponding slice and gives the notion of \emph{lax} limits. The case of lax squares when $n=2$ and $\cJ$ is the cospan shape corresponds to the weight given by the diagram of categories $\{[0]\xrightarrow{0}[1]\xleftarrow{1}[0]\}$ and gives the notion of \emph{comma} objects.

Following \cite{ShulmanHomotopyLimits}, one sets:
\begin{unnumbereddefinition}[Enriched approach]
    An object $\ell$ of $\cC$ is a $W$-weighted limit of $F$ if there is an equivalence of hom $(\infty,n-1)$-categories
\begin{equation}
\label{defEnrichedWeighted}
\cC(c,\ell)\simeq [\cJ,\cat_{(\infty,n-1)}](W,\cC(c,F-))
\end{equation}
natural in objects $c\in \cC$.
\end{unnumbereddefinition}

When working with a general weight, in the internal approach, the nature of the weight is encoded by a map (whose homotopy theory is that of appropriate fibrations). The conical case is recovered as the case of the identity map at $\cJ$, while for $n=2$ the lax case is recovered as the case of the unstraightening construction of the weight encoding lax cones. We then set, as \cref{def:weightedlimit}:

\begin{unnumbereddefinition}[Internal approach]
 An object $\ell\in \cC$ is a $p$-weighted limit of $f$ if there is an element $\lambda$ in the fiber of $\slice{\cC}{F\circ p}\to \cC$ at $\ell$ such that the canonical projection
\begin{equation}
\label{defInternalWeighted}
\slice{(\slice{\cC}{F\circ p})}{(\ell,\lambda)}\xrightarrow{\simeq}\slice{\cC}{F\circ p}
\end{equation}
is an equivalence of double $(\infty,n-1)$-categories.
\end{unnumbereddefinition}

Again, we are able to show---as \cref{EquivalenceOfLimitsWeighted}---that the two agree:
\begin{unnumberedtheorem}
An object $\ell$ of $\cC$ is a weighted limit of $F$ in the sense of \eqref{defEnrichedWeighted} if and only if it is in the sense of \eqref{defInternalWeighted}.
\end{unnumberedtheorem}

The main goal of this paper is to make this theorem into a precise mathematical statement and we will now delve into the details.

\subsection*{The main result}

We now explain the precise form of the main theorem of the paper. Similar implementations have been pursued in other contexts for the case $n=1$ \cite{RVLimits,RovelliLimits} and for the case $n=2$ \cite{GHL2}.

First, let $\MSThnsset$ denote the model structure on the category $\Thnsset$ for $(\infty,n-1)$-categories given by the complete Segal $\Thn$-spaces from \cite{RezkTheta}. We consider two different models of $(\infty,n)$-categories: the injective-like model structure $\pcatinj$ on the subcategory $\pcatThn$ of $\sThnsset$ spanned by those $X$ such that $X_0$ is discrete for Segal category objects in complete Segal $\Thn$-spaces from \cite{br1}, and the model structure $\MSThncat$ on the category $\Thncat$ of strictly enriched categories in $\Thnsset$ for enriched categories that are locally complete Segal $\Thn$-categories from \cite{br1}. Consider the homotopy coherent categorification-nerve Quillen equivalence from \cite{MRR1}. 
\begin{tz}
\node[](1) {$\pcatinj$}; 
\node[right of=1,xshift=2.9cm](2) {$\MSThncat$}; 

\draw[->] ($(2.west)-(0,5pt)$) to node[below,la]{$\Nh$} ($(1.east)-(0,5pt)$);
\draw[->] ($(1.east)+(0,5pt)$) to node[above,la]{$\Ch$} ($(2.west)+(0,5pt)$);

\node[la] at ($(1.east)!0.5!(2.west)$) {$\bot$};
\end{tz}

\begin{itemize}[leftmargin=*]
\item[\guillemotright] \emph{The indexing $(\infty,n)$-category $\cJ$:}
In the internal approach, the indexing shape can be realized as an object $J$ of $\pcatThn$.
Because of the Quillen equivalence $\Ch\dashv\Nh$, in the enriched approach the indexing shape can be realized without loss of generality as an object $\Ch J$ in $\MSThncat$. 
\begin{center}
\fbox{\quad$\Ch J \quad \leftrightsquigarrow \quad J$\quad}
\end{center}
\item[\guillemotright] \emph{The ambient $(\infty,n)$-category $\cC$}:
In the enriched approach, the value $(\infty,n)$-category can be realized as a fibrant object $\cC$ of $\MSThncat$. Because of the Quillen equivalence $\Ch\dashv\Nh$, in the internal approach the indexing shape can be realized without loss of generality as a fibrant object $\Nh\cC$ of $\pcatinj$.
\begin{center}
\fbox{\quad$\cC \quad \leftrightsquigarrow \quad \Nh\cC$\quad}
\end{center}
    \item[\guillemotright]
    \emph{The $(\infty,n)$-diagram $F\colon\cJ\to\cC$}: In the enriched approach, such a diagram can be realized as a morphism $F\colon \Ch J\to \cC$ in $\MSThncat$, which corresponds under the Quillen adjunction $\Ch \dashv \Nh$ to a morphism $F^\flat\colon J\to \Nh\cC$ in $\pcatinj$ in the internal approach.
    \begin{center}
\fbox{\quad$F\colon \Ch J\to \cC \quad \leftrightsquigarrow \quad F^\flat\colon J\to \Nh \cC$\quad}
\end{center}
\end{itemize}
Next, let $[\Ch J^{\op},\MSThnsset]_\proj$ denote the projective model structure on the category of enriched functors $\Ch J^{\op}\to\cC$ from \cite{Moserinj}, and let $\MSrightfib{J}$ denote the model structure on the slice category $\sThnssetslice{J}$ for double $(\infty,n-1)$-right fibrations from \cite{rasekh2023Dyoneda}. Consider the straightening-unstraightening Quillen equivalence from \cite{MRR2}. 
\begin{tz}
\node[](1) {$\MSrightfib{J}$}; 
\node[right of=1,xshift=3.8cm](2) {$[\Ch J^{\op},\MSThnsset]_\proj$}; 

\draw[->] ($(2.west)-(0,5pt)$) to node[below,la]{$\Un_J$} ($(1.east)-(0,5pt)$);
\draw[->] ($(1.east)+(0,5pt)$) to node[above,la]{$\St_J$} ($(2.west)+(0,5pt)$);

\node[la] at ($(1.east)!0.5!(2.west)$) {$\bot$};
\end{tz}

\begin{itemize}[leftmargin=*]
\item[\guillemotright] \emph{The weight $W\colon\cJ\to \cat_{(\infty,n-1)}$}: In the internal approach, the datum of the weight can be realized as a morphism $p\colon A\to J$ in $\sThnsset$. Because of the Quillen equivalence $\St_J\dashv\Un_J$, in the enriched approach the weight can be realized without loss of generality as a morphism $\St_J(p)\colon \Ch J\to \Thnsset$ in $\Thncat$.
\begin{center}
\fbox{\quad$\St_J(p)\colon \Ch J\to \Thnsset\quad \leftrightsquigarrow \quad p\colon A\to J$\quad}
\end{center}
\item[\guillemotright] \emph{The limit cone $\lambda$}: In the enriched approach, the limit cone can be realized as an enriched natural transformation $\lambda\colon \St_J(p)\Rightarrow \cC(\ell,F-)$, while, in the internal approach, it is realized by an element $\lambda^\flat$ of the fiber of $\slice{\cC}{F\circ p}\to \cC$ at $\ell$. In \cref{lem:webtwfibers}, we make such correspondence precise. 
\begin{center}
\fbox{\quad$ \lambda\colon \St_J(p)\Rightarrow \cC(\ell,F-)\quad \leftrightsquigarrow \quad \lambda^\flat\in (\slice{\cC}{F\circ p})_\ell $\quad}
\end{center}
\item[\guillemotright] \emph{The limit}: In the enriched context, as defined in \cite[\textsection5]{ShulmanHomotopyLimits}, the property for an object $\ell\in \cC$ to be a $W$-weighted limit of the diagram $F\colon\Ch J\to\cC$ is that the canonical map realizes a weak equivalence in $[\cC^{\op},\MSThnsset]_\mathrm{proj}$ 
 \[ \lambda^*\colon \Hom_\cC(-,\ell)\to \Hom_{[\Ch J,\Thnsset]}(\St_J(p),\Hom_\cC(-,F)). \] Alternatively, as defined in \cref{def:weightedlimit}, in the internal world the property for an object $\ell\in \cC$ to be the $p$-weighted limit of the diagram $F\colon J\to\Nh\cC$ is that there is an element $\lambda$ in the fiber of $\slice{\Nh\cC}{F^\flat\circ p}\to \Nh\cC$ at $\ell$ such that the canonical projection realizes a weak equivalence in $\MSrightfib{\Nh\cC}$
\[
\slice{(\slice{\Nh\cC}{F^\flat\circ p})}{(\ell,\lambda)}\xrightarrow{\simeq}\slice{\Nh\cC}{F^\flat\circ p}.
\]
 We prove in this paper as \cref{EquivalenceOfLimitsWeighted} that $\ell$ is a $\St_J(p)$-weighted limit of $F\colon\Ch J\to\cC$ if and only if $\ell$ is a $p$-weighted limit of $F^{\flat}\colon J\to\Nh\cC$.
\begin{center}
\fbox{ $\lim{}^{\St_J(p)} F \quad \leftrightsquigarrow \quad \lim{}^p F^\flat$}
\end{center}
By taking $p=\id$ we deduce as \cref{EquivalenceOfLimitsConical} an analogous correspondence for conical limits.
\begin{center}
\quad\fbox{ $\lim{}^{\St_J(\id)} F \quad \leftrightsquigarrow \quad \lim{} F^\flat$}
\end{center}
\end{itemize}

\subsection*{$(\infty,n)$-(Co)completeness of key examples}

Recall that, in the case $n=1$, the class of $(\infty,1)$-categories obtained as localizations of $(\infty,1)$-categories of $(\infty,1)$-presheaves are precisely the \emph{presentable $(\infty,1)$-categories}; see \cite[Theorem 5.5.1.1]{LurieHTT}. As a first application of the main result, we prove an $(\infty,n)$-categorical generalization of the well-known result that presentable $(\infty,1)$-categories are $(\infty,1)$-complete and $(\infty,1)$-cocomplete, as proven in  \cite[Corollary 5.5.2.4]{LurieHTT}. 

Explicitly, we establish that some important classes of  $(\infty,n)$-categories---here modeled by Segal category objects in $(\infty,n-1)$-categories \cite{br1}---have all weighted $(\infty,n)$-(co)limits:
\begin{itemize}[leftmargin=0.6cm]
    \item The $(\infty,n)$-category of $(\infty,n-1)$-categories.
    \item The $(\infty,n)$-category of $(\infty,n)$-presheaves, i.e., $(\infty,n-1)$-category valued presheaves.
    \item Localizations of the above examples.
    \item The $(\infty,n)$-category of $(\infty,n)$-sheaves, i.e., $(\infty,n-1)$-category valued sheaves, on an $(\infty,1)$-site.
\end{itemize}
All of these important examples follow from a more general result, establishing the $(\infty,n)$-(co)comp\-lete\-ness of any $(\infty,n)$-category which arises as the underlying $(\infty,n)$-category of an appropriately enriched model category; see \cref{subsec:cocompletenessmodel}.

\subsection*{Slices and cone constructions}

While categories of cones are classical concepts in the categorical literature, they particularly rose to prominence in the context of higher category theory due to their ability to efficiently package higher coherences. Hence, early adopters of higher categorical methods, starting with Joyal \cite{JoyalNotes} and then later Lurie \cite{LurieHTT}, Cisinski \cite{CisinskiBook}, Rasekh \cite{rasekh2023yoneda}, Riehl--Verity \cite{RVelements} employed categories of cones and slices and various abstractions thereof, so-called fibrations, as proper higher categorical analogues to (representable) presheaves; a statement that has been made precise via various Grothendieck constructions, straightening-unstraightening, and a fibrational Yoneda lemma. The centrality of this concept has motivated the development of alternative characterization of higher categorical cones, suitable to particular tasks.

Hence two prominent and equivalent $(\infty,1)$-categorical cone constructions have been established, relying on two join constructions:  the fat join $A\star B$ and the neat join $A\diamond B$. More precisely, given an $(\infty,1)$-category $C$ and an object $c$, we can define two slices of $(\infty,1)$-category: the fat version $C_{/c}$, based on the fat join, and neat slices $C_{/^\star c}$, based on the neat join, and we can analogously for a given functor $f\colon A\to C$ define two cone constructions: $C_{/f}$ and $C_{/^\star f}$. Hence, as part of their work various authors have shown that these slice and cone constructions coincide, notably Joyal \cite{JoyalNotes}, Lurie \cite[\textsection4.2.1]{LurieHTT}, Riehl-Verity \cite[\textsection D.6]{RVelements}, Rovelli \cite[\textsection2.4]{RovelliLimits}. Something similar is done for $n=2$ in \cite[\textsection4.2]{GHL2} for $(\infty,2)$-categories.

In this paper we enhance the whole story to the context of double $(\infty,n-1)$-categories (recovering the classical case when $n=1$). Precisely, given an $(\infty,n)$-category or more generally a double $(\infty,n-1)$-category $C$ and an object $c$, we define/recall two slice double $(\infty,n-1)$-categories: the fat version $\slice{C}{c}$ and the neat version $\joinslice{C}{c}$. More generally for a functor $f\colon A\to C$, we define the fat cone construction $\slice{C}{f}$ and the neat cone construction $\joinslice{C}{f}$. They are based on two versions of appropriate join constructions: the fat join $A\star B$ and the neat join $A\diamond B$. We recall the fat constructions in \cref{sec:defnfatcone}, we introduce the neat constructions in \cref{sec:defnneatcone}, and we show in \cref{SectionFatVsNeat} that the neat and fat versions of these constructions are equivalent.

The ability to alternate between two equivalent constructions of slices and cones enable explicit computations that make comparison between the enriched and internal definitions of limits possible. Indeed, while it was known that for a given strictly enriched category $\cC$ and object $c$ the two fibrations $\slice{\Nh\cC}{c}$ and $\Un_\cC\Hom_\cC(-,c)$ are equivalent (due to work in \cite{rasekh2023Dyoneda,MRR2}), we show in \cref{Unvsjoinslice} that $\Un_\cC\Hom_\cC(-,c)$ is in fact isomorphic to the neat slice $\joinslice{\Nh\cC}{c}$; an explicit point-set computation that in turn is a major ingredient of our comparison theorem.

\subsection*{Notations and prerequisites}

Recall Joyal's cell category $\Thn$ from \cite{JoyalDisks}. Throughout the paper we will use the following notational conventions. We write:
\begin{itemize}[leftmargin=0.6cm]
\item $\repD\in \set^{\Dop}$ for the representable at $m\geq 0$, 
    \item $\repDThnS\in \sThnsset$ for the representable at $([m],\defThn,[\defS])\in \DThnS$, 
    \item $F[m,X]\coloneqq F[m]\times X\in \sThnsset$ for the product of the representable at $m\geq 0$ and an object $X\in \Thnsset$. 
\end{itemize}
The categories $\Dset$ and $\Thnsset$ are naturally included into $\sThnsset$, and we regard objects of these subcategories as objects of $\sThnsset$ without further specification. We refer to an object of $\Thnsset$ as a \emph{$\Thn$-space}.

We further assume the reader to be familiar with standard model categorical techniques; see e.g.~\cite{Hovey,Hirschhorn}. In particular, we will make use without mention of the following facts: 
\begin{itemize}[leftmargin=0.6cm]
    \item a model structure is determined by its class of cofibrations (equivalently its class of trivial fibrations) and its class of fibrant objects; see \cite[Proposition E.1.10]{JoyalVolumeII},
    \item in the left Bousfield localization---referred to as \emph{localization} here---of a model structure, the class of cofibrations and the class of weak equivalences and fibrations between fibrant objects stay unchanged.
\end{itemize}

\subsection*{Acknowledgments}
During the realization of this work, the first author was a member of the Collaborative Research Centre ``SFB 1085: Higher
Invariants'' funded by the Deutsche Forschungsgemeinschaft (DFG).
The third author is grateful for support from the National Science Foundation under Grant No.~DMS-2203915.

\part{Review: \texorpdfstring{$(\infty,n)$}{(infinity,n)}-categories and \texorpdfstring{$(\infty,n)$}{(infinity,n)}-presheaves} \label{part1}

\section{Homotopy coherent categorification--nerve for \texorpdfstring{$(\infty,n)$}{(infinity,n)}-categories}

\label{SectionNerve}

In this section, we recall the models of $(\infty,n)$-categories given by the categories strictly enriched over $(\infty,n-1)$-categories in \cref{enrichedmodel} and the internal category objects (or Segal objects) in $(\infty,n-1)$-categories satisfying either a completeness or a discreteness condition in \cref{internalmodel}. We then recall in \cref{Nerve} the homotopy coherent categorification--nerve from \cite{MRR1}, which gives an equivalence between the enriched and internal models of $(\infty,n)$-categories.

\subsection{Enriched model of \texorpdfstring{$(\infty,n)$}{(infinity,n)}-categories} \label{enrichedmodel}

Recall from \cite{RezkTheta} that there is a model structure $\MSThnsset$ on the category $\Thnsset$ for $(\infty,n-1)$-categories given by the complete Segal $\Thn$-spaces. If $\MSspace$ denotes the Kan--Quillen model structure on $\sset$ for $(\infty,0)$-categories given by the Kan complexes, the model structure $\MSThnsset$ is obtained as a localization of the injective model structure $\injThnspace$ on the category of $\Thn$-presheaves valued in $\MSspace$; see \cite[\textsection 1.1]{MRR1} for more details. Moreover, by \cite[Theorem 8.1]{RezkTheta}, the model structure $\MSThnsset$ is cartesian closed. This is enough to guarantee that the model structure $\MSThnsset$ is excellent in the sense of \cite[Definition A.3.2.16]{LurieHTT}.

As a consequence, the category $\Thncat$ of $\Thnsset$-enriched categories and $\Thnsset$-enriched functors supports the model structure $\MSThncat$ from \cite[\textsection3.10]{br1}, obtained as a special instance of \cite[Proposition A.3.2.4 and Theorem A.3.2.24]{LurieHTT}. In this model structure, 
\begin{itemize}[leftmargin=0.6cm]
    \item a $\Thnsset$-enriched category $\cC$ is fibrant if, for all objects $a,b\in \cC$, the hom $\Thn$-space $\Hom_\cC(a,b)$ is fibrant in $\MSThnsset$,
    \item a $\Thnsset$-enriched functor $F\colon \cC\to \cC$ is a trivial fibration if it is surjective on objects, and for all objects $a,b\in \cC$ the induced map
    \[F_{a,b}\colon\Hom_{\cC}(a,b)\to\Hom_\cC(Fa,Fb)\]
   a trivial fibration in $\MSThnsset$.
\end{itemize}

\subsection{Internal models of \texorpdfstring{$(\infty,n)$}{(infinity,n)}-categories} \label{internalmodel}

Throughout the paper, we will consider the following models of double $(\infty,n-1)$-categories and $(\infty,n)$-categories. In what follows, we denote by $\injsThnsset$ the injective model structure on the category of simplicial objects valued in $\MSThnsset$. 

Recall from \cite{BR2} that there is a model structure $\SegsThnsset$ on the category $\sThnsset$ for double $(\infty,n-1)$-categories, in which 
\begin{itemize}[leftmargin=0.6cm]
    \item an object $D\in \sThnsset$ is fibrant if it is fibrant in $\injsThnsset$ and, for every $m\geq 2$, the Segal map 
    \[ D_m\to D_1\times_{D_0}\ldots \times_{D_0} D_1\]
    is a weak equivalence in $\MSThnsset$,
    \item a map $A\to B$ is a cofibration if it is a monomorphism in $\sThnsset$.
\end{itemize}
This model structure is obtained as a localization of the injective model structure $\injsThnsset$. Moreover, by \cite[Theorem 5.2]{BR2}, the model structure $\SegsThnsset$ is cartesian closed. 

Recall also from \cite{BR2} that there is a model structure $\CSSsThnsset$ on the category $\sThnsset$ for $(\infty,n)$-categories, which is obtained as a further localization of the above model structure. In the model structure $\CSSsThnsset$, 
\begin{itemize}[leftmargin=0.6cm]
    \item an object $C\in \sThnsset$ is fibrant if it is fibrant in $\SegsThnsset$, for all $\theta\in \Thn$, the map 
    \[ C_{0,[0]}\to C_{0,\theta} \]
    is a weak equivalence in $\MSspace$, and the map 
    \[ \Hom_{\sThnsset}(N\bI, C)\to C_0 \]
    is a weak equivalence in $\Thnsset$, where $N\bI\in \Dset$ denotes the nerve of the ``free-living isomorphism'',
    \item a map $A\to B$ is a cofibration if it is a monomorphism in $\sThnsset$.
\end{itemize}
This model structure is obtained as a localization of the injective model structure $\injsThnsset$. However, the model structure $\CSSsThnsset$ is \emph{not} cartesian closed.

We also consider the following model of $(\infty,n)$-categories introduced in \cite{br1}. We denote by $\pcatThn$ the full subcategory of $\sThnsset$ spanned by those objects $A\in \sThnsset$ such that $A_0$ is discrete, i.e., in the image of $\set\hookrightarrow \Thnsset$. As mentioned in \cite[\textsection7]{br1}, the inclusion $\pcatThn\hookrightarrow \sThnsset$ admits both adjoints.

There is a model structure $\pcatinj$ on the category $\pcatThn$ from \cite{br1}, in which
\begin{itemize}[leftmargin=0.6cm]
    \item an object $C\in \pcatThn$ is fibrant if it is fibrant in $\SegsThnsset$, 
    \item a map $A\to B$ is a cofibration if it is a monomorphism in $\pcatThn$.
\end{itemize}

Moreover, by \cite[Theorem 9.6]{BR2}, the model structure $\pcatinj$ is Quillen equivalent to the model structure $\CSSsThnsset$ considered above.

\subsection{Comparison of enriched and internal models of \texorpdfstring{$(\infty,n)$}{(infinity,n)}-categories} \label{Nerve}

In \cite{MRR1}, we construct a Quillen equivalence between the models $\pcatinj$ and $\MSThncat$ of $(\infty,n)$-categories, which is an $(\infty,n)$-categorical version of the homotopy coherent categorification--nerve adjunction. We recall briefly this construction; more details can be found in \cite[\textsection2.3]{MRR1}. 

The homotopy coherent categorification--nerve adjunction $c^h\dashv N^h$ of Cordier--Porter \cite{CordierPorterHomotopyCoherent} induces by post-composition an adjunction as below left, which restricts to an adjunction between full subcategories as below right.
\begin{tz}
\node[](1) {$(\Dset)^{\DThnop}$}; 
\node[right of=1,xshift=3.3cm](2) {$(\sCat)^{\DThnop}$}; 

\draw[->] ($(2.west)-(0,5pt)$) to node[below,la]{$\NL_*$} ($(1.east)-(0,5pt)$);
\draw[->] ($(1.east)+(0,5pt)$) to node[above,la]{$\CL_*$} ($(2.west)+(0,5pt)$);

\node[la] at ($(1.east)!0.5!(2.west)$) {$\bot$};

\node[right of=2,xshift=2.5cm](1) {$\pcatThn$}; 
\node[right of=1,xshift=2.7cm](2) {$\sThncat$}; 

\draw[->] ($(2.west)-(0,5pt)$) to node[below,la]{$\NL_*$} ($(1.east)-(0,5pt)$);
\draw[->] ($(1.east)+(0,5pt)$) to node[above,la]{$\CL_*$} ($(2.west)+(0,5pt)$);

\node[la] at ($(1.east)!0.5!(2.west)$) {$\bot$};
\end{tz}
Then the diagonal functor $\delta\colon \Delta\to \Delta\times \Delta$ induces an adjunction as below left, which induces by base-change an adjunction between categories of enriched categories as below right.
\begin{tz}
\node[](1) {$\SSset^{\Thnop}$}; 
\node[right of=1,xshift=1.8cm](2) {$\sset^{\Thnop}$}; 

\draw[->] ($(2.west)-(0,5pt)$) to node[below,la]{$(\delta_*)_*$} ($(1.east)-(0,5pt)$);
\draw[->] ($(1.east)+(0,5pt)$) to node[above,la]{$\diag$} ($(2.west)+(0,5pt)$);

\node[la] at ($(1.east)!0.5!(2.west)$) {$\bot$};

\node[right of=2,xshift=3cm](1) {$\sThncat$}; 
\node[right of=1,xshift=2.5cm](2) {$\Thncat$}; 

\draw[->] ($(2.west)-(0,5pt)$) to node[below,la]{$((\delta_*)_*)_*$} ($(1.east)-(0,5pt)$);
\draw[->] ($(1.east)+(0,5pt)$) to node[above,la]{$\diag_*$} ($(2.west)+(0,5pt)$);

\node[la] at ($(1.east)!0.5!(2.west)$) {$\bot$};
\end{tz}

The homotopy coherent categorification-nerve adjunction is then defined to be the following composite of adjunctions.
\begin{tz}
\node[](1) {$\pcatThn$}; 
\node[right of=1,xshift=2.7cm](2) {$\sThncat$}; 
\node[right of=2,xshift=2.5cm](3) {$\Thncat$}; 

\node at ($(1.west)-(6pt,.3pt)$) {$\Ch\colon$};
\node at ($(3.east)+(7pt,-1.5pt)$) {$\colon\! \Nh$};

\draw[->] ($(3.west)-(0,5pt)$) to node[below,la]{$((\delta_*)_*)_*$} ($(2.east)-(0,5pt)$);
\draw[->] ($(2.east)+(0,5pt)$) to node[above,la]{$\diag_*$} ($(3.west)+(0,5pt)$);

\node[la] at ($(1.east)!0.5!(2.west)$) {$\bot$};

\draw[->] ($(2.west)-(0,5pt)$) to node[below,la]{$\NL_*$} ($(1.east)-(0,5pt)$);
\draw[->] ($(1.east)+(0,5pt)$) to node[above,la]{$\CL_*$} ($(2.west)+(0,5pt)$);

\node[la] at ($(2.east)!0.5!(3.west)$) {$\bot$};
\end{tz} 

The following appears as \cite[Theorem~4.3.3]{MRR1}. 

\begin{thm} \label{nerve-cat}
    The adjunction
\begin{tz}
\node[](1) {$\pcatinj$}; 
\node[right of=1,xshift=2.9cm](2) {$\MSThncat$}; 

\draw[->] ($(2.west)-(0,5pt)$) to node[below,la]{$\Nh$} ($(1.east)-(0,5pt)$);
\draw[->] ($(1.east)+(0,5pt)$) to node[above,la]{$\Ch$} ($(2.west)+(0,5pt)$);

\node[la] at ($(1.east)!0.5!(2.west)$) {$\bot$};
\end{tz}
is a Quillen equivalence. 
\end{thm}

\section{Straightening--unstraightening for \texorpdfstring{$(\infty,n)$}{(infinity,n)}-presheaves}

\label{SectionStraightening}

In this section, we recall the models of $(\infty,n)$-presheaves, i.e., $(\infty,n)$-functors valued in the $(\infty,n)$-category of $(\infty,n-1)$-categories, given by the strictly enriched functors valued in the model category $\MSThnsset$ in \cref{enrichedfunctor} and the double $(\infty,n-1)$-right fibrations introduced in \cite{rasekh2023Dyoneda} in \cref{internalfunctor}. We then recall in \cref{straightening} the straightening--unstraightening from \cite{MRR3}, which gives an equivalence between the enriched and internal models of $(\infty,n)$-presheaves.

\subsection{Enriched model of \texorpdfstring{$(\infty,n)$}{(infinity,n)}-presheaves} \label{enrichedfunctor}

The category $\Thnsset$ is cartesian closed and so it can be seen as a $\Thnsset$-enriched category. Given a $\Thnsset$-enriched category $\cJ$, we denote by $[\cJ^{\op},\Thnsset]$ the category of $\Thnsset$-enriched functors from $\cJ^{\op}$ to $\Thnsset$ and $\Thnsset$-enriched natural transformations between them. By \cite[Theorem 4.4 (ii)]{Moserinj}, the category $[\cJ^{\op},\Thnsset]$ of $\Thnsset$-enriched functors $\cJ^{\op}\to \Thnsset$ supports the projective model structure $[\cJ^{\op},\MSThnsset]_\proj$, in which
\begin{itemize}[leftmargin=0.6cm]
    \item a $\Thnsset$-enriched functor $F\colon \cJ^{\op}\to \Thnsset$ is fibrant if, for every object $j\in \cJ$, the $\Thn$-space $F(j)$ is fibrant in $\MSThnsset$, 
    \item a map $F\to G$ in $[\cJ^{\op},\Thnsset]$ is a trivial fibration if, for every object $j\in \cJ$, the induced map $F(j)\to G(j)$ is a trivial fibration in $\MSThnsset$. 
\end{itemize}
Moreover, by \cite[ Theorem 5.4]{Moserinj}, the model structure $[\cJ^{\op},\MSThnsset]_\mathrm{proj}$ is enriched over $\MSThnsset$.

Recall that every $\Thnsset$-enriched functor $F$ induces an adjunction $F_!\dashv F^*$ by enriched left Kan extension and pre-composition (see \cite[Theorem 4.51]{KellyEnriched}), which has the following homotopical property, as we now recall from \cite[Propositions~A.3.3.7(1) and~A.3.3.8(1)]{LurieHTT}. 

\begin{prop} \label{LKE-precomp}
    Let $F\colon \cI\to \cJ$ be a $\Thnsset$-enriched functor. The adjunction
    \begin{tz}
\node[](1) {$[\cI^{\op},\MSThnsset]_\mathrm{proj}$}; 
\node[right of=1,xshift=3.6cm](2) {$[\cJ^{\op},\MSThnsset]_\mathrm{proj}$}; 

\draw[->] ($(2.west)-(0,5pt)$) to node[below,la]{$F^*$} ($(1.east)-(0,5pt)$);
\draw[->] ($(1.east)+(0,5pt)$) to node[above,la]{$F_!$} ($(2.west)+(0,5pt)$);

\node[la] at ($(1.east)!0.5!(2.west)$) {$\bot$};
\end{tz}
is a Quillen pair  enriched over $\MSThnsset$, which is further a Quillen equivalence when $F$ is a weak equivalence in $\MSThncat$.
\end{prop}

\subsection{Internal model of \texorpdfstring{$(\infty,n)$}{(infinity,n)}-presheaves} \label{internalfunctor}

Given an object $J\in \sThnsset$, we denote by $\sThnssetslice{J}$ the slice category over $J$. By \cite[Theorem 5.15]{rasekh2023Dyoneda}, there is a model structure $\MSrightfib{J}$ on $\sThnssetslice{J}$ for double $(\infty,n-1)$-right fibrations. In this model structure, 
\begin{itemize}[leftmargin=0.6cm]
    \item an object $A\to J$ is fibrant if it is a fibration in $\injsThnsset$ and, for every $m\geq 1$, the map 
    \[ A_m\to A_0\times_{J_0} J_m \]
    induced by the morphism $\langle m\rangle \colon [0]\to [m]$ of $\Delta$ is a weak equivalence in $\MSThnsset$, 
    \item a map $A\to B$ over $J$ is cofibration if it is a monomorphism in $\sThnsset$. 
\end{itemize}
This model structure is obtained as a localization of the slice model structure ${\injsThnsset}_{/J}$. Moreover, by \cite[ Theorem 5.15]{rasekh2023Dyoneda}, the model structure $\MSrightfib{J}$ is enriched over $\MSThnsset$.

We make the following observations about double $(\infty,n-1)$-right fibrations.

\begin{rmk}
    Let $p\colon A\to J$ in $\sThnsset$ be a double $(\infty,n-1)$-right fibration, i.e., a fibrant object in $\MSrightfib{J}$, and let $x\in J_{0,[0],0}$ be an element. The \textbf{fiber of $p$ at $x$} is the (homotopy) pullback in $\MSThnsset$ 
    \begin{tz}
        \node[](1) {$A_x$}; 
        \node[right of=1,xshift=.5cm](2) {$A_0$}; 
        \node[below of=1](3) {$\Delta[0]$}; 
        \node[below of=2](4) {$J_0$}; 

        \draw[->] (1) to  (2); 
        \draw[->>] (1) to (3); 
        \draw[->>] (2) to node[right,la]{$p_0$} (4); 
        \draw[->] (3) to node[below,la]{$x$} (4);
        \pullback{1};
    \end{tz}
    In particular, the $\Thn$-space $A_x$ is a fibrant object in $\MSThnsset$. 
\end{rmk}

\begin{prop} \label{sourceofdblrightfib}
    Let $D$ be a fibrant object in $\SegsThnsset$ and $p\colon A\to D$ in $\sThnsset$ be a double $(\infty,n-1)$-right fibration. Then $A$ is also fibrant in $\SegsThnsset$.
\end{prop}

The following result appears as \cite[Theorem 4.62]{rasekh2023Dyoneda}.

\begin{prop} \label{webtwdoubleright}
    Let $J$ be an object of $\pcatThn$, seen as an object of $\sThnsset$, and let $A\to J$ and $B\to J$ be double $(\infty,n-1)$-right fibrations. Then the following are equivalent for a map $f\colon A\to B$ in $\sThnssetslice{J}$:
    \begin{rome}[leftmargin=1cm]
        \item the map $f\colon A\to B$ is a weak equivalence in $\MSrightfib{J}$, 
        \item for all $x\in J_{0,[0],0}$, the map on fibers $f_x\colon A_x\to B_x$ is a weak equivalence in~$\MSThnsset$.
    \end{rome}
\end{prop}

Recall that every map $f$ in $\sThnsset$ induces an adjunction $f_!\dashv f^*$ by post-composition and pullback, which has the following homotopical property, as we now recall from \cite[Theorem 5.38]{rasekh2023Dyoneda}.

\begin{prop} \label{postcomp-pullback}
    Let $f\colon I\to J$ be a map in $\sThnsset$. The adjunction
    \begin{tz}
\node[](1) {$\MSrightfib{I}$}; 
\node[right of=1,xshift=3.9cm](2) {$\MSrightfib{J}$}; 

\draw[->] ($(2.west)-(0,5pt)$) to node[below,la]{$f^*$} ($(1.east)-(0,5pt)$);
\draw[->] ($(1.east)+(0,5pt)$) to node[above,la]{$f_!$} ($(2.west)+(0,5pt)$);

\node[la] at ($(1.east)!0.5!(2.west)$) {$\bot$};
\end{tz}
is a Quillen pair enriched over $\MSThnsset$, which is further a Quillen equivalence when $f$ is such that, for every object $\theta\in \Thn$, the induced map $f_{-,\theta}\colon I_{-,\theta}\to J_{-,\theta}$ is a weak equivalence in $\MSDsset$. 
\end{prop}

\begin{rmk}
    There is also a model structure $\MSleftfib{J}$ for double $(\infty,n-1)$-left fibrations, where the left fibration condition is defined using the morphism $\langle 0\rangle \colon [0]\to [m]$ of $\Delta$ in place of $\langle m\rangle$. It satisfies analogous properties as the model structure for double $(\infty,n-1)$-right fibrations.
\end{rmk}

\subsection{Comparison of enriched and internal model of \texorpdfstring{$(\infty,n)$}{(infinity,n)}-presheaves} \label{straightening}

Let us fix an object $J\in \pcatThn$, seen as an object of $\sThnsset$. In \cite{MRR2}, we construct a Quillen equivalence between the models $[\Ch J^{\op},\MSThnsset]_\mathrm{proj}$ and $\MSrightfib{J}$ of $(\infty,n)$-presheaves. This Quillen equivalence is an $(\infty,n)$-categorical version of straightening--unstraightening. We recall briefly this construction; more details can be found in \cite[\textsection3.1]{MRR2}.

Let $m,k\geq 0$, $\theta\in \Thn$, and $\sigma\colon \repD[m,\defThn,\defS]\to J$ be a map in~$\sThnsset$. Since $J$ is in $\pcatThn$, the map $\sigma$ corresponds to a map $\sigma\colon L\repD[m,\defThn,\defS]\to J$ in $\pcatThn$, where $L\colon \sThnsset\to \pcatThn$ denotes the left adjoint of the inclusion. We write~$J_\sigma$ for the following pushout in $\pcatThn$ (and hence in $\sThnsset$).
    \begin{tz}
        \node[](1) {$L\repD[m,\defThn,\defS]$}; 
        \node[right of=1,xshift=1.6cm](2) {$J$}; 
        \node[below of=1](3) {$L\repD[m+1,\defThn,\defS]$}; 
        \node[below of=2](4) {$J_\sigma$}; 

        \draw[->] (1) to node[above,la]{$\sigma$} (2); 
        \draw[->] (1) to node[left,la]{$L [d^{m+1},\defThn,\defS]$} (3); 
        \draw[->] (2) to node[right,la]{$\iota_\sigma$} (4); 
        \draw[->] (3) to node[below,la]{$\sigma'$} (4);
        \pushout{4};
    \end{tz}
As pushouts in $\sThnsset$ can be computed levelwise in $\Thnsset$, we have $(J_\sigma)_0\cong J_0\amalg \{\top\}$, where $\top$ is the image under $\sigma'$ of the object $m+1\in L\repD[m+1,\defThn,\defS]_0\cong \{0,1,\ldots,m+1\}$.

Then we define the $\Thnsset$-enriched functor  $\St_J(\sigma)\colon \Ch J^{\op}\to \Thnsset$ to be the following composite 
    \[ \St_J(\sigma)\colon \Ch J^{\op}\xrightarrow{\Ch (\iota_\sigma)} (\Ch J_\sigma)^{\op}\xrightarrow{\Hom_{\Ch J_\sigma}(-,\top)} \Thnsset. \]
    This construction extends to a functor
    \[ \textstyle\St_J\colon \int_{\DThnS} J\to [\Ch J^{\op}, \Thnsset],\] and the straightening--unstraightening adjunction is obtained by left Kan extending along the Yoneda embedding
\begin{tz}
 \node[](1) {$\sThnssetslice{J}$}; 
\node[right of=1,xshift=2.7cm](2) {$[\Ch J^{\op},\Thnsset]$}; 
\punctuation{2}{.};

\draw[->] ($(2.west)-(0,5pt)$) to node[below,la]{$\Un_J$} ($(1.east)-(0,5pt)$);
\draw[->] ($(1.east)+(0,5pt)$) to node[above,la]{$\St_J$} ($(2.west)+(0,5pt)$);

\node[la] at ($(1.east)!0.5!(2.west)$) {$\bot$};    
\end{tz}

The following appears as \cite[Theorem 4.4.10]{MRR2}.

\begin{thm} \label{st-unst}
    Let $J$ be an object in $\pcatThn$. The adjunction
\begin{tz}
\node[](1) {$\MSrightfib{J}$}; 
\node[right of=1,xshift=3.8cm](2) {$[\Ch J^{\op},\MSThnsset]_\proj$}; 

\draw[->] ($(2.west)-(0,5pt)$) to node[below,la]{$\Un_J$} ($(1.east)-(0,5pt)$);
\draw[->] ($(1.east)+(0,5pt)$) to node[above,la]{$\St_J$} ($(2.west)+(0,5pt)$);

\node[la] at ($(1.east)!0.5!(2.west)$) {$\bot$};
\end{tz}
 is a Quillen equivalence  enriched over $\MSThnsset$.
\end{thm}

\begin{notation} \label{enrichedst-unst}
    Let $\cC$ be a fibrant $\MSThnsset$-enriched category. We write $\St_\cC\dashv \Un_\cC$ for the composite of $\MSThnsset$-enriched Quillen equivalences from \cref{LKE-precomp,st-unst}
    \begin{tz}
\node[](1) {$\MSrightfib{\Nh\cC}$}; 
\node[right of=1,xshift=3.9cm](2) {$[\Ch\Nh\cC^{\op},\MSThnsset]_\proj$}; 
\node[right of=2,xshift=3.8cm](3) {$[\cC^{\op},\MSThnsset]_\proj$}; 

\node at ($(1.west)-(10pt,2pt)$) {$\St_\cC\colon$};
\node at ($(3.east)+(13pt,-3pt)$) {$\colon\! \Un_\cC$};

\draw[->] ($(3.west)-(0,5pt)$) to node[below,la]{$(\varepsilon_\cC)^*$} ($(2.east)-(0,5pt)$);
\draw[->] ($(2.east)+(0,5pt)$) to node[above,la]{$(\varepsilon_\cC)_!$} ($(3.west)+(0,5pt)$);

\node[la] at ($(1.east)!0.5!(2.west)$) {$\bot$};

\draw[->] ($(2.west)-(0,5pt)$) to node[below,la]{$\Un_{\Nh\cC}$} ($(1.east)-(0,5pt)$);
\draw[->] ($(1.east)+(0,5pt)$) to node[above,la]{$\St_{\Nh\cC}$} ($(2.west)+(0,5pt)$);

\node[la] at ($(2.east)!0.5!(3.west)$) {$\bot$};
\end{tz} 
where $\varepsilon_\cC\colon \Ch\Nh\cC\to \cC$ is the component of the (derived) counit at $\cC$ of the Quillen equivalence $\Ch\dashv \Nh$ from \cref{nerve-cat}, which is therefore a weak equivalence in $\MSThncat$.
\end{notation}

The straightening functors satisfy the following naturality condition. 

\begin{lemma} \label{lem:naturalityofSt}
    Let $J$ be an object in $\pcatThn$ and $\cC$ be a fibrant $\Thnsset$-enriched category. Let $F\colon \Ch J\to \cC$ be a $\Thnsset$-enriched functor and $F^\flat\colon J\to \Nh\cC$ be the adjunct of $F$ in $\sThnsset$. The following square of functors commutes up to natural isomorphism.
    \begin{tz}
        \node[](1) {$\sThnssetslice{J}$}; 
        \node[right of=1,xshift=2.7cm](2) {$[\Ch J^{\op},\Thnsset]$}; 
        \node[below of=1](3) {$\sThnssetslice{\Nh\cC}$}; 
        \node[below of=2](4) {$[\cC^{\op},\Thnsset]$}; 

        \draw[->] (1) to node[above,la]{$\St_J$} (2); 
        \draw[->] (1) to node[left,la]{$(F^\flat)_!$} (3); 
        \draw[->] (3) to node[below,la]{$\St_\cC$} (4);
        \draw[->] (2) to node[right,la]{$F_!$} (4);
    \end{tz}
\end{lemma}

\begin{proof}
    This follows from \cite[Proposition 3.2.1]{MRR2} using that by definition $\St_\cC=(\varepsilon_\cC)_!\St_{\Nh\cC} $ and $F=\varepsilon_\cC(\Ch F^\flat)$.
\end{proof}

\begin{rmk}
    Similarly, we have Quillen equivalences $\St_J\dashv \Un_J$ and $\St_\cC\dashv \Un_\cC$ when removing ``op'' and replacing ``right'' with ``left'', and the naturality of straightening also holds in this case.
\end{rmk}

\part{Slice and cone constructions for double \texorpdfstring{$(\infty,n-1)$}{(infinity,n-1)}-categories} \label{part2}

\section{Fat slice and cone constructions}

\label{SectionFat}

In this section, we enhance to the context of double $(\infty,n-1)$-categories the fat slice and cone constructions introduced in \cite[\textsection9.17-9.19]{JoyalNotes} and further studied in  \cite[\textsection4.2.1]{LurieHTT}, \cite[\textsection2.4]{RVqCat} (with more details in \cite[Appendix~A]{RVqCatv3}) and \cite[Appendix~D.2]{RVelements}. For this, we first recall in \cref{sec:defnfatcone} the slice and cone constructions for double $(\infty,n-1)$-categories considered in \cite{MRR3,rasekh2023Dyoneda}. We show in \cref{sec:fibrantfatcone} that, over a double $(\infty,n-1)$-category, the canonical projections from the slice and cone constructions are double $(\infty,n-1)$-right fibrations.

\subsection{The fat constructions}
\label{sec:defnfatcone}

In order to give a first definition of slice and cone constructions, we start by defining a notion of fat join. 

\begin{constr} \label{constr:join}
    The \textbf{fat join functor}
    \[ -\diamond - \colon \sThnsset\times \sThnsset\to \sThnsset \]
    is induced by postcomposition along the classical fat join functor $-\diamond-\colon \Dset\times \Dset\to \Dset$. Since pushouts in $\sThnsset$ are defined levelwise in $\Dset$, this functor is given by sending a pair $(A,B)$ of objects of $\sThnsset$ to the following pushout in $\sThnsset$.
    \begin{tz}
        \node[](1) {$A\times B \amalg A\times B$}; 
        \node[below of=1](2) {$A\times F[1]\times B$}; 
        \node[right of=1,xshift=2cm](3) {$A\amalg B$}; 
        \node[below of=3](4) {$A\diamond B$}; 
        \pushout{4};

        \draw[right hook->](1) to (2); 
        \draw[right hook->](3) to (4); 
        \draw[->](1) to (3); 
        \draw[->](2) to (4); 
    \end{tz}
\end{constr}

\begin{rmk} 
 \label{fatjoinprescolim}
    Given an object $A\in \sThnsset$, then the induced functors \[ A\diamond -, -\diamond A\colon \sThnsset\to {}^{A/} \sThnsset \]
    preserve colimits. In particular, the functors \[ A\diamond -, -\diamond A\colon \sThnsset\to \sThnsset \] 
    preserve connected colimits.
\end{rmk} 

By the adjoint functor theorem, we then get the following result. 

\begin{prop}
There are adjunctions
    \begin{tz}
\node[](1) {$\sThnsset$}; 
\node[right of=1,xshift=2.7cm](2) {${}^{A/} \sThnsset$}; 

\draw[->] ($(2.west)-(0,5pt)$) to node[below,la]{$\slice{(-)}{(-)}$} ($(1.east)-(0,5pt)$);
\draw[->] ($(1.east)+(0,5pt)$) to node[above,la]{$-\diamond A$} ($(2.west)+(0,5pt)$);
\punctuation{2}{,};

\node[la] at ($(1.east)!0.5!(2.west)$) {$\bot$};

\node[right of=2,xshift=2.4cm](1) {$\sThnsset$}; 
\node[right of=1,xshift=2.7cm](2) {${}^{A/} \sThnsset$}; 
\punctuation{2}{.};

\draw[->] ($(2.west)-(0,5pt)$) to node[below,la]{$\sliceunder{(-)}{(-)}$} ($(1.east)-(0,5pt)$);
\draw[->] ($(1.east)+(0,5pt)$) to node[above,la]{$A\diamond -$} ($(2.west)+(0,5pt)$);

\node[la] at ($(1.east)!0.5!(2.west)$) {$\bot$};
\end{tz}
\end{prop}

We now study the right adjoint of the functor $-\diamond A$, but everything can be dualized to $A\diamond -$. 

\begin{defn}
Let $Z$ be an object in $\sThnsset$ and $x\in Z_{0,0,0}$. We call the object $\slice{Z}{x}$ of $\sThnsset$ the \textbf{(fat) slice construction}. By definition, for $m,k\geq 0$ and $\theta\in \Thn$, an element of the set $(\slice{Z}{x})_{m,\theta,k}$ is given by a commutative diagram in $\sThnsset$ of the form
    \begin{tz}
        \node[](1) {$\repDThnS[m]\diamond \repD[0]$}; 
        \node[right of=1,xshift=1.2cm](2) {$Z$}; 
        \node[above of=1](3) {$\repD[0]$}; 
        \draw[->](1) to (2); 
        \draw[right hook->](3) to (1);
        \draw[->](3) to  node[right,la,yshift=5pt]{$x$} (2); 
    \end{tz}
    The canonical inclusions $\repDThnS\hookrightarrow \repDThnS\diamond \repD[0]$ induce a projection map 
    \[ \slice{Z}{x}\to Z. \]
\end{defn}

\begin{defn}
Let $f\colon A\to Z$ be a map in $\sThnsset$. We call the object $\slice{Z}{f}$ of $\sThnsset$ the \textbf{(fat) cone construction}. By definition, for $m,k\geq 0$ and $\theta\in \Thn$, an element of the set $(\slice{Z}{f})_{m,\theta,k}$ is given by a commutative diagram in $\sThnsset$ of the form
    \begin{tz}
        \node[](1) {$\repDThnS[m]\diamond A$}; 
        \node[right of=1,xshift=1.2cm](2) {$Z$}; 
        \node[above of=1](3) {$A$}; 
        \draw[->](1) to (2); 
        \draw[right hook->](3) to (1);
        \draw[->](3) to  node[right,la,yshift=5pt]{$f$} (2); 
    \end{tz}
    The canonical inclusions $\repDThnS\hookrightarrow \repDThnS\diamond A$ induce a projection map 
    \[ \slice{Z}{f}\to Z. \]
\end{defn}

It is straightforward to see that the slice and cone constructions also admit the following descriptions, which were taken as definitions in \cite[Notations 1.4.21 and 1.4.22]{MRR3}.

\begin{prop} \label{sliceaspullback}
Let $Z$ be an object in $\sThnsset$, $x\in Z_{0,0,0}$, and $f\colon A\to Z$ be a map in $\sThnsset$. We have the following pullbacks in $\sThnsset$,
\begin{tz}
    \node[](1) {$\slice{Z}{x}$}; 
    \node[right of=1,xshift=1cm](2) {$Z^{\repD[1]}$}; 
    \node[below of=1](3) {$Z$}; 
    \node[below of=2](4) {$Z\times Z$}; 
    \pullback{1};

    \draw[->](1) to (2); 
    \draw[->](1) to (3); 
    \draw[->](3) to node[below,la]{$\id_Z\times \{x\}$} (4); 
    \draw[->](2) to node[right,la]{$(d^1,d^0)$} (4);

    \node[right of=2,xshift=1.3cm](1) {$\slice{Z}{f}$}; 
    \node[right of=1,xshift=1.5cm](2) {$(Z^A)^{F[1]}$}; 
    \node[below of=1](3) {$Z$}; 
    \node[below of=2](4) {$Z^A\times Z^A$}; 
    \pullback{1};

    \draw[->](1) to (2); 
    \draw[->](1) to (3); 
    \draw[->](3) to node[below,la]{$\Delta\times \{f\}$} (4); 
    \draw[->](2) to node[right,la]{$(d^1,d^0)$} (4);
\end{tz}
where $\Delta\colon Z\to Z^A$ is the map induced by precomposition along the unique map $A\to F[0]$. 
\end{prop}

\subsection{Homotopical properties of the fat constructions} \label{sec:fibrantfatcone}

We prove that, for a double $(\infty,n-1)$-category, the canonical projections from the slice and cone constructions are double $(\infty,n-1)$-right fibrations. For the slice, this is proven in \cite[Theorem 5.25]{rasekh2023Dyoneda}.

\begin{prop}
\label{prop:fibrantslice}
    Let $D$ be a fibrant object in $\SegsThnsset$ and $d\in D_{0,0,0}$. Then the canonical projection  \[ \slice{D}{d}\to D \]
    is a double $(\infty,n-1)$-right fibration.   
\end{prop}

As a consequence, we can show that the canonical projection from the cone construction is also a double $(\infty,n-1)$-right fibration.

\begin{prop} \label{prop:fibrantcone}
    Let $D$ be a fibrant object in $\SegsThnsset$ and $f\colon A\to D$ be a map in $\sThnsset$. Then the canonical projection
    \[ \slice{D}{f}\to D \]
    is a double $(\infty,n-1)$-right fibration. 
\end{prop}

\begin{proof}
    This follows from the fact that the projection $\slice{D}{f}\to D$ is the pullback along $\Delta\colon D\to D^A$ of the projection $\slice{D^A}{f}\to D$ from the slice over $f\in (D^A)_{0,[0],0}$ which is a double $(\infty,n-1)$-right fibration by cartesian closedness of $\SegsThnsset$ and \cref{prop:fibrantslice}.
\end{proof}

\begin{rmk}
    Note that, if $C$ is a fibrant object in $\CSSsThnsset$ and $f\colon A\to C$ is a map in $\sThnsset$, then $\slice{C}{f}$ is always a double $(\infty,n-1)$-category, but in general not an $(\infty,n)$-category, as $\CSSsThnsset$ is not cartesian closed.
\end{rmk}

We also observe the following result.

\begin{cor} \label{diamondjoinLQ}
The adjunction
    \begin{tz}
\node[](1) {$\SegsThnsset$}; 
\node[right of=1,xshift=2.7cm](2) {${}^{A/} \SegsThnsset$}; 

\draw[->] ($(2.west)-(0,5pt)$) to node[below,la]{$\slice{(-)}{(-)}$} ($(1.east)-(0,5pt)$);
\draw[->] ($(1.east)+(0,5pt)$) to node[above,la]{$-\diamond A$} ($(2.west)+(0,5pt)$);

\node[la] at ($(1.east)!0.5!(2.west)$) {$\bot$};
\end{tz}
is a Quillen pair.
\end{cor}

\begin{proof}
    This follows from the fact that $-\diamond A$ is defined as a (homotopy) pushout of left Quillen functors using cartesian closedness of $\SegsThnsset$ and so is also left Quillen. 
\end{proof}

\section{Neat slice and cone constructions}

\label{SectionNeat}

In this section, we enhance to the context of double $(\infty,n-1)$-categories the slice and cone constructions from \cite{EhlersPorter}
(see also \cite[\textsection9.1-9.16]{JoyalNotes}, \cite[\textsection1.2.9]{LurieHTT}, \cite[\textsection2.4]{RVqCat} and \cite[Appendix~D.2]{RVelements}). For this, we introduce in \cref{sec:defnneatcone} another slice and cone constructions for double $(\infty,n-1)$-categories---which we refer to as \emph{neat}---, based on a join functor for double $(\infty,n-1)$-categories given levelwise by the usual join of simplicial sets. We then show in \cref{sec:pushoutjoin} that the pushout-join functor has good homotopical properties, from which we deduce in  \cref{sec:fibrantneatcone} that, over a double $(\infty,n-1)$-category, the canonical projections from the neat slice and cone constructions are double $(\infty,n-1)$-right fibrations. We further give in \cref{sec:fibersneatcone} a formula that is used at a later stage for the fibers of the canonical projection from the neat cone construction.

\subsection{The neat constructions} \label{sec:defnneatcone}

We define another join functor on the category $\sThnsset$ and provide yet other slice and cone constructions using this join.

\begin{rmk} \label{joinsset}
    Recall the classical join functor $-\star-\colon \Dset\times \Dset\to \Dset$. We will use the following properties of the join. For all $m,m'\geq 0$, there is an isomorphism in $\Dset$
    \[
    F[m]\star F[m']\cong F[m+1+m'],
    \]
    for all objects $A,B\in \Dset$, there is a canonical inclusion $A\amalg B\hookrightarrow A\star B$ and, by \cite[Lemma~D.2.7]{RVelements}, for all objects $A\in \Dset$, there are adjunctions
    \begin{tz}
\node[](1) {$\Dset$}; 
\node[right of=1,xshift=1.5cm](2) {${}^{A/} \Dset$}; 

\draw[->] ($(2.west)-(0,5pt)$) to ($(1.east)-(0,5pt)$);
\draw[->] ($(1.east)+(0,5pt)$) to node[above,la]{$-\star A$} ($(2.west)+(0,5pt)$);
\punctuation{2}{,};

\node[la] at ($(1.east)!0.5!(2.west)$) {$\bot$};

\node[right of=2,xshift=1.2cm](1) {$\Dset$}; 
\node[right of=1,xshift=1.5cm](2) {${}^{A/} \Dset$}; 

\punctuation{2}{.};

\draw[->] ($(2.west)-(0,5pt)$) to ($(1.east)-(0,5pt)$);
\draw[->] ($(1.east)+(0,5pt)$) to node[above,la]{$A\star -$} ($(2.west)+(0,5pt)$);

\node[la] at ($(1.east)!0.5!(2.west)$) {$\bot$};
\end{tz}
\end{rmk}

\begin{constr} 
    The \textbf{(neat) join functor}
    \[ -\star - \colon \sThnsset\times \sThnsset\to \sThnsset \]
    is induced by postcomposition along the join functor $-\star-\colon \Dset\times \Dset\to \Dset$. Given objects $A,B\in \sThnsset$, there is a canonical inclusion 
    \[ A\amalg B\hookrightarrow A\star B. \]
\end{constr}

\begin{rmk} 
 \label{joinprescolim}
    Given an object $A\in \sThnsset$, then the induced functors \[ A\star -, -\star A\colon \sThnsset\to {}^{A/} \sThnsset \]
    preserve colimits as those are computed levelwise in $\Dset$. In particular, the functors \[ A\star -, -\star A\colon \sThnsset\to \sThnsset \] 
    preserve connected colimits.
\end{rmk} 

By the adjoint functor theorem, we then get the following result. 

\begin{prop}
There are adjunctions
    \begin{tz}
\node[](1) {$\sThnsset$}; 
\node[right of=1,xshift=2.7cm](2) {${}^{A/} \sThnsset$}; 

\draw[->] ($(2.west)-(0,5pt)$) to node[below,la]{$\joinslice{(-)}{(-)}$} ($(1.east)-(0,5pt)$);
\draw[->] ($(1.east)+(0,5pt)$) to node[above,la]{$-\star A$} ($(2.west)+(0,5pt)$);

\punctuation{2}{,};

\node[la] at ($(1.east)!0.5!(2.west)$) {$\bot$};

\node[right of=2,xshift=2.4cm](1) {$\sThnsset$}; 
\node[right of=1,xshift=2.7cm](2) {${}^{A/} \sThnsset$}; 

\punctuation{2}{.};

\draw[->] ($(2.west)-(0,5pt)$) to node[below,la]{$\joinsliceunder{(-)}{(-)}$} ($(1.east)-(0,5pt)$);
\draw[->] ($(1.east)+(0,5pt)$) to node[above,la]{$A\star -$} ($(2.west)+(0,5pt)$);

\node[la] at ($(1.east)!0.5!(2.west)$) {$\bot$};
\end{tz}
\end{prop}

We now study the right adjoint of the functor $-\star A$, but everything can be dualized to $A\star -$.

\begin{notation} \label{def:neatslice}
    Let $Z$ be an object in $\sThnsset$ and $x\in Z_{0,0,0}$. We call the object $\joinslice{Z}{x}$ of $\sThnsset$ the \textbf{neat slice construction}. By definition, for $m,k\geq 0$ and $\theta\in \Thn$, an element of the set $(\joinslice{Z}{x})_{m,\theta,k}$ is given by a commutative diagram in $\sThnsset$ of the form
    \begin{tz}
        \node[](1) {$\repDThnS[m]\star \repD[0]$}; 
        \node[right of=1,xshift=1.2cm](2) {$Z$}; 
        \node[above of=1](3) {$\repD[0]$}; 
        \draw[->](1) to (2); 
        \draw[right hook->](3) to (1);
        \draw[->](3) to  node[right,la,yshift=5pt]{$x$} (2); 
    \end{tz}
    The canonical inclusions $\repDThnS\hookrightarrow \repDThnS\star \repD[0]$ induce a projection map 
    \[ \joinslice{Z}{x}\to Z. \]
\end{notation} 

\begin{defn}
    Let $f\colon A\to Z$ be a map in $\sThnsset$. We call the object $\joinslice{Z}{f}$ the \textbf{neat cone construction}. By definition, for $m,k\geq 0$ and $\theta\in \Thn$, an element of the set $(\joinslice{Z}{f})_{m,\theta,k}$ is given by a commutative diagram in $\sThnsset$ of the form
    \begin{tz}
        \node[](1) {$\repDThnS[m]\star A$}; 
        \node[right of=1,xshift=1.2cm](2) {$Z$}; 
        \node[above of=1](3) {$A$}; 
        \draw[->](1) to (2); 
        \draw[right hook->](3) to (1);
        \draw[->](3) to  node[right,la,yshift=5pt]{$f$} (2); 
    \end{tz}
    The canonical inclusions $\repDThnS\hookrightarrow \repDThnS\star A$ induce a projection map 
    \[ \joinslice{Z}{f}\to Z. \]
\end{defn}

\subsection{Homotopical properties of pushout-joins} \label{sec:pushoutjoin}

We aim to show that, for any double $(\infty,n-1)$-category, the canonical projections from the neat slice and cone constructions are double $(\infty,n-1)$-right fibrations. For this, we first show the following result in this section, generalizing \cite[\textsection 9.10]{JoyalNotes} and \cite[Lemma 2.1.2.3]{LurieHTT}.

\begin{thm} \label{thm:pushprodofjoins}
    Let $i\colon A\hookrightarrow B$ and $i'\colon A'\hookrightarrow B'$ be monomorphisms in $\sThnsset$. Then the pushout-join map 
    \[ i\pushprodstar i'=(A\star B'\amalg_{A\star A'} B\star A'\to B\star B')\]
    is a monomorphism in $\sThnsset$. Moreover, if $i$ or $i'$ is a trivial cofibration in $\injsThnsset$, then so is $i\pushprodstar i'$.
\end{thm}

\begin{proof}
    As joins, pushouts, and monomorphisms in $\sThnsset$ can be computed levelwise in $\Dset$, the fact that $i\pushprodstar i'$ is a monomorphism follows directly from the similar statement for~$\Dset$. It remains to show that, if $i$ or $i'$ is a trivial cofibration in $\injsThnsset$, then so is $i\pushprodstar i'$, which we will address as \cref{prop:pushprodofjoins}.
\end{proof}

To prove \cref{prop:pushprodofjoins}, we first give explicit ways of computing joins in terms of pushouts. Recall from \cite[Definition 3.1.5]{MRR1} that a $\Thn$-space $X$ is said to be \emph{connected} if the set $\pi_0 X$ is isomorphic to a point, where $\pi_0\colon \Thnsset\to \set$ denotes the left adjoint of the inclusion.

\begin{lemma} \label{lemma:join reps}
    Let $m,m'\geq 0$ and $X,X'$ be connected $\Thn$-spaces. The join $F[m,X]\star F[m',X']$ of representables can be computed as the following pushout in $\sThnsset$.
    \begin{tz}
        \node[](1) {$(\repD\amalg \repD[m'])\times X\times X'$}; 
        \node[below of=1](2) {$(\repD[m]\star F[m'])\times X\times X'$}; 
        \node[right of=1,xshift=4cm](3) {$F[m,X] \amalg F[m',X']$}; 
        \node[below of=3](4) {$F[m,X]\star F[m',X']$}; 
        \pushout{4}; 
        \draw[right hook->](1) to (2); 
        \draw[->](1) to (3); 
        \draw[->](2) to (4); 
        \draw[right hook->](3) to (4);
    \end{tz}
\end{lemma}

\begin{proof}
    This follows from the levelwise definition of the join 
    and the formulas for coproducts in ${}^{A/}\Dset$ in terms of pushouts in $\Dset$.
\end{proof}

\begin{lemma} \label{lem:joinofconnected}
    Let $A,A'\in \Dset$ and $X,X'\in \Thnsset$ be connected objects. Then the join $(A\times X)\star (A'\times X')$ can be computed as the following pushout in $\sThnsset$
    \begin{tz}
        \node[](1) {$(A\amalg A')\times X\times X'$}; 
        \node[below of=1](2) {$(A\star A')\times X\times X'$}; 
        \node[right of=1,xshift=3.2cm](3) {$A\times X\amalg A'\times X'$}; 
        \node[below of=3](4) {$(A\times X)\star (A'\times X')$}; 
        \pushout{4}; 
        \draw[right hook->](1) to (2); 
        \draw[->](1) to (3); 
        \draw[->](2) to (4); 
        \draw[right hook->](3) to (4);
    \end{tz}
\end{lemma}

\begin{proof}
    This follows from the fact that the join preserves connected colimits by \cref{joinprescolim} and from \cref{lemma:join reps}.
\end{proof}

\begin{notation}
    For $\otimes\in \{\times,\star\}$ and monomorphisms $f\colon A\hookrightarrow B$ and $g\colon X\hookrightarrow Y$ in $\sThnsset$, we write
    \[ S(f\,\widehat{\otimes}\, g)\coloneqq A\otimes Y\amalg_{A\otimes X} B\otimes X \]
    for the source of the pushout-product map $(A\xhookrightarrow{f} B)\,\widehat{\otimes}\, (X\xhookrightarrow{g} Y)=(A\otimes Y\amalg_{A\otimes X} B\otimes X\hookrightarrow B\otimes Y)$.
\end{notation}

\begin{lemma} \label{lem:sourceofpojoin}
    Let $f\colon A\hookrightarrow B$, $f'\colon A'\hookrightarrow B'$ in $\Dset$ and $g\colon X\hookrightarrow Y$, $g'\colon X'\hookrightarrow Y'$ in $\Thnsset$ be monomorphisms with connected target. The source $S((f\pushprod g)\pushprodstar (f'\pushprod g'))$ of the pushout-join can be computed as the following pushout in $\sThnsset$
     \begin{tz}
        \node[](1) {$B\times Y\times Y'\amalg B'\times Y'\times Y$}; 
        \node[below of=1](2) {$S((f\pushprodstar f')\pushprod (g\pushprod g'))$}; 
        \node[right of=1,xshift=3.8cm](3) {$B\times Y\amalg B'\times Y'$}; 
        \node[below of=3](4) {$S((f\pushprod g)\pushprodstar (f'\pushprod g'))$}; 
        \pushout{4}; 
        \draw[right hook->](1) to (2); 
        \draw[->](1) to (3); 
        \draw[->](2) to (4); 
        \draw[right hook->](3) to (4);
    \end{tz} 
\end{lemma}

\begin{proof}
    This can be deduced from \cref{lem:joinofconnected} through multiple tedious but rather straightforward applications of Fubini's theorem for colimits.
\end{proof}

\begin{lemma} \label{cor:pushoutforjoins}
    Let $f\colon A\hookrightarrow B$, $f'\colon A'\hookrightarrow B'$ in $\Dset$ and $g\colon X\hookrightarrow Y$, $g'\colon X'\hookrightarrow Y'$ in $\Thnsset$ be monomorphisms with connected target. We have the following pushout in $\sThnsset$.
    \begin{tz}
        \node[](1) {$S((f\pushprodstar f')\pushprod (g\pushprod g'))$}; 
        \node[right of=1,xshift=3.6cm](2) {$S((f\pushprod g)\pushprodstar (f'\pushprod g'))$}; 
        \node[below of=1](3) {$(B\star B')\times Y\times Y'$}; 
        \node[below of=2](4) {$(B\times Y)\star (B'\times Y')$}; 
        \pushout{4}; 
        \draw[->](1) to (2); 
        \draw[right hook->](1) to node[left,la]{$(f\pushprodstar f')\pushprod (g\pushprod g')$} (3); 
        \draw[right hook->](2) to node[right,la]{$(f\pushprod g)\pushprodstar (f'\pushprod g')$} (4); 
        \draw[->](3) to (4);
    \end{tz} 
\end{lemma}

\begin{proof}
    This follows from the cancellation property of pushouts using the pushouts from \cref{lem:joinofconnected,lem:sourceofpojoin}.
\end{proof}

Using these computations, we can now prove that, if we have a monomorphism and a trivial cofibration in $\SegsThnsset$ of specific forms, then their pushout-join with respect to the join is also a trivial cofibration in $\SegsThnsset$. 

\begin{lemma} \label{lem:pushprodjoin2}
    Let $f\colon A\hookrightarrow B$, $f'\colon A'\hookrightarrow B'$ in $\Dset$ and $g\colon X\hookrightarrow Y$, $g'\colon X'\hookrightarrow Y'$ in $\Thnsset$ be monomorphisms with connected target such that $g$ is a trivial cofibration in $\MSThnsset$. Then the pushout-join map 
    \[ (f\pushprod g)\pushprodstar (f'\pushprod g')\]
    is a weak equivalence in $\injsThnsset$. 
\end{lemma}

\begin{proof}
    Note that, if $g$ is a trivial cofibration in $\MSThnsset$, then $g\pushprod g'$ is a trivial cofibration in $\MSThnsset$ by cartesian closedness. Hence the pushout-product map \[ (f\pushprodstar f')\pushprod (g\pushprod g') \]
    is also a trivial cofibration $\injsThnsset$ by cartesian closedness. By \cref{cor:pushoutforjoins}, the map 
    \[ (f\pushprod g)\pushprodstar (f'\pushprod g') \]
    is a pushout of the trivial cofibration $(f\pushprodstar f')\pushprod (g\pushprod g')$, and so it is also a trivial cofibration in $\injsThnsset$, as desired.
\end{proof}

We can now prove the desired result.

\begin{prop} \label{prop:pushprodofjoins}
    Let $i\colon A\hookrightarrow B$ and $i'\colon A'\hookrightarrow B'$ be monomorphisms in $\sThnsset$. If $i$ or $i'$ is a trivial cofibration in $\injsThnsset$, then so is $i\pushprodstar i'$.
\end{prop}

\begin{proof}
   Without loss of generality, we will assume that $i$ is a trivial cofibration in $\injsThnsset$ and the case for $i'$ follows analogously. 
   
   Let us first fix $i$ to be of the form $f\pushprod g$ as in \cref{lem:pushprodjoin2}. Note that the class of monomorphisms $i'$ for which $(f\pushprod g)\pushprodstar i'$ is a trivial cofibration in $\injsThnsset$ is weakly saturated. Hence it is enough to show the result for $i'$ of the form $f'\pushprod g'$ as in \cref{lem:pushprodjoin2}, which is true by \cref{lem:pushprodjoin2}. Let us now fix $i'$ to be an arbitrary monomorphism. As before, note that the class of monomorphisms $i$ for which $i\pushprodstar i'$ is a trivial cofibration in $\injsThnsset$ is weakly saturated. Hence it is enough to show the result for $i$ of the form $f\pushprod g$ as in \cref{lem:pushprodjoin2}, which follows from the first case.  
\end{proof}

We further prove that other relevant pushout-joins of monomorphisms are trivial cofibrations, which we will use to deduce the right fibration condition of the neat slice and cone~projections.

\begin{lemma} \label{lemma:join triv segal}
    Let $m,m'\geq 0$. Then the pushout-join map 
    \[ (F[0]\xhookrightarrow{\langle m\rangle} F[m]) \pushprodstar (\emptyset \hookrightarrow F[m']) \]
    is a trivial cofibration in $\SegsThnsset$. 
\end{lemma}

\begin{proof}
    By definition, the pushout-join map is given by the map
    \[ F[0]\star F[m']\amalg_{F[0]} F[m] \to F[m]\star F[m']. \]
    This map can be computed to be the Segal map $F[1+m']\amalg_{F[0]} F[m]\to F[m+1+m']$ and so is a trivial cofibration in $\SegsThnsset$. 
\end{proof}

\begin{lemma}
    Let $m,m'\geq 0$, $X\hookrightarrow Y$ be a monomorphism in $\Thnsset$ with connected target, and $X'$ be a connected $\Thn$-space. Then the pushout-join map
    \[ ((F[0]\xhookrightarrow{\langle m\rangle} F[m])\pushprod (X\hookrightarrow Y)) \pushprodstar (\emptyset \hookrightarrow F[m',X']) \]
    is a trivial cofibration in $\SegsThnsset$. 
\end{lemma}

\begin{proof}
    By \cref{lemma:join triv segal}, the pushout-join map $(F[0]\xhookrightarrow{\langle m\rangle} F[m])\pushprodstar (\emptyset \hookrightarrow F[m'])$ is a trival cofibration in $\SegsThnsset$. Hence, by cartesian closedness, the pushout-product map 
    \[ ((F[0]\xhookrightarrow{\langle m\rangle} F[m])\pushprodstar (\emptyset \hookrightarrow F[m'])) \pushprod ( (X\hookrightarrow Y) \times X')\]
    is also a trivial cofibration in $\SegsThnsset$. By \cref{cor:pushoutforjoins}, the map
    \[ ((F[0]\xhookrightarrow{\langle m\rangle} F[m])\pushprod (X\hookrightarrow Y)) \pushprodstar (\emptyset \hookrightarrow F[m',X']) \] 
    is a pushout of the above trivial cofibration, and so is also a trivial cofibration in $\SegsThnsset$. 
\end{proof}

\begin{prop} \label{prop:pushjoinofcovariant}
    Let $m\geq 0$, $X\hookrightarrow Y$ be a monomorphism in $\Thnsset$, and $A\in \sThnsset$ be an object. Then the pushout-join map
    \[ ((F[0]\xhookrightarrow{\langle m\rangle} F[m])\pushprod (X\hookrightarrow Y)) \pushprodstar (\emptyset \hookrightarrow A) \]
    is a weak equivalence in $\SegsThnsset$. 
\end{prop}

\begin{proof}
    The argument works as in \cref{prop:pushprodofjoins}.
\end{proof}

\subsection{Homotopical properties of the neat constructions} \label{sec:fibrantneatcone}

We are now ready to prove that, for any double $(\infty,n-1)$-category, the canonical projections from the neat slice and cone construction are double $(\infty,n-1)$-right fibrations. 

\begin{prop}
\label{joinconefibrant}
    Let $D$ be a fibrant object in $\SegsThnsset$ and $f\colon A\to D$ be a map in $\sThnsset$. Then the projection map 
    \[ \joinslice{D}{f}\to D\]
    is a double $(\infty,n-1)$-right fibration.
\end{prop}

\begin{proof}
Let $i'\colon A'\hookrightarrow B'$ be either a trivial cofibration in $\injsThnsset$ or the following pushout-product map, for $m\geq 0$ and $X\hookrightarrow Y$ a monomorphism in $\Thnsset$.
\[ (F[0]\xhookrightarrow{\langle m\rangle} F[m])\pushprod (X\hookrightarrow Y) \] We need to show that there is a lift in the below left commutative diagram in $\sThnsset$. By definition of the neat cone construction $\joinslice{D}{f}$, this is equivalent to showing that there is a lift in the below right commutative diagram in $\sThnsset$.
    \begin{tz}
        \node[](1) {$A'$}; 
        \node[right of=1,xshift=.5cm](2) {$\joinslice{D}{f}$}; 
        \node[below of=1](3) {$B'$}; 
        \node[below of=2](4) {$D$};

        \draw[->](1) to (2); 
        \draw[right hook->](1) to (3); 
        \draw[->](2) to (4); 
        \draw[->](3) to (4);
        \draw[->,dashed] (3) to (2);
        
        \node[right of=2,xshift=2cm](1) {$A'\star A\amalg_{A'} B'$}; 
        \node[right of=1,xshift=1.5cm](2) {$D$}; 
        \node[below of=1](3) {$B'\star A$};

        \draw[->](1) to (2); 
        \draw[right hook->](1) to (3); 
        \draw[->,dashed] (3) to (2);
        \end{tz}
        However, since the left-hand map is the pushout-join $(A'\hookrightarrow B')\pushprodstar (\emptyset\hookrightarrow A)$, it is a trivial cofibration in $\injsThnsset$  by \cref{thm:pushprodofjoins} in the case where $A'\hookrightarrow B'$ is a trivial cofibration in $\injsThnsset$, and in $\SegsThnsset$ by \cref{prop:pushjoinofcovariant} in the case where $A'\hookrightarrow B'$ is the pushout-product map $(F[0]\xhookrightarrow{\langle m\rangle} F[m])\pushprod (X\hookrightarrow Y)$. Hence the desired lift exists.
\end{proof}

As a special case, we obtain the following. 

\begin{cor} \label{joinslicefibrant}
    Let $D$ be a fibrant object in $\SegsThnsset$ and $d\in D_{0,[0],0}$. Then the projection map 
    \[ \joinslice{D}{d}\to D\]
    is a double $(\infty,n-1)$-right fibration.
\end{cor}

As a corollary, we obtain the following result.

\begin{cor} \label{starjoinLQ}
The adjunction
    \begin{tz}
\node[](1) {$\SegsThnsset$}; 
\node[right of=1,xshift=2.7cm](2) {${}^{A/} \SegsThnsset$}; 

\draw[->] ($(2.west)-(0,5pt)$) to node[below,la]{$\joinslice{(-)}{(-)}$} ($(1.east)-(0,5pt)$);
\draw[->] ($(1.east)+(0,5pt)$) to node[above,la]{$-\star A$} ($(2.west)+(0,5pt)$);

\node[la] at ($(1.east)!0.5!(2.west)$) {$\bot$};
\end{tz}
is a Quillen pair.
\end{cor}

\begin{proof} 
    To prove that it is a Quillen pair, by \cite[Proposition E.2.14]{JoyalVolumeII} it is enough to show that the left adjoint $-\star A$ sends (trivial) cofibrations in $\injsThnsset$ to (trivial) cofibrations in $\injsThnsset$, and that the right adjoint $\joinslice{(-)}{(-)}$ preserves fibrant objects. 

    Let $A'\hookrightarrow B'$ be a (trivial) cofibration in $\injsThnsset$. Consider the following diagram in $\injsThnsset$,
    \begin{tz}
        \node[](1) {$A'$}; 
        \node[below of=1](2) {$B'$}; 
        \node[right of=1,xshift=1cm](3) {$A'\star A$}; 
        \node[below of=3](4) {$A'\star A\amalg_{B'} A'$}; 
        \node[below right of=4,xshift=1.3cm](5) {$B'\star A$};
        \pushout{4};

        \draw[right hook->](1) to node[left,la]{$(\simeq)$} (2);
        \draw[right hook->](3) to node[right,la]{$(\simeq)$} (4);
        \draw[right hook->](1) to (3); 
        \draw[right hook->](2) to (4); 
        \draw[right hook->](4) to node[above,la,xshift=5pt,yshift=-2pt]{$(\simeq)$} (5);
        \draw[right hook->,bend left](3) to (5);
        \draw[right hook->,bend right=20](2) to (5);
    \end{tz}
    where the map $A'\star A\to A'\star A\amalg_{B'} A'$ is a (trivial) cofibration as a pushout of a trivial cofibration, and $A'\star A\amalg_{B'} A'\to B'\star A$ is a (trivial) cofibration in $\injsThnsset$ by \cref{thm:pushprodofjoins}, as it is the pushout-join $(A'\hookrightarrow B')\pushprodstar (\emptyset\hookrightarrow A)$. Hence the map $A'\star A\to B'\star A$ is also a (trivial) cofibration in $\injsThnsset$ as a composite of such.

    Now, let $f\colon A\to D$ be a map in $\sThnsset$ such that $D$ is fibrant in $\SegsThnsset$. By \cref{joinconefibrant}, the projection $\joinslice{D}{f}\to D$ is a double $(\infty,n-1)$-right fibration, and so $\joinslice{D}{f}$ is also fibrant in $\SegsThnsset$ by \cref{sourceofdblrightfib}. 
\end{proof}

\subsection{The study of neat fibers} \label{sec:fibersneatcone}

We now give a description of the fibers of the double $(\infty,n-1)$-right fibration given by the neat cone construction, which will be useful later.

\begin{lemma} \label{lem:equalpushoutneat}
    Let $X$ be a connected $\Thn$-space and $A$ be an object of $\sThnsset$. Consider the following pushouts in $\sThnsset$.
    \begin{tz}
        \node[](1) {$F[0,X]\amalg A$}; 
        \node[below of=1](2) {$F[0,X]\star A$}; 
        \node[right of=1,xshift=1.9cm](3) {$F[0]\amalg A$}; 
        \node[below of=3](4) {$P(A,X)$}; 
        \pushout{4};

        \draw[right hook->](1) to (2); 
        \draw[->](1) to (3); 
        \draw[->](2) to (4); 
        \draw[right hook->](3) to (4);

        \node[right of=3,xshift=2.5cm](1) {$F[0]\amalg (A\times X)$}; 
        \node[below of=1](2) {$F[0]\star (A\times X)$}; 
        \node[right of=1,xshift=2.2cm](3) {$F[0]\amalg A$}; 
        \node[below of=3](4) {$P'(A,X)$}; 
        \pushout{4};

        \draw[right hook->](1) to (2); 
        \draw[->](1) to (3); 
        \draw[->](2) to (4); 
        \draw[right hook->](3) to (4);
    \end{tz}
    Then there is a canonical isomorphism in ${}^{F[0]\amalg A/}\sThnsset$
    \[ P(A,X)\cong P'(A,X). \]
\end{lemma}

\begin{proof}
    First note that the constructions $P(A,X)$ and $P'(A,X)$ define colimit-preserving functors
    \[ P(-,X),P'(-,X)\colon \sThnsset\to {}^{F[0]/}\sThnsset. \]
    Indeed, the functor $P(-,X)$ is the composite of colimit-preserving functors \[ \sThnsset\xrightarrow{F[0,X]\star -} {}^{F[0,X]/}\sThnsset \xrightarrow{\varphi} {}^{F[0]/}\sThnsset, \]
where $\varphi(F[0,X]\to Z)$ is the pushout of $F[0]\leftarrow F[0,X]\to Z$, and the functor $P'(-,X)$ is the composite of colimit-preserving functors
\[ \sThnsset\xrightarrow{\psi} ({}^{F[0]/}\sThnsset)^{\bullet\leftarrow \bullet \to \bullet} \xrightarrow{\mathrm{colim}}{}^{F[0]/}\sThnsset, \]
where $\psi(A)$ is the span $F[0]\amalg A\leftarrow F[0]\amalg (A\times X)\to F[0]\star (A\times X)$. Hence, to show that they agree on all objects $A\in \sThnsset$, it is enough to show that they agree on representables. 

    Let $m'\geq 0$ and $X'$ be a connected $\Thn$-spaces. By pasting the pushouts from the statement defining $P(F[m',X'],X)$ and $P'(F[m',X'],X)$ to the pushouts from \cref{lemma:join reps} giving $F[0,X]\star F[m',X']$ and $F[0]\star (F[m',X']\times X)$ respectively, we see that $P(F[m',X'],X)$ and $P'(F[m',X'],X)$ can both be computed as the pushout of the following span
    \begin{tz}
        \node[](1) {$(F[0]\star F[m'])\times X\times X'$}; 
        \node[right of=1,xshift=4cm](2) {$(F[0]\amalg F[m'])\times X\times X'$}; 
        \node[right of=2,xshift=3.3cm](3) {$F[0]\amalg F[m',X']$}; 
        \draw[left hook->](2) to (1); 
        \draw[->](2) to (3);
    \end{tz}
    and so they are canonically isomorphic.
\end{proof}

\begin{lemma} \label{joinconeonfibers}
    Let $Z\in \pcatThn$ and $f\colon A\to Z$ be a map in $\sThnsset$. Given an element $x\in Z_{0,0,0}$, there is a canonical isomorphism
    \[ (\joinslice{Z}{f})_x\xrightarrow{\cong} \Hom_{\sThnssetslice{Z}}(A\xrightarrow{f} Z,\joinsliceunder{Z}{x}\to Z). \]
\end{lemma}

\begin{proof}
    First observe that mapping out of a connected $\Thn$-space into the left-hand (resp.~right-hand) side amounts to a map in ${}^{F[0]\amalg A/}\sThnsset$ 
    \[ P(A,X)\to Z \quad (\text{resp.~} P'(A,X)\to Z). \]
    Then the result follows from \cref{lem:equalpushoutneat}.
\end{proof}

\section{Comparison of slice and cone constructions}

\label{SectionFatVsNeat}

In this section, we enhance to the context of double $(\infty,n-1)$-categories the arguments from \cite[Appendix~D.4]{RVelements} (see also~\cite[\textsection9.18]{JoyalNotes}, \cite[\textsection4.2.1]{LurieHTT} and \cite[Appendix~A]{RVqCatv3}), which compares the fat and neat slice and cone constructions. For this, we first construct in \cref{sec:comparisonmap} a comparison map between the fat and neat slice and cone constructions. Then, in \cref{sec:comparisonmapequivalence}, we prove that, over a double $(\infty,n-1)$-category, the comparison map is an equivalence of double $(\infty,n-1)$-right fibrations.

\subsection{Construction of the comparison map} \label{sec:comparisonmap}

We now construct a comparison map between the fat and neat slice and cone constructions. For this, we start by constructing a map between their associated joins. 

\begin{constr}
    Let $m,m'\geq 0$. We define a map in $\Dset$
    \[ \gamma_{m,m'}\colon \repD\times \repD[1]\times \repD[m']\to \repD[m+1+m'], \quad 
         (i,\varepsilon,i') \mapsto \begin{cases}
        i & \text{if } \varepsilon=0 \\
        m+1+i' & \text{if } 
        \varepsilon=1
    \end{cases}\]
    It is straightforward to check that this map is natural in $m,m'$. Moreover, it induces a map
    \[ \gamma_{m,m'}\colon F[m]\diamond F[m']\to F[m]\star F[m']. \]
\end{constr}

\begin{lemma} \label{gammaonreps}
    Let $m,m'\geq 0$ and $X,X'$ be connected $\Thn$-spaces. Then there is a pushout in $\sThnsset$ of the form 
\begin{tz}
        \node[](1) {$(\repD\diamond \repD[m'])\times X\times X'$}; 
        \node[below of=1](2) {$(\repD\star\repD[m'])\times X\times X'$}; 
        \node[right of=1,xshift=4cm](3) {$F[m,X]\diamond F[m',X']$}; 
        \node[below of=3](4) {$F[m,X]\star F[m',X']$}; 
        \pushout{4};

        \draw[->](1) to  node[left,la]{$\gamma_{m,m'}\times X\times X'$} (2); 
        \draw[->](1) to (3); 
        \draw[->](2) to (4); 
        \draw[->](3) to node[right,la]{$\gamma_{F[m,X],F[m',X']}$} (4);
    \end{tz}
\end{lemma}

\begin{proof}
    Consider the following diagram in $\sThnsset$.
    \begin{tz}
    \node[](1) {$\coprod_2 F[m,X]\times F[m',X']$}; 
    \node[right of=1,xshift=4.4cm](2) {$(F[m]\amalg F[m'])\times X\times X'$}; 
    \node[below of=1](3) {$F[m,X]\times F[1]\times F[m',X']$}; 
    \node[below of=2](4) {$(F[m]\diamond F[m'])\times X\times X'$}; \node[below of=4](40) {$(F[m]\star F[m'])\times X\times X'$};
    \pushout{4};
    
    \node[right of=2,xshift=4cm](2') {$F[m,X]\amalg F[m',X']$};  
    \node[below of=2'](4') {$F[m,X]\diamond F[m',X']$};  
    \node[below of=4'](41) {$F[m,X]\star F[m',X']$};

    \draw[->](1) to (2); 
    \draw[right hook->](1) to (3);  
    \draw[->](2) to (2');
    \draw[->](3) to (4); 
    \draw[->,dashed](4) to (4'); 
    \draw[right hook->](2) to (4);
    \draw[right hook->](2') to (4');
    \draw[->](4) to node[left,la]{$\gamma_{m,m'}\times X\times X'$} (40); 
    \draw[->,dashed](4') to node[right,la]{$\gamma_{F[m,X],F[m',X']}$} (41); 
    \draw[->] (40) to (41);
\end{tz}
The top left square is a pushout by definition of $\diamond$ and the fact that products commute with pushouts in $\sThnsset$. Hence, the pushout from \cref{constr:join} defining $\diamond$ yields an induced horizontal dashed map such that the composite of the two top squares is this pushout. By the cancellation property of pushouts, the top right square is a pushout. Hence, the pushout from \cref{lemma:join reps} yields an induced vertical dashed map such that the composite of the two right-hand squares is this pushout. Finally, by the cancellation property of pushouts, the bottom right square is a pushout, as desired. 
\end{proof}

The above maps induce a comparison between the two slice constructions.

\begin{constr} \label{def:comparisonslices}
    Let $Z$ be an object in $\sThnsset$ and $x\in Z_{0,0,0}$. Then the maps
    \[ \gamma_{\repDThnS,F[0]}\colon \repDThnS\diamond F[0]\to \repDThnS\star F[0] \]
    from \cref{gammaonreps}
    induce a comparison map in $\sThnssetslice{Z}$
    \[ \joinslice{Z}{x}\to \slice{Z}{x}. \]
\end{constr}

To further give a comparison map between the cone constructions, we first need the following.

\begin{constr} \label{constr:generalgamma}
    Let $m\geq 0$, $X$ be a connected $\Thn$-space, and $A\in \sThnsset$ be an object. Using that the functors 
    \[ F[m,X]\diamond -,F[m,X]\star -\colon \sThnsset \to {}^{F[m,X]/}\sThnsset \]
    preserve colimits, we define the map in $\sThnsset$
    \[ \gamma_{F[m,X],A}\colon F[m,X]\diamond A\to F[m,X]\star A \]
    to be given by the colimit in ${}^{F[m,X]/}\sThnsset$
    \[ \colim_{F[m',\theta',k']\to A} (F[m,X]\diamond F[m',\theta',k']\xrightarrow{\gamma_{F[m,X],F[m',\theta',k']}} F[m,X]\star F[m',\theta',k']), \]
    where $\gamma_{F[m,X],F[m',\theta',k']}$ is the map from \cref{gammaonreps}.

    Similarly, given objects $A,B\in \sThnsset$, we can further define a map in $\sThnsset$
    \[ \gamma_{B,A}\colon B\diamond A\to B\star A. \]
\end{constr}

\begin{constr} \label{def:comparisoncones}
    Let $f\colon A\to Z$ be a map in $\sThnsset$. The maps
    \[ \gamma_{\repDThnS,A}\colon \repDThnS\diamond A\to \repDThnS\star A \]
    from \cref{constr:generalgamma}
    induce a comparison map in $\sThnssetslice{Z}$
    \[ \joinslice{Z}{f}\to \slice{Z}{f}.\]
\end{constr}

\subsection{Homotopical properties of the comparison map}
\label{sec:comparisonmapequivalence}

We now want to prove that, for any double $(\infty,n-1)$-category, the comparison map between slice and cone constructions is a weak equivalence. Recall the following result from \cite[Lemma~D.2.16]{RVelements} (cf.~\cite[Proposition~A.4.7]{RVqCat}). 

\begin{lemma} \label{gammaishtpyeq}
    Given $m,m'\geq 0$, the comparison map 
    \[ \gamma_{m,m'}\colon F[m]\diamond F[m']\to F[m]\star F[m'] \]
    is a homotopy equivalence in $\injsThnsset$. 
\end{lemma}

As a consequence, we can deduce that the comparison map between joins is a weak equivalence. 

\begin{prop} \label{gammaonrepiswe}
    Given $m,m'\geq 0$ and $X,X'$ connected $\Thn$-spaces, the comparison map 
    \[ \gamma_{F[m,X],F[m',X']}\colon F[m,X]\diamond F[m',X]\to F[m,X]\star F[m',X] \]
    is a weak equivalence in $\injsThnsset$, and so also in $\SegsThnsset$. 
\end{prop}

\begin{proof}
By the proof of \cref{gammaonreps}, we have the following pushout diagrams in $\injsThnsset$,
    \begin{tz}
    \node[](2) {$(F[m]\amalg F[m'])\times X\times X'$};  
    \node[right of=2,xshift=4.2cm](4) {$(F[m]\diamond F[m'])\times X\times X'$}; 
    \node[right of=4,xshift=4.8cm](40) {$(F[m]\star F[m'])\times X\times X'$};
    
    \node[below of=2](2') {$F[m,X]\amalg F[m',X']$};  
    \node[below of=4](4') {$F[m,X]\diamond F[m',X']$};  
    \node[below of=40](41) {$F[m,X]\star F[m',X']$};
    
    \pushout{4'};
    \pushout{41};

 \draw[right hook->,bend right=8](2') to (41);
    \draw[->](2) to (2'); 
    \draw[->](4) to (4'); 
    \draw[right hook->](2) to (4);
    \draw[right hook->](2') to (4');
    \draw[->](4) to node[above,la]{$\gamma_{m,m'}\times X\times X'$} (40); 
    \draw[->](4') to node[above,la]{$\gamma_{F[m,X],F[m',X']}$} (41); 
    \draw[->] (40) to (41);
\end{tz}
where the map $\gamma_{m,m'}\times X\times X'$ is a weak equivalence in $\injsThnsset$ by \cref{gammaishtpyeq} and cartesian closedness of $\injsThnsset$. Since the pushout functor 
\[ {}^{(F[m]\amalg F[m'])\times X\times X'/} \injsThnsset\to {}^{F[m,X]\amalg F[m',X']/}\injsThnsset  \]
preserves weak equivalences between cofibrant objects, it follows that the map $\gamma_{F[m,X],F[m',X']}$ is also a weak equivalence in $\injsThnsset$, as desired.
\end{proof}

\begin{prop} \label{gammaiswe}
    Let $A$ and $B$ be objects in $\sThnsset$. The map 
    \[ \gamma_{B,A}\colon B\diamond A\to B\star A \]
    is a weak equivalence in $\SegsThnsset$.
\end{prop}

\begin{proof}
For $m\geq 0$ and $X$ a connected $\Thn$-space, let $\cA$ be the class of all objects $A\in \sThnsset$ for which the map \[ \gamma_{F[m,X],A}\colon F[m,X]\diamond A\to F[m,X]\star A \] is a weak equivalence in $\SegsThnsset$. By \cref{gammaonrepiswe}, all representables of $\sThnsset$ are in $\cA$. Moreover, the class $\cA$ is saturated by monomorphisms in the sense of \cite[Definition 1.3.9]{CisinskiBook} since the functors 
\[ F[m,X]\diamond -, F[m,X]\star -\colon \SegsThnsset\to {}^{F[m,X]/} \SegsThnsset \]
are left Quillen by \cref{diamondjoinLQ,starjoinLQ}. Hence, we deduce by \cite[Corollary 1.3.10.]{CisinskiBook} that every object $A\in \sThnsset$ is in $\cA$. By fixing $A\in \sThnsset$ and repeating the above argument to the class $\cB$ of all objects $B\in \sThnsset$ for which the map $\gamma_{B,A}\colon B\diamond A\to B\star A$ is a weak equivalence in $\SegsThnsset$, we get the desired result. 
\end{proof}

As a corollary, we obtain the desired result.

    \begin{thm} \label{thm:comparisoniswe}
    Let $D$ be a fibrant object in $\SegsThnsset$ and $f\colon A\to D$ be a map in $\sThnsset$. Then the comparison map 
    \[ \joinslice{D}{f}\to \slice{D}{f} \]
    is a weak equivalence in $\MSrightfib{D}$, and so also in $\injsThnsset$. 
\end{thm}

\begin{proof}
    This follows directly from \cref{diamondjoinLQ,starjoinLQ,gammaiswe} using \cite[Corollary~1.4.4(b)]{Hovey}.
\end{proof}

As a special case, we get the following.

\begin{cor} \label{cor:comparisoniswe}
    Let $D$ be a fibrant object in $\SegsThnsset$ and $d\in D_{0,[0],0}$. Then the comparison map 
    \[ \joinslice{D}{d}\to \slice{D}{d} \]
    is a weak equivalence in $\MSrightfib{D}$, and so also in $\injsThnsset$.
\end{cor}

\section{Invariance of cone constructions under Dwyer-Kan equivalences} 

\label{SectionDK}

In this section, we study the invariance of the slice and cone constructions under certain weak equivalences. In \cref{sec:invariancelevelwise}, we first show that these constructions are invariant under levelwise Dwyer-Kan equivalences of Segal spaces. Then, in \cref{sec:invarianceDKeq}, we deduce that they are invariant under Dwyer-Kan equivalences (equivalently weak equivalences in $\CSSsThnsset$) between double $(\infty,n-1)$-categories whose $(\infty,n-1)$-category of objects is equivalent to a space. 

This result can then be applied to the Dwyer-Kan equivalence given by a fibrant replacement in $\CSSsThnsset$ of a fibrant object in $\pcatinj$, and so this implies that the cone construction associated to a fibrant object in $\pcatinj$ is homotopically meaningful. 

\subsection{Invariance under levelwise Dwyer-Kan equivalences} \label{sec:invariancelevelwise}

We now prove that the cone constructions are invariant under levelwise weak equivalence in $\MSDsset$ between fibrant objects in $\SegsThnsset$. In particular, by \cite[Theorem 7.7]{RezkCSS}, these precisely coincide with the levelwise \emph{Dwyer-Kan equivalences} of Segal spaces. For this, we first recall the following definition.

\begin{defn}
    Let $C$ and $D$ be fibrant objects in $\SegsThnsset$. A map $w\colon C \to D$ in $\sThnsset$ is \textbf{homotopically fully faithful} if the induced map 
\[ C^{F[1]} \to D^{F[1]} \times_{D^{\partial F[1]}} C^{\partial F[1]} \]
is a weak equivalence in $\SegsThnsset$. 
\end{defn}

We first show that, if a map $w\colon C\to D$ between fibrant objects in $\SegsThnsset$ is homotopically fully faithful, then it induces a weak equivalence between the cone construction $\slice{C}{f}$ and the pullback $w^*\slice{D}{wf}$ of the cone construction $\slice{D}{wf}$ along $w$.

\begin{lemma} \label{lemma:exponent ff}
    Let $C$ and $D$ be fibrant objects in $\SegsThnsset$ and $A\in \sThnsset$ be an object. If a map $w\colon C \to D$ in $\sThnsset$ is homotopically fully faithful, then so is the induced map $w^A\colon C^A \to D^A$.
\end{lemma}

\begin{proof}
    This follows from the fact that $\SegsThnsset$ is cartesian closed. 
\end{proof}

\begin{lemma} \label{adjunctiswe}
    Let $C$ and $D$ be fibrant objects in $\SegsThnsset$, $w\colon C\to D$ be a homotopically fully faithful map in $\sThnsset$, and $f\colon A\to C$ be a map in $\sThnsset$. Then the map
    \[ \slice{C}{f} \to w^*\slice{D}{wf} \]
    is a weak equivalence in $\MSrightfib{C}$.
\end{lemma}

\begin{proof}
    First note that $\slice{C}{f}\to C$ and $w^*\slice{D}{wf}\to C$ are double $(\infty,n-1)$-right fibrations by \cref{postcomp-pullback,prop:fibrantcone}, and so $\slice{C}{f}$ and $w^*\slice{D}{wf}$ are fibrant in $\SegsThnsset$ by \cref{sourceofdblrightfib}. Hence, proving that $\slice{C}{f} \to w^*\slice{D}{wf}$ is a weak equivalence in $\MSrightfib{C}$ is equivalent to proving that it is a weak equivalence in $\injsThnsset$ or equivalently in $\SegsThnsset$. We show the latter.

  Consider the following commutative diagram in $\SegsThnsset$.
    \begin{tz}
        \node[](1) {$\slice{C}{f}$}; 
        \node[right of=1,xshift=3.1cm](1') {$w^*\slice{D}{wf}$}; 
        \node[right of=1',xshift=3.2cm](1'') {$C$}; 
        \node[below of=1](2) {$(C^A)^{F[1]}$}; 
        
        \node[below of=1''](2'') {$C^A\times C^A$}; 
        
        \draw[->>,bend right=12](2) to (2'');
        \node[below of=1'](2') {$(D^A)^{F[1]}\times_{(D^A)^{\partial F[1]}} (C^A)^{\partial F[1]}$}; 
        \pullback{1'};

        \draw[->](1) to (2); 
        \draw[->](1') to (2'); 
        \draw[->](1'') to node[right,la]{$\Delta\times \{f\}$} (2''); 
        \draw[->](1) to (1'); 
        \draw[->>](1') to (1''); 
        \draw[->](2) to node[above,la]{$\simeq$} (2'); 
        \draw[->>](2') to (2'');
    \end{tz}
    where the map marked with $\simeq$ is a weak equivalence by \cref{lemma:exponent ff}, the right-hand and total squares are pullbacks using \cref{sliceaspullback} and the definition of the pullback functor $w^*$. Then it follows from the cartesian closedness of $\SegsThnsset$ that all maps marked with $\twoheadrightarrow$ are fibrations. Hence, since the pullback functor 
    \[ (\Delta\times \{f\})^*\colon {\SegsThnsset}_{/C^A\times C^A}\to {\SegsThnsset}_{/C} \]
    preserves weak equivalences between fibrant objects, it follows that the map $\slice{C}{f}\to w^*\slice{D}{wf}$ is also a weak equivalence in $\SegsThnsset$, as desired.  
\end{proof}

Finally, we show that, if a map between fibrant objects in $\SegsThnsset$ is levelwise a weak equivalence in $\MSDsset$, then it induces a weak equivalence between the associated cone constructions.

\begin{lemma} \label{wisff}
    Let $C$ and $D$ be fibrant objects in $\SegsThnsset$ and $w\colon C\to D$ be a map in $\sThnsset$ such that, for every object $\theta\in \Thn$, the induced map $w_{-,\theta}\colon C_{-,\theta}\to D_{-,\theta}$ is a weak equivalence in $\MSDsset$. Then the map $w$ is homotopically fully faithful.
\end{lemma}

\begin{proof}
    Given an object $\theta\in \Thn$, since $w_{-,\theta}\colon C_{-,\theta}\to D_{-,\theta}$ is a weak equivalence in $\MSDsset$ between Segal spaces, the induced map in $\Dsset_{\;/(C_{-,\theta})^{\partial F[1]}}$ between \emph{bifibrations} (see \cite[Definition 2.4.7.2]{LurieHTT})
    \[ (C_{-,\theta})^{F[1]}\to (D_{-,\theta})^{F[1]}\times_{(D_{-,\theta})^{\partial F[1]}} (C_{-,\theta})^{\partial F[1]} \]
    is such that the induced maps on fibers are weak equivalences in $\MSspace$ by \cite[Proposition 13.8]{RezkCSS}, and so it is a weak equivalence in $\MSDsset$ by \cite[Proposition 2.4.7.6]{LurieHTT}. But the above weak equivalence is simply the $\theta$-component of the map 
    \[ C^{F[1]}\to D^{F[1]}\times_{D^{\partial F[1]}} C^{\partial F[1]}, \]
    which is therefore a weak equivalence in $\SegsThnsset$, as desired. 
\end{proof}

\begin{prop}
\label{DKeqoncones}
    Let $C$ and $D$ be fibrant objects in $\SegsThnsset$, $w\colon C\to D$ be a map in $\sThnsset$ such that, for every object $\theta\in \Thn$, the induced map $w_{-,\theta}\colon C_{-,\theta}\to D_{-,\theta}$ is a weak equivalence in $\MSDsset$, and $f\colon A\to C$ be a map in $\sThnsset$. The following statements hold:
    \begin{enumerate}[leftmargin=1cm]
        \item the map $w_! \slice{C}{f}\to \slice{D}{wf}$ is a weak equivalence in $\MSrightfib{D}$, 
        \item the map $w_! \joinslice{C}{f}\to \joinslice{D}{wf}$ is a weak equivalence in $\MSrightfib{D}$. 
    \end{enumerate}
\end{prop}

\begin{proof}
    We first prove (1). By \cref{postcomp-pullback}, the adjunction
    \begin{tz}
\node[](1) {$\MSrightfib{C}$}; 
\node[right of=1,xshift=3.9cm](2) {$\MSrightfib{D}$}; 

\draw[->] ($(2.west)-(0,5pt)$) to node[below,la]{$w^*$} ($(1.east)-(0,5pt)$);
\draw[->] ($(1.east)+(0,5pt)$) to node[above,la]{$w_!$} ($(2.west)+(0,5pt)$);

\node[la] at ($(1.east)!0.5!(2.west)$) {$\bot$};
\end{tz}
is a Quillen equivalence. So, in order to show that the map $w_! \slice{C}{f}\to \slice{D}{wf}$ is a weak equivalence in $\MSrightfib{D}$, it suffices to show that the adjunct map 
\[ \slice{C}{f} \to w^*\slice{D}{wf} \]
    is a weak equivalence in $\MSrightfib{C}$, using that $\slice{D}{wf}\to D$ is a double $(\infty,n-1)$-right fibration by \cref{prop:fibrantcone}. However, this follows from \cref{adjunctiswe,wisff}.

    Then (2) follows from (1) and \cref{thm:comparisoniswe} using that $w_!$ preserves weak equivalences.
\end{proof}

\subsection{Invariance under Dwyer-Kan equivalences} \label{sec:invarianceDKeq}

Recall from \cite[Theorem 8.18]{BR2} that weak equivalences in $\CSSsThnsset$ between fibrant objects in $\SegsThnsset$ coincide with \emph{Dwyer-Kan equivalences}, which are given via explicit conditions on mapping $\Thn$-spaces and homotopy categories; see e.g.~\cite[\textsection1.3]{MRR1}.

We prove here that the cone constructions are invariant under Dwyer-Kan equivalences between fibrant objects in $\SegsThnsset$ whose level $0$ are homotopically constant. 

\begin{defn}
    A $\Thn$-space $X$ is \textbf{homotopically constant} if, for every object $\theta\in \Thn$, the induced map $X_{[0]}\to X_\theta$ is a weak equivalence in $\MSspace$.
\end{defn}

We first state the following result, whose proof can be obtained using the same arguments as in the third paragraph of the proof of \cite[Proposition 2.1.5]{MRR2}.

\begin{lemma} \label{lemma:fib rep}
    Let $C$ be a fibrant object in $\SegsThnsset$ with $C_0$ homotopically constant. There is a monomorphism $i_C\colon C \hookrightarrow \widehat{C}$ in $\sThnsset$ with the following properties:
    \begin{enumerate}[leftmargin=1cm]
        \item for every object $\theta\in \Thn$, the map $(i_C)_{-,\theta}\colon C_{-,\theta}\hookrightarrow \widehat{C}_{-,\theta}$ is a weak equivalence in $\MSDsset$,
        \item the map $i_C\colon C\hookrightarrow \widehat{C}$ is a weak equivalence in $\CSSsThnsset$ (equivalently a Dwyer-Kan equivalence), 
        \item the object $\widehat{C}$ is fibrant in $\CSSsThnsset$.
    \end{enumerate}
\end{lemma}

As a consequence, we obtain the following.

\begin{prop} \label{prop:DKeqbtwhtyconstant} 
    Let $C$ and $D$ be fibrant objects in $\SegsThnsset$ with $C_0$ and $D_0$ homotopically constant, and let $w\colon C\to D$ be a weak equivalence in $\CSSsThnsset$ (equivalently a Dwyer-Kan equivalence). Then, for every object $\theta\in \Thn$, the induced map 
    \[ w_{-,\theta}\colon C_{-,\theta}\to D_{-,\theta} \]
    is a weak equivalence in $\MSDsset$.
\end{prop}

\begin{proof}
    First note that a map $\widehat{w}\colon \widehat{C}\to \widehat{D}$ between fibrant objects in $\CSSsThnsset$ is a weak equivalence in $\CSSsThnsset$ if and only if it is a weak equivalence in $\injsThnsset$ if and only if, for every object $\theta\in \Thn$, the induced map $\widehat{w}_{-,\theta}$ is a weak equivalence in $\MSDsset$. Hence the result follows by $2$-out-of-$3$ when considering the maps $i_C\colon C\to \widehat{C}$ and $i_D\colon D\to \widehat{D}$ from \cref{lemma:fib rep} and a lift $\widehat{w}\colon \widehat{C}\to \widehat{D}$ such that $i_D\circ w=\widehat{w}\circ i_C$. 
\end{proof}

\begin{rmk} \label{fibreplofprecat}
    If $C$ is a fibrant object in $\pcatinj$ and $w\colon C\to D$ is a fibrant replacement in $\CSSsThnsset$, it is straightforward to see that $C$, $D$, and $w\colon C\to D$ satisfy the hypotheses of \cref{prop:DKeqbtwhtyconstant}.
\end{rmk}

\begin{cor} \label{cor:coneofprecat}
    Let $C$ and $D$ be fibrant objects in $\SegsThnsset$ with $C_0$ and $D_0$ homotopically constant. Let $w\colon C\to D$ be a weak equivalence in $\CSSsThnsset$ (equivalently a Dwyer-Kan equivalence), and $f\colon A\to C$ be a map in $\sThnsset$. Then the following statements hold:
    \begin{enumerate}[leftmargin=1cm]
        \item the map $w_! \slice{C}{f}\to \slice{D}{wf}$ is a weak equivalence in $\MSrightfib{D}$, 
        \item the map $w_! \joinslice{C}{f}\to \joinslice{D}{wf}$ is a weak equivalence in $\MSrightfib{D}$. 
    \end{enumerate}
    This holds in particular if $C$ is a fibrant object in $\pcatinj$ and $w\colon C\to D$ is a fibrant replacement in $\CSSsThnsset$. 
\end{cor}

\begin{proof}
    This follows from \cref{DKeqoncones,prop:DKeqbtwhtyconstant,fibreplofprecat}.
\end{proof}

\part{Comparison of \texorpdfstring{$(\infty,n)$}{(infinity,n)}-limits} \label{part3}

\section{Enriched vs internal approaches to \texorpdfstring{$(\infty,n)$}{(infinity,n)}-limits}

\label{SectionLimits}

In this section, we first recall in \cref{enrichedlimit} the notion of weighted $(\infty,n)$-limits in the context of $(\infty,n)$-categories modeled by $\MSThnsset$-enriched categories. Then, in \cref{internalterminal}, we introduce and characterize double $(\infty,n-1)$-terminal objects in preparation for the definition of weighted $(\infty,n)$-limits in the context of $(\infty,n)$-categories modeled by complete Segal object in $\MSThnsset$. These are defined in \cref{internallimit} as double $(\infty,n-1)$-terminal objects in the corresponding cone construction. 

To prove the comparison between the two notions of weighted $(\infty,n)$-limits, we first show in \cref{comparisonslices} that there is an isomorphism between the unstraightening of the $\Thnsset$-enriched functor $\Hom_\cC(-,c)\colon \cC^{\op}\to \Thnsset$ and the slice construction $\joinslice{\Nh\cC}{c}$. Finally, in \cref{comparisonlimit}, we prove our main theorem, \cref{EquivalenceOfLimitsWeighted}, showing that the notions of weighted $(\infty,n)$-limits in the enriched and internal approaches agree, assuming a technical result proven in \cref{sec:Yonedamap}.

\subsection{Enriched approach to \texorpdfstring{$(\infty,n)$}{(infinity,n)}-limits} \label{enrichedlimit}

We first recall the definition of a weighted $(\infty,n)$-limit in the enriched setting. 

\begin{defn}
    Let $\cJ$ and $\cC$ be $\Thnsset$-enriched categories with $\cC$ fibrant in $\MSThncat$. Let $W\colon \cJ\to \Thnsset$ and $F\colon \cJ\to \cC$ be $\Thnsset$-enriched functors such that $W$ is cofibrant in $[\cJ,\MSThnsset]_\proj$. An object $\ell\in \cC$ is a \textbf{$W$-weighted $(\infty,n)$-limit of $F$} if there is a $\Thnsset$-enriched natural transformation $\lambda\colon W\Rightarrow \cC(\ell,F-)$ in $[\cJ,\Thnsset]$ such that the induced map
    \[ \lambda^*\colon \Hom_\cC(-,\ell)\to \Hom_{[\cJ,\Thnsset]}(W,\Hom_\cC(-,F)) \]
    is a weak equivalence in $[\cC^{\op},\MSThnsset]_\proj$.
\end{defn}

As a special case, we obtain \emph{conical} $(\infty,n)$-limits. 

\begin{defn}
    Let $\cJ$ and $\cC$ be $\Thnsset$-enriched categories with $\cC$ fibrant in $\MSThncat$. Let $F\colon \cJ\to \cC$ be a $\Thnsset$-enriched functor. Then an object $\ell\in \cC$ is a \textbf{conical $(\infty,n)$-limit of $F$} if it is a $W$-weighted $(\infty,n)$-limit of $F$ for $W$ a cofibrant replacement in $[\cJ,\MSThnsset]_\proj$ of the terminal weight $\Delta[0]\colon \cJ\to \Thnsset$.
\end{defn}

\begin{rmk}
    Let $J$ be an object in $\pcatThn$ and $\cC$ be a fibrant $\MSThnsset$-enriched category. Let $F\colon \Ch J\to \cC$ be a $\Thnsset$-enriched functor. Then a cofibrant replacement in $[\Ch J,\MSThnsset]_\proj$ of the constant weight $\Delta[0]\colon \Ch J\to \Thnsset$ is given by the (derived) counit of $\St_J\dashv \Un_J$ at $\Delta[0]$ \[ \St_J(\id_J)\cong\St_J(\Un_J(\Delta[0]))\xrightarrow{\simeq} \Delta[0]. \] 
    Hence the conical $(\infty,n)$-limit of $F$ can be computed as the $\St_J(\id_J)$-weighted $(\infty,n)$-limit of~$F$.
\end{rmk}

\subsection{Internal approach to terminal objects} \label{internalterminal}

In order to define $(\infty,n)$-limits in the internal setting, we need a notion of terminal objects. We first study equivalent ways of formulating double $(\infty,n-1)$-terminal objects. 

The following result says that the identity at an object is terminal in the corresponding slice.

\begin{prop} \label{slicehasterminalobject}
    Let $D$ be a fibrant object in $\SegsThnsset$ and $d\in D_{0,0,0}$. Then the following statements hold:
    \begin{enumerate}[leftmargin=1cm]
        \item the map $\id_d\colon \repD[0]\to \slice{D}{d}$ is a weak equivalence in $\MSrightfib{D}$, 
    \item the map $\id_d\colon \repD[0]\to \joinslice{D}{d}$ is a weak equivalence in $\MSrightfib{D}$.
    \end{enumerate}
\end{prop}

\begin{proof}
    (1) is proven in \cite[Theorem 5.25]{rasekh2023Dyoneda}, and (2) follows from (1) and \cref{cor:comparisoniswe}.
\end{proof}

From this statement, we can deduce the following result, which is an implementation of the Yoneda lemma in the fibrational setting.

\begin{rmk}
\label{prop:Yonedamap}
    Let $D$ be a fibrant object in $\SegsThnsset$ and $d\in D_{0,[0],0}$. Let $p\colon A\to D$ be a double $(\infty,n-1)$-right fibration, and $a\in (A_d)_{[0],0}$ be an element of the fiber of $p$ at $d$. Then there is a (homotopically unique) lift in either of the following diagrams in $\MSrightfib{D}$,
    \begin{tz}
        \node[](1) {$F[0]$}; 
        \node[below of=1](2) {$\slice{D}{d}$}; 
        \node[right of=1,xshift=.5cm](3) {$A$}; 
        \node[below of=3](4) {$D$}; 

        \draw[right hook->](1) to node[left,la]{$\id_d$} node[right,la]{$\simeq$} (2); 
        \draw[->](1) to (3); 
        \draw[->>](3) to node[right,la]{$p$} (4); 
        \draw[->](2) to (4);

        \draw[->,dashed](2) to node[left,la,yshift=5pt]{$a_*$} (3);

        \node[right of=3,xshift=1.2cm](1) {$F[0]$}; 
        \node[below of=1](2) {$\joinslice{D}{d}$}; 
        \node[right of=1,xshift=.5cm](3) {$A$}; 
        \node[below of=3](4) {$D$}; 

        \draw[right hook->](1) to node[left,la]{$\id_d$} node[right,la]{$\simeq$} (2); 
        \draw[->](1) to (3); 
        \draw[->>](3) to node[right,la]{$p$} (4); 
        \draw[->](2) to (4);

        \draw[->,dashed](2) to node[left,la,yshift=5pt]{$a_*$} (3);
    \end{tz}
    where $\id_d\colon F[0]\to \slice{D}{d}$ and $\id_d\colon F[0]\to \joinslice{D}{d}$ are weak equivalences by \cref{slicehasterminalobject}. In what follows, we will denote by $a_*$ any choice of such a lift and refer to $a_*$ as the \emph{Yoneda map}.
\end{rmk}

We now introduce terminal objects.

\begin{defn}
    Let $D$ be a fibrant object in $\SegsThnsset$ and $p\colon A\to D$ be a double $(\infty,n-1)$-right fibration. A \textbf{double $(\infty,n-1)$-terminal object} of $A$ is an element $a\in A_{0,[0],0}$ such that the canonical projection 
    \[ \slice{A}{a}\to A \]
    is a weak equivalence in $\SegsThnsset$.
\end{defn}

We have the following characterization of double $(\infty,n-1)$-terminal objects.

\begin{prop} \label{lem:terminalequivrelative}
    Let $D$ be a fibrant object in $\SegsThnsset$ and $d\in D_{0,0,0}$. Let $p\colon A\to D$ be a double $(\infty,n-1)$-right fibration and $a\in (A_d)_{[0],0}$ be an element of the fiber of $p$ at $d$. Then the following are equivalent: 
    \begin{rome}[leftmargin=1cm]
        \item the element $a\in A_{0,[0],0}$ is a double $(\infty,n-1)$-terminal object in $A$, i.e., the canonical projection $\slice{A}{a}\to A$ is a weak equivalence in $\SegsThnsset$, 
        \item the canonical projection $\joinslice{A}{a}\to A$ is a weak equivalence in $\SegsThnsset$,
        \item the map $a\colon \repD[0]\to A$ is a weak equivalence in $\MSrightfib{D}$, 
        \item the Yoneda map $a_*\colon \slice{D}{d}\to A$ is a weak equivalence in $\MSrightfib{D}$, 
        \item the Yoneda map $a_*\colon \joinslice{D}{d}\to A$ is a weak equivalence in $\MSrightfib{D}$.
    \end{rome}
\end{prop}

\begin{proof}
    First, note that, using \cref{prop:fibrantslice} (resp.~\cref{joinslicefibrant}), the maps $\slice{A}{a} \to A\to D$ (resp.~$\joinslice{A}{a}\to A\to D$) and $A\to D$ are double $(\infty,n-1)$-right fibrations, and so $\slice{A}{a}$ (resp.~$\joinslice{A}{a}$) and $A$ are also fibrant in $\SegsThnsset$ by \cref{sourceofdblrightfib} as $D$ is so. Hence the map $\slice{A}{a} \to A$ (resp.~$\joinslice{A}{a}\to A$) is a weak equivalence in $\SegsThnsset$ if and only if it is a weak equivalence in $\injsThnsset$ if and only if it is a weak equivalence in $\MSrightfib{D}$.

    Now, consider the following commutative diagram in $\MSrightfib{D}$.
    \begin{tz}
        \node[](1) {$\repD[0]$}; 
        \node[above right of=1,xshift=1.5cm](2') {$\joinslice{A}{a}$}; 
        \node[right of=2',xshift=1cm](2'') {$\slice{A}{a}$};
        \node[below right of=1,xshift=1.5cm](2) {$\joinslice{D}{d}$}; 
        \node[right of=2,xshift=1cm](3') {$\slice{D}{d}$};
        \node[above right of=3',xshift=1.5cm](3) {$A$}; 
        \node[below of=2,xshift=1.25cm](4) {$D$}; 

        \draw[->](1) to node[above,la,xshift=5pt]{$\simeq$} node[below,la,xshift=-5pt]{$\id_d$} (2);
        \draw[->](1) to node[above,la,xshift=-5pt]{$\id_a$} node[below,la,xshift=5pt]{$\simeq$} (2');
        \draw[->](2') to node[above,la]{$\simeq$} (2'');
        \draw[->](2'') to (3);
        \draw[->](1) to node[above,la]{$a$} (3);
        \draw[->](2) to node[above,la]{$\simeq$} (3');
        \draw[->](3') to node[below,la,xshift=5pt]{$a_*$} (3);
        \draw[->,bend right=20](1) to node[left,la,yshift=-5pt]{$d$} (4); 
        \draw[->](2) to (4);
        \draw[->](3') to (4); 
        \draw[->,bend left=20] (3) to node[right,la,yshift=-5pt]{$p$} (4);
    \end{tz}
    where the maps marked with $\simeq$ are weak equivalences in $\MSrightfib{D}$ by \cref{slicehasterminalobject,cor:comparisoniswe,thm:comparisoniswe} and the fact that the functor \[ p_!\colon \MSrightfib{A}\to \MSrightfib{D} \]
    preserves weak equivalences by \cref{postcomp-pullback}. Hence the desired result follows from the $2$-out-of-$3$ property of weak equivalences.
\end{proof}

\subsection{Internal approach to \texorpdfstring{$(\infty,n)$}{(infinity,n)}-limits} \label{internallimit}

We now introduce weighted $(\infty,n)$-limits in the internal setting. Since weighted $(\infty,n)$-limits will be indexed by any objects of $\sThnsset$, we need to require that the indexing category of an $(\infty,n)$-diagram looks like an $(\infty,n)$-category. We therefore make the following definition, which coincides with the \emph{weakly constant objects} property defined in \cite[Definition 1.66]{rasekh2023Dyoneda}, as follows from \cite[Lemma 1.77]{rasekh2023Dyoneda}.

\begin{defn}
    An object $J\in \sThnsset$ is \textbf{homotopically globular} if a fibrant replacement $J\to \widehat{J}$---and hence every fibrant replacement---in $\SegsThnsset$ is such that $\widehat{J}$ is fibrant in~$\CSSsThnsset$. 
\end{defn}

\begin{ex}
    If $J$ is an object in $\pcatThn$, then $J$ is homotopically globular. This follows from the third paragraph of the proof of \cite[Proposition 2.1.5]{MRR2}.
\end{ex}

We are now ready to introduce weighted $(\infty,n)$-limits in the internal setting.

\begin{defn} \label{def:weightedlimit}
    Let $J$ be a homotopically globular object in $\sThnsset$ and $C$ be a fibrant object in $\CSSsThnsset$. Let $p\colon A\to J$ and $f\colon J\to C$ be maps in $\sThnsset$. Then an element $\ell\in C_{0,[0],0}$ is a \textbf{$p$-weighted $(\infty,n)$-limit of $f$} if there exists an element $\lambda\in ((\slice{C}{f\circ p})_\ell)_{[0],0}$ of the fiber of $\slice{C}{f\circ p}\to C$ at $\ell$ such that the pair $(\ell,\lambda)$ is a double $(\infty,n-1)$-terminal object in the cone construction $\slice{C}{f\circ p}$. 
\end{defn}

By taking $p=\id_J$, we obtain \emph{conical} $(\infty,n)$-limits, introduced in \cite[Definition 1.4.24]{MRR3}. 

\begin{defn}
   Let $J$ be a homotopically globular object in $\sThnsset$ and $C$ be a fibrant object in $\CSSsThnsset$. Let $f\colon J\to C$ be a map in $\sThnsset$. Then an element $\ell\in C_{0,[0],0}$ is a \textbf{conical $(\infty,n)$-limit of $f$} if it is an $\id_J$-weighted $(\infty,n)$-limit of $f$, i.e., there exists an element $\lambda\in ((\slice{C}{f})_\ell)_{[0],0}$ of the fiber of $\slice{C}{f}\to C$ at $\ell$ such that the pair $(\ell,\lambda)$ is a double $(\infty,n-1)$-terminal object in the cone construction $\slice{C}{f}$. 
\end{defn}

We have the following characterization of (weighted) $(\infty,n)$-limits, taking $p=\id_J$ for the conical case. 

\begin{prop} \label{prop:comparedefnlimits}
    Let $J$ be a homotopically globular object in $\sThnsset$ and $C$ be a fibrant object in $\CSSsThnsset$. Let $p\colon A\to J$ and $f\colon J\to C$ be maps in $\sThnsset$. Then an element $\ell\in C_{0,[0],0}$ is a $p$-weighted $(\infty,n)$-limit of $f$ if and only if there is an element $\lambda \in ((\joinslice{C}{f\circ p})_\ell)_{[0],0}$ of the fiber of $\joinslice{C}{f\circ p}\to C$ at $\ell$ satisfying one of the following equivalent conditions: 
    \begin{rome}[leftmargin=1cm]
        \item the canonical projection $\slice{(\slice{C}{f\circ p})}{(\ell,\lambda)}\to \slice{C}{f\circ p}$ is a weak equivalence in $\SegsThnsset$, 
        \item the canonical projection $\joinslice{(\joinslice{C}{f\circ p})}{(\ell,\lambda)}\to \joinslice{C}{f\circ p}$ is a weak equivalence in $\SegsThnsset$,
        \item the Yoneda map $\lambda_*\colon \slice{C}{\ell}\to \slice{C}{f\circ p}$ is a weak equivalence in $\MSrightfib{C}$, 
        \item the Yoneda map $\lambda_*\colon \joinslice{C}{\ell}\to \joinslice{C}{f\circ p}$ is a weak equivalence in $\MSrightfib{C}$.
    \end{rome}
\end{prop}

\begin{proof}
    This follows from \cref{cor:comparisoniswe,thm:comparisoniswe,lem:terminalequivrelative}.
\end{proof}

If one starts with an $(\infty,n)$-category presented by a fibrant object $C$ of $\pcatinj$, one would a priori need to fibrantly replace it in $\CSSsThnsset$ before computing an $(\infty,n)$-limit in $C$. The following result says that this step is not necessary. 

\begin{prop} 
\label{prop:limit in precat}
    Let $J$ be a homotopically globular object in $\sThnsset$ and let $C$ and $D$ be fibrant objects in $\SegsThnsset$ with $C_0$ and $D_0$ homotopically constant. Let $w\colon C\to D$ be a weak equivalence in $\CSSsThnsset$ (equivalently a Dwyer-Kan equivalence) and let $p\colon A\to J$ and $f\colon J\to C$ be maps in $\sThnsset$. Then, for elements $\ell\in C_{0,[0],0}$ and $\lambda \in ((\joinslice{C}{f\circ p})_\ell)_{[0],0}$ in the fiber of $\joinslice{C}{f\circ p}\to C$ at $\ell$, the following conditions are equivalent: 
    \begin{rome}[leftmargin=1cm]
        \item the pair $(\ell,\lambda)$ is a double $(\infty,n-1)$-terminal object in $\slice{C}{f\circ p}$ or equivalently in $\joinslice{C}{f\circ p}$, 
        \item the pair $(w(\ell),w(\lambda))$ is a double $(\infty,n-1)$-terminal object in $\slice{D}{w f\circ p}$ or equivalently in~$\joinslice{D}{w f\circ p}$.
    \end{rome}
    This holds in particular if $C$ is a fibrant object in $\pcatinj$ and $w\colon C\to D$ is a fibrant replacement in $\CSSsThnsset$. In this case, we call $\ell$ a \emph{$p$-weighted limit of $f$}.
\end{prop}

\begin{proof}
    This follows directly from \cref{cor:coneofprecat,prop:comparedefnlimits} using that the functor $w_!\colon \MSrightfib{C}\to \MSrightfib{D}$ preserves and reflects weak equivalences by \cref{postcomp-pullback,prop:DKeqbtwhtyconstant}.
\end{proof}

\subsection{Comparison between enriched and internal approaches to slices} \label{comparisonslices}

In order to compare the enriched and internal definition of $(\infty,n)$-limits, we first show that, for a $\Thnsset$-enriched category $\cC$ and an object $\ell\in \cC$, we have an isomorphism in $\sThnssetslice{\Nh\cC}$
\[ \Un_\cC \Hom_\cC(-,\ell)\cong \joinslice{\Nh\cC}{\ell}. \]

Recall the left adjoint $L\colon \sThnsset\to \pcatThn$ of the inclusion.

\begin{lemma} \label{rem:F0join}
    Let $m\geq 0$ and $X$ be a connected $\Thn$-space. There is a canonical isomorphism in $\pcatThn$
    \[ L(\repD[m,X]\star \repD[0]) \cong L\repD[m+1,X]. \]
\end{lemma}

\begin{proof}
    Since $LF[0,X]\cong F[0]$, this follows directly from applying the colimit-preserving functor $L\colon \sThnsset\to \pcatThn$ to the pushout from \cref{lemma:join reps} giving $F[m,X]\star F[0]$. 
\end{proof}

\begin{lemma} \label{lem:enrfunctorvsnat}
    Let $m\geq 0$ and $X$ be a connected $\Thn$-space. Let $\tau\colon F[m,X]\to \Nh\cC$ be a map in $\sThnsset$ and $\ell\in \cC$ be an object. There is a bijection between the set of $\Thnsset$-enriched natural transformations in $[\Ch LF[m,X]^{\op},\Thnsset]$ of the form
    \[ \varphi\colon \Hom_{\Ch L(F[m,X]\star \repD[0])}(-,\top)\Rightarrow \Hom_\cC(\tau^\sharp(-),\ell) \] 
    and the set of $\Thnsset$-enriched functors $F\colon \Ch L(F[m,X]\star\repD[0])\to \cC$ making the following diagram in $\Thncat$ commute,
    \begin{tz}
         \node[](1) {$\Ch L(F[m,X]\star \repD[0])$}; 
        \node[right of=1,xshift=1.7cm](2) {$\cC$}; 
        \node[above of=1](3) {$\Ch LF[m,X]\amalg [0]$}; 
        \draw[->](1) to node[below,la]{$F$} (2); 
        \draw[right hook->](3) to (1);
        \draw[->](3) to  node[right,la,yshift=5pt]{$\tau^\sharp +\ell$} (2);
    \end{tz}
    where $\tau^\sharp\colon \Ch LF[m,X]\to \cC$ denotes the adjunct of $\tau$. 
\end{lemma}

\begin{proof}
    A $\Thnsset$-enriched functor $F$ as above determines a $\Thnsset$-enriched natural transformations in $[\Ch LF[m,X]^{\op},\Thnsset]$
    \[ F_{\tau^\sharp(-),\top}\colon \Hom_{\Ch L(F[m,X]\star \repD[0])}(-,\top)\Rightarrow \Hom_\cC(\tau^\sharp(-),\ell). \]

    Conversely, a $\Thnsset$-enriched natural transformation $\varphi$ as above determines a $\Thnsset$-enriched functor 
    \[ F_\varphi\colon \Ch L(F[m,X]\star\repD[0])\to \cC \]
    given on objects by sending $0\leq i\leq m$ to $\tau^\sharp (i)$ and $\top$ to $\ell$, and given on hom $\Thn$-spaces with $0\leq i<j\leq m$ by 
    \[ (F_\varphi)_{i,j}\coloneq \tau^\sharp_{i,j}\colon \Hom_{\Ch L(F[m,X]\star \repD[0])}(i,j)\Rightarrow \Hom_\cC(\tau^\sharp(i),\tau^\sharp(j))\]
    and on hom $\Thn$-spaces with $0\leq i\leq m$ by
    \[ (F_\varphi)_{i,\top}\coloneqq \varphi_i\colon \Hom_{\Ch L(F[m,X]\star \repD[0])}(i,\top)\Rightarrow \Hom_\cC(\tau^\sharp(i),\ell). \]
    The enriched naturality of $\varphi$ implies that $F_\varphi$ is compatible with composition. 

    These assignments are clearly inverse to each other. 
\end{proof}

\begin{lemma} \label{lem:nattransfoutofStvsfunctors}
    Let $m\geq 0$ and $X$ be a connected $\Thn$-space. Let $\tau\colon F[m,X]\to \Nh\cC$ be a map in $\sThnsset$ and $\ell\in \cC$ be an object. There is a bijection between the set of $\Thnsset$-enriched natural transformations in $[\cC^{\op},\Thnsset]$ of the form
    \[ \varphi\colon \St_\cC(\tau)\Rightarrow \Hom_\cC(-,\ell) \] 
    and the set of $\Thnsset$-enriched functors $F\colon \Ch L(F[m,X]\star\repD[0])\to \cC$ making the following diagram in $\Thncat$ commute,
    \begin{tz}
         \node[](1) {$\Ch L(F[m,X]\star \repD[0])$}; 
        \node[right of=1,xshift=1.7cm](2) {$\cC$}; 
        \node[above of=1](3) {$\Ch LF[m,X]\amalg [0]$}; 
        \draw[->](1) to node[below,la]{$F$} (2); 
        \draw[right hook->](3) to (1);
        \draw[->](3) to  node[right,la,yshift=5pt]{$\tau^\sharp +\ell$} (2);
    \end{tz}
    where $\tau^\sharp\colon \Ch LF[m,X]\to \cC$ denotes the adjunct of $\tau$. 
\end{lemma}

\begin{proof} 
    By \cref{lem:naturalityofSt}, a $\Thnsset$-enriched natural transformation $\varphi$ as above corresponds to a $\Thnsset$-enriched natural transformation in $[\cC^{\op},\Thnsset]$
    \[ \varphi\colon (\tau^\sharp)_!\St_{LF[m,X]}(\id_{F[m,X]})\Rightarrow \Hom_\cC (-,\ell). \]
    By the adjunction $(\tau^\sharp)_!\dashv (\tau^\sharp)^*$, this corresponds to a $\Thnsset$-enriched natural transformation in $[\Ch LF[m,X]^{\op},\Thnsset]$
    
    \[ \varphi\colon \St_{LF[m,X]}(\id_{F[m,X]})\Rightarrow (\tau^\sharp)^*\Hom_\cC (-,\ell). \] 
    By definition of straightening and \cref{rem:F0join}, this corresponds to a $\Thnsset$-enriched natural transformation in $[\Ch LF[m,X]^{\op},\Thnsset]$
   \[ \varphi\colon \Hom_{\Ch L(F[m,X]\star\repD[0])}(-,\top)\Rightarrow \Hom_\cC(\tau^\sharp(-),\ell). \] 
   By \cref{lem:enrfunctorvsnat}, this corresponds to a $\Thnsset$-enriched functor $F$ as desired.
\end{proof}

\begin{prop} \label{Unvsjoinslice}
   Let $\cC$ be a $\Thnsset$-enriched category and $\ell\in \cC$ be an object. There is an isomorphism in $\sThnssetslice{\Nh\cC}$
    \[ \Un_\cC \Hom_\cC(-,\ell)\cong \joinslice{\Nh\cC}{\ell}. \]
\end{prop}

\begin{proof}   
    Let $m\geq 0$ and $X$ be a connected $\Thn$-space. By the adjunction $\St_\cC\dashv \Un_\cC$, a simplex 
    \begin{tz}
    \node[](1) {$F[m,X]$}; 
    \node[below right of=1,xshift=.8cm](2) {$\Nh\cC$}; 
    \node[above right of=2,xshift=.8cm](3) {$\Un_\cC \Hom_\cC(-,\ell)$}; 

    \draw[->] (1) to node[above,la]{$\sigma$} (3); 
    \draw[->] (1) to node[left,la,yshift=-5pt]{$\tau$} (2);
    \draw[->] (3) to (2);
\end{tz}
    corresponds to a $\Thnsset$-enriched natural transformation in $[\cC^{\op},\Thnsset]$
    \[ \sigma^\sharp\colon \St_{\cC}(\tau)\Rightarrow \Hom_\cC (-,\ell). \] 
    By \cref{lem:nattransfoutofStvsfunctors}, this corresponds to a commutative diagram in $\Thncat$ as below left, which by the adjunction $\Ch L\dashv \Nh$ corresponds to a commutative diagram in $\sThnsset$ as below right.
   \begin{tz}
       \node[](1) {$\Ch L(F[m,X]\star \repD[0])$}; 
        \node[right of=1,xshift=1.7cm](2) {$\cC$}; 
        \node[above of=1](3) {$\Ch LF[m,X]\amalg [0]$}; 
        \draw[->](1) to node[below,la]{$\sigma^\sharp$} (2); 
        \draw[right hook->](3) to (1);
        \draw[->](3) to  node[right,la,yshift=5pt]{$\tau^\sharp +\ell$} (2); 
        
        \node[right of=1,xshift=5.2cm](1) {$F[m,X]\star \repD[0]$}; 
        \node[right of=1,xshift=1.5cm](2) {$\Nh \cC$}; 
        \node[above of=1](3) {$F[m,X]\amalg \repD[0]$}; 
        \draw[->](1) to node[below,la]{$\sigma$} (2); 
        \draw[right hook->](3) to (1);
        \draw[->](3) to  node[right,la,yshift=5pt]{$\tau +\ell$} (2);
    \end{tz}
    The latter is, by definition of the neat slice construction, a simplex
    \begin{tz}
    \node[](1) {$F[m,X]$}; 
    \node[below right of=1,xshift=.4cm](2) {$\Nh\cC$}; 
    \node[above right of=2,xshift=.4cm](3) {$\joinslice{\Nh\cC}{\ell}$}; 

    \draw[->] (1) to node[above,la]{$\sigma$} (3); 
    \draw[->] (1) to node[left,la,yshift=-5pt]{$\tau$} (2);
    \draw[->] (3) to (2);
\end{tz}
    This proves the desired isomorphism.
\end{proof}

\subsection{Comparison between enriched and internal approaches to \texorpdfstring{$(\infty,n)$}{(infinity,n)}-limits} \label{comparisonlimit}

In this section, we prove the main theorem giving the equivalence between the enriched and internal definitions of $(\infty,n)$-limits. 

\begin{data} \label{setting:limit}
    We fix the following data:
\begin{itemize}[leftmargin=0.6cm] 
    \item an object $J\in\pcatThn$, 
    \item a $\Thnsset$-enriched category $\cC$,
    \item a map $p\colon A\to J$ in $\sThnsset$ with straightening $\St_J(p)\colon \Ch J\to \Thnsset$,
    \item a $\Thnsset$-enriched functor $F\colon \Ch J\to \cC$ with adjunct $F^\flat\colon J\to \Nh\cC$ in $\sThnsset$.
\end{itemize}
\end{data}

We first show that the cone constructions representing the $(\infty,n)$-limits in each setting agree at an object of $\cC$. 

\begin{prop} \label{lem:webtwfibers}
   In \cref{setting:limit}, for every object $c\in \cC$, there is a canonical isomorphism in~$\Thnsset$
    \[ \Hom_{[\Ch J,\Thnsset]}(\St_J(p),\Hom_\cC(c,F-))\cong (\joinslice{\Nh\cC}{F^\flat\circ p})_c,  \]
    where $(\joinslice{\Nh\cC}{F^\flat\circ p})_c$ denotes the fiber of $\joinslice{\Nh\cC}{F^\flat\circ p}\to \Nh\cC$ at $c$. 
\end{prop}

\begin{proof}
    We have the following natural isomorphisms in $\Thnsset$
    \begin{align*}
    \Hom_{[\Ch J,\Thnsset]}&(\St_J(p),\Hom_\cC(c,F-)) \\
    & \cong \Hom_{[\cC,\Thnsset]}(F_!\St_J(p),\Hom_\cC(c,-)) & F_!\dashv F^* \\
    & \cong \Hom_{[\cC,\Thnsset]}(\St_\cC((F^\flat)_! p),\Hom_\cC(c,-)) & \text{\cref{lem:naturalityofSt}} \\
    & \cong \Hom_{\sThnssetslice{\Nh\cC}}(A\xrightarrow{F^\flat\circ p} \Nh\cC,\Un_\cC \Hom_\cC(c,-)\to \Nh\cC) & \St_\cC\dashv \Nh_\cC \\
    & \cong \Hom_{\sThnssetslice{\Nh\cC}}(A\xrightarrow{F^\flat\circ p} \Nh\cC,\joinsliceunder{\Nh\cC}{c}\to \Nh\cC) &  \text{\cref{Unvsjoinslice}}  \\
    & \cong (\joinslice{\Nh\cC}{F^\flat \circ p})_c & \text{\cref{joinconeonfibers}} 
    \end{align*}
    as desired.
\end{proof}

In particular, this gives a way of translating the $(\infty,n)$-limit cones from the enriched to the internal setting, and conversely. 

\begin{notation} \label{notn:enrvsfibcones}
    In \cref{setting:limit}, let $\ell\in \cC$ be an object and $\lambda\colon \St_J(p)\Rightarrow \Hom_\cC(\ell,F-)$ be a $\Thnsset$-enriched natural transformation  in $[\Ch J,\Thnsset]$. We denote by $\lambda^\flat \in ((\joinslice{\Nh\cC}{F^\flat \circ p})_\ell)_{[0],0}$ the image of $\lambda$ under the isomorphism from \cref{lem:webtwfibers}
    \[ \Hom_{[\Ch J,\Thnsset]}(\St_J(p),\Hom_\cC(\ell,F-))\cong (\joinslice{\Nh\cC}{F^\flat \circ p})_\ell. \]
\end{notation}

\begin{rmk}
    We believe that there is a map in $\sThnssetslice{\Nh\cC}$
   \[
    \Un_\cC \Hom_{[\Ch J,\Thnsset]}(\St_J(p), \Hom_\cC(-,F))\to \joinslice{\Nh\cC}{F^\flat\circ p}
    \]
that acts as \cref{lem:webtwfibers} on fibers, but it is difficult to build and so we bypass this issue by constructing explicit Yoneda maps that we can handle on fibers.
\end{rmk}

We state the following result whose proof is delayed to \cref{sec:Yonedamap}.

\begin{prop} \label{prop:existYoneda} 
In \cref{setting:limit}, let $\ell\in \cC$ be an object, $\lambda\colon \St_J(p)\Rightarrow \Hom_\cC(\ell,F-)$ be a $\Thnsset$-enriched natural transformation in $[\Ch J,\Thnsset]$, and $\lambda^\flat \in ((\joinslice{\Nh\cC}{F^\flat \circ p})_\ell)_{[0],0}$ be as in \cref{notn:enrvsfibcones}. There is a map in $\sThnssetslice{\Nh\cC}$
\[ \lambda^\flat_*\colon \joinslice{\Nh\cC}{\ell}\to \joinslice{\Nh\cC}{F^\flat\circ p} \]
sending $\id_\ell$ to $\lambda^\flat$ and
making the following diagrams commute, for every object $c\in \cC$,
\begin{tz}
        \node[](1) {$\Hom_\cC(c,\ell)$}; 
        \node[below of=1](2) {$\Hom_{\Nh\cC}(c,\ell)$}; 
        \node[right of=1,xshift=4.3cm,yshift=-3pt](3) {$\Hom_{[\Ch J,\Thnsset]}(\St_J(p),\Hom_\cC(c,F-))$}; 
        \node[below of=3,yshift=3pt](4) {$(\joinslice{\Nh\cC}{F^\flat\circ p})_c$}; 

        \draw[d] (1) to (2); 
        \draw[->] (1) to node[above,la]{$\lambda^*_c$} ($(3.west)+(0,3pt)$); 
        \draw[->] (2) to node[below,la]{$(\lambda^\flat_*)_c$} (4); 
        \draw[->] (3) to node[right,la]{$\cong$} (4); 
    \end{tz}
    where the right-hand map is the isomorphism from \cref{lem:webtwfibers}.
\end{prop}

As a consequence, we obtain the following.

\begin{cor} \label{fiberwiseeq}
    In \cref{setting:limit}, let $\ell\in \cC$ be an object, $\lambda\colon \St_J(p)\Rightarrow \Hom_\cC(\ell,F-)$ be a $\Thnsset$-enriched natural transformation in $[\Ch J,\Thnsset]$, and $\lambda^\flat \in ((\joinslice{\Nh\cC}{F^\flat \circ p})_\ell)_{[0],0}$ be as in \cref{notn:enrvsfibcones}. Then, for every object $c\in \cC$,
    the map induced by precomposing with $\lambda$
    \[ \lambda^*_c\colon \Hom_\cC(c,\ell)\to \Hom_{[\Ch J,\Thnsset]}(\St_J(p),\Hom_\cC(c,F-)) \]
    is a weak equivalence in $\MSThnsset$ if and only if the map induced by $\lambda^\flat_*\colon \joinslice{\Nh\cC}{\ell}\to \joinslice{\Nh\cC}{F^\flat\circ p}$ from \cref{prop:existYoneda} on fibers
    \[ (\lambda^\flat_*)_c\colon \Hom_{\Nh\cC}(c,\ell)\to (\joinslice{\Nh\cC}{F^\flat \circ p})_c \]
    is a weak equivalence in $\MSThnsset$. 
\end{cor}

\begin{proof}
    This follows by the $2$-out-of-$3$ property of the weak equivalences in $\MSThnsset$ applied to the commutative diagram from \cref{prop:existYoneda}.
\end{proof}

We are now ready to state and prove the main theorem.

\begin{thm}
\label{EquivalenceOfLimitsWeighted}
    Let $J$ be an object in $\pcatThn$ and $\cC$ be a fibrant $\MSThnsset$-enriched category. Let $p\colon A\to J$ be a map in $\sThnsset$ with straightening $\St_J(p)\colon \Ch J\to \Thnsset$ and $F\colon \Ch J\to \cC$ be a $\Thnsset$-enriched functor with adjunct $F^\flat\colon J\to \Nh\cC$ in $\sThnsset$. Let $\ell\in \cC$ be an object, $\lambda\colon \St_J(p)\Rightarrow \Hom_\cC(\ell,F-)$ be a $\Thnsset$-enriched natural transformation in $[\Ch J,\Thnsset]$, and $\lambda^\flat \in ((\joinslice{\Nh\cC}{F^\flat \circ p})_\ell)_{[0],0}$ be as in \cref{notn:enrvsfibcones}. Then the following are equivalent: 
    \begin{rome}[leftmargin=1cm]
    \item the pair $(\ell,\lambda)$ is the $\St_J(p)$-weighted $(\infty,n)$-limit of $F$, 
    \item the pair $(\ell,\lambda^\flat)$ is the $p$-weighted $(\infty,n)$-limit of $F^\flat$.
    \end{rome}
\end{thm}

\begin{proof}
    We have that the pair $(\ell,\lambda)$ is a $\St_J(p)$-weighted $(\infty,n)$-limit of $F$ if and only if, by definition, the induced map 
    \[ \lambda^*\colon \Hom_\cC(-,\ell)\to \Hom_{[\Ch J,\Thnsset]}(\St_J(p),\Hom_\cC(-,F)) \]
    is a weak equivalence in $[\cC^{\op},\MSThnsset]_\proj$ if and only if, by definition of the weak equivalences in $[\cC^{\op},\MSThnsset]_\proj$, for every object $c\in \cC$, the induced map 
    \[ \lambda^*_c\colon \Hom_\cC(c,\ell)\to \Hom_{[\Ch J,\Thnsset]}(\St_J(p),\Hom_\cC(c,F-)) \]
    is a weak equivalence in $\MSThnsset$ if and only if, by \cref{fiberwiseeq}, for every object $c\in \cC$, the~map 
    \[ (\lambda^\flat_*)_c\colon \Hom_{\Nh\cC}(c,\ell)\to (\joinslice{\Nh\cC}{(\ell,\lambda^\flat)})_c \]
    is a weak equivalence in $\MSThnsset$ if and only if, by \cref{webtwdoubleright}, the induced map 
    \[ \lambda^\flat_*\colon \joinslice{\Nh\cC}{\ell} \to \joinslice{\Nh\cC}{F^\flat\circ p} \]
    is a weak equivalence in $\MSrightfib{\Nh\cC}$ if and only if, by \cref{lem:terminalequivrelative,prop:limit in precat}, the pair $(\ell,\lambda^\flat)$ is a $p$-weighted $(\infty,n)$-limit of $F^\flat$.
\end{proof}

By applying the above result to $p=\id_J$, we get the following: 

\begin{cor}
\label{EquivalenceOfLimitsConical}
    Let $J$ be an object in $\pcatThn$, $\cC$ be a fibrant $\MSThnsset$-enriched category, and $F\colon \Ch J\to \cC$ be a $\Thnsset$-enriched functor with adjunct $F^\flat\colon J\to \Nh\cC$ in $\sThnsset$. Let $\ell\in \cC$ be an object, $\lambda\colon \St_J(\id_J)\Rightarrow \Hom_\cC(\ell,F-)$ be a $\Thnsset$-enriched natural transformation in $[\Ch J,\Thnsset]$, and $\lambda^\flat \in ((\joinslice{\Nh\cC}{F^\flat})_\ell)_{[0],0}$ be as in \cref{notn:enrvsfibcones}. Then the following are equivalent: 
    \begin{rome}[leftmargin=1cm]
    \item the pair $(\ell,\lambda)$ is the $\St_J(\id_J)$-weighted $(\infty,n)$-limit of $F$ (=conical $(\infty,n)$-limit), 
    \item the pair $(\ell,\lambda^\flat)$ is the conical $(\infty,n)$-limit of $F^\flat$. 
    \end{rome}
\end{cor}

\section{Key examples of \texorpdfstring{$(\infty,n)$}{(infinity,n)}-(co)complete \texorpdfstring{$(\infty,n)$}{(infinity,n)}-categories}

\label{SectionApplications}

In this section, we show that the following $(\infty,n)$-categories have all weighted $(\infty,n)$-(co)limits:
\begin{itemize}[leftmargin=0.6cm]
    \item The $(\infty,n)$-category of $(\infty,n-1)$-categories.
    \item The $(\infty,n)$-category of $(\infty,n)$-presheaves, i.e., $(\infty,n-1)$-category valued presheaves.
    \item Localizations of the above examples.
    \item The $(\infty,n)$-category of $(\infty,n)$-sheaves, i.e., $(\infty,n-1)$-category valued sheaves, on an $(\infty,1)$-site.
\end{itemize}
For this, we first study in \cref{subsec:cocompletenessmodel} the $(\infty,n)$-(co)completeness of $(\infty,n)$-categories arising as homotopy coherent nerves of certain $\MSThnsset$-enriched model categories. Then, in \cref{subsec:cocompleteness(pre)sheaves}, we prove that the above examples are (co)complete in the $(\infty,n)$-categorical sense.

Throughout this section, we refer to fibrant objects in $\pcatinj$ as \emph{$(\infty,n)$-categories} and study their  $(\infty,n)$-(co)completeness. In light of \cref{prop:limit in precat}, the same $(\infty,n)$-(co)\-complete\-ness results also hold for the corresponding fibrant replacement in $\CSSsThnsset$.

\subsection{(Co)completeness of $(\infty,n)$-categories arising as homotopy coherent nerves} \label{subsec:cocompletenessmodel}

We start by studying the (co)completeness of $(\infty,n)$-categories arising as homotopy coherent nerves of $\MSThnsset$-enriched model categories.

\begin{notation}
    Let $\cM$ be a $\MSThnsset$-enriched model category. We denote by $\cM^{\mathrm{cf}}$ its full $\Thnsset$-enriched subcategory of bifibrant objects. Note that this is a fibrant $\MSThnsset$-enriched category.
\end{notation}

In the next result, we call \emph{strict} the weighted $\Thnsset$-enriched limits where the universal property is given by an isomorphism in $\Thnsset$, rather than a weak equivalence in $\MSThnsset$. 

\begin{prop} \label{lem:strictweightedlimit}
    Let $\cM$ be a $\MSThnsset$-enriched model category, and $F\colon \cJ\to \cM^\mathrm{cf}$ and $W\colon \cJ\to \Thnsset$ be $\Thnsset$-enriched functors such that $W$ is cofibrant in $[\cJ,\MSThnsset]_\mathrm{proj}$. Then the strict $W$-weighted limit of $F$ exists and is fibrant in $\cM$.
\end{prop}

\begin{proof}
    A proof can be easily adapted from \cite[Proposition 4.1.5(i)]{RVLimits}. 
\end{proof}

\begin{prop} \label{Mcfiscomplete}
    Let $\cM$ be a $\MSThnsset$-enriched model category, and let $F\colon \cJ\to \cM^\mathrm{cf}$ and $W\colon \cJ\to \Thnsset$ be $\Thnsset$-enriched functors such that $W$ is cofibrant in $[\cJ,\MSThnsset]_\mathrm{proj}$. Then the $W$-weighted $(\infty,n)$-limit of $F$ exists in $\cM^{\mathrm{cf}}$. 
\end{prop}

\begin{proof}
Adapting the argument from the first part of the proof of 
\cite[Proposition 6.2.3]{RVLimits}, we see that, if $\ell$ is a strict $W$-weighted limit of $F$, then the object $\widehat{\ell}$ given by a cofibrant replacement $\widehat{\ell}\to \ell$ in $\cM$ gives the desired $W$-weighted $(\infty,n)$-limit of $F$ in $\cM^\mathrm{cf}$.
\end{proof}

\begin{defn}
    Let $C$ be an $(\infty,n)$-category. We say that $C$ is \textbf{$(\infty,n)$-complete} (resp.~\textbf{$(\infty,n)$-cocomplete}) if it has all weighted $(\infty,n)$-limits (resp.~weighted $(\infty,n)$-colimits). 
\end{defn}

\begin{thm} \label{thm:htpynervearecomplete}
    Let $\cM$ be a $\MSThnsset$-enriched model category. Then the homotopy coherent nerve $\Nh(\cM^\mathrm{cf})$ is $(\infty,n)$-complete and $(\infty,n)$-cocomplete. 
\end{thm}

\begin{proof}
    This is obtained by combining \cref{EquivalenceOfLimitsWeighted,Mcfiscomplete}. 
\end{proof}

We can deduce (co)completeness of the $(\infty,n)$-categories of functors valued in $\Nh(\cM^{\mathrm{cf}})$.

\begin{thm} \label{catoffunctorsarecomplete}
    Let $\cM$ be a combinatorial $\MSThnsset$-enriched model category and $J$ be an object in $\pcatThn$. Let $\mathrm{Fun}(J,\Nh(\cM^\mathrm{cf}))$ denote the $(\infty,n)$-category of functors $J\to \Nh(\cM^\mathrm{cf})$. Then the $(\infty,n)$-category $\mathrm{Fun}(J,\Nh(\cM^\mathrm{cf}))$ is $(\infty,n)$-complete and $(\infty,n)$-cocomplete.
\end{thm}

\begin{proof}
    For every object $K\in \pcatThn$, we have the following natural isomorphisms: 
\begin{align*}
\Ho(\Pcat&(\MSThnsset)_\mathrm{inj})(K, \mathrm{Fun}(J,\Nh(\cM^\mathrm{cf}))) & \\
    & \cong \Ho(\Pcat(\MSThnsset)_\mathrm{inj})(K\times J, \Nh(\cM^\mathrm{cf}))  & 
 \Ho(-\times J)\dashv \Ho\mathrm{Fun}(J,-)\\
    & \cong \Ho(\MSThncat)(\Ch(K\times J),\cM^{\mathrm{cf}}) & \Ch\dashv \Nh \text{ Quillen pair} \\
    & \cong \Ho(\MSThncat)(\Ch K\times \Ch J,\cM^{\mathrm{cf}}) & \Ho(\Ch) \text{ equivalence so pres.~product} \\
    & \cong \Ho(\MSThncat)(\Ch K,([\Ch J,\cM]_\mathrm{proj})^\mathrm{cf}) & \text{\cite[Proposition~A.3.4.13]{LurieHTT}} \\
    & \cong \Ho(\pcatinj)(K,\Nh(([\Ch J,\cM]_\mathrm{proj})^\mathrm{cf})) & \Ch\dashv \Nh \text{ Quillen pair}
\end{align*}
Then, by the Yoneda lemma, the $(\infty,n)$-categories $\Nh(([\Ch J,\cM]_\mathrm{proj})^\mathrm{cf})$ and $\mathrm{Fun}(J,\Nh(\cM^\mathrm{cf}))$ are equivalent. Hence we conclude by \cref{prop:limit in precat,thm:htpynervearecomplete}. 
\end{proof}

We now study $(\infty,n)$-localizations of $\Nh(\cM^{\mathrm{cf}})$. For this, we first recall the definition of an enriched Bousfield localization from \cite[Definition 4.42]{barwick2010bousfield}.

\begin{defn}
    Let $\cM$ be a $\MSThnsset$-enriched model category and $S$ be a set of $1$-morphisms in $\cM^\mathrm{cf}$. Then an object $c\in \cM$ is \textbf{$S$-local} if it is fibrant in $\cM$ and, for every $1$-morphism $f\colon x\to y$ in $S$, the induced map 
    \[ f^*\colon \Hom_\cM(y,c)\to \Hom_\cM(x,c) \]
    is a weak equivalence in $\MSThnsset$.
\end{defn}

The following appears as \cite[Theorem 4.46]{barwick2010bousfield}.

\begin{thm} \label{Barwick}
    Let $\cM$ be a left proper and combinatorial $\MSThnsset$-enriched model category and $S$ be a set of $1$-morphisms in $\cM^{\mathrm{cf}}$. Then there is a model structure $L_S\cM$ on $\cM$, called the \emph{enriched left Bousfield localization of $\cM$ at $S$}, whose cofibrations are those of $\cM$ and fibrant objects are the $S$-local objects. 
\end{thm}

We now introduce $(\infty,n)$-localizations of $(\infty,n)$-categories. 

\begin{notation} \label{notn:Mapfc}
Let $C$ be an $(\infty,n)$-category and $f\colon x\to y$ be a $1$-morphism in $C$, i.e., an element $f\in C_{1,[0],0}$. Consider the following (homotopy) pullbacks in $\MSThnsset$
\begin{tz}
        \node[](1) {$\Map_C(f,c)$}; 
        \node[right of=1,xshift=1.2cm](2) {$C_2$}; 
        \node[below of=1](3) {$\Delta[0]$}; 
        \node[below of=2](4) {$C_1\times C_0$}; 

        \draw[->] (1) to  (2); 
        \draw[->>] (1) to (3); 
        \draw[->>] (2) to node[right,la]{$\langle 0,1\rangle^*\times \langle 2\rangle^*$} (4); 
        \draw[->] (3) to node[below,la]{$(f,c)$} (4);
        \pullback{1};

        \node[right of=2,xshift=2cm](1) {$\Map_C(x,c)$}; 
        \node[right of=1,xshift=1.2cm](2) {$C_1$}; 
        \node[below of=1](3) {$\Delta[0]$}; 
        \node[below of=2](4) {$C_0\times C_0$}; 

        \draw[->] (1) to  (2); 
        \draw[->>] (1) to (3); 
        \draw[->>] (2) to node[right,la]{$\langle 0\rangle^*\times \langle 1\rangle^*$} (4); 
        \draw[->] (3) to node[below,la]{$(x,c)$} (4);
        \pullback{1};
    \end{tz}
    Then the maps $\langle 0,2\rangle^*\colon C_2\to C_1$ and $\langle 0\rangle^*\times \id_{C_0}\colon C_1\times C_0\to C_0\times C_0$ (resp.~$\langle 1,2\rangle^*\colon C_2\to C_1$ and $\langle 1\rangle^*\times \id_{C_0}\colon C_1\times C_0\to C_0\times C_0$) induce a map in $\Thnsset$
    \[ \Map_C(f,c)\to \Map_C(x,c) \quad \text{(resp. }\Map_C(f,c)\to \Map_C(y,c)\text{)}. \]
    \end{notation}

    \begin{lemma} \label{rmk:Mapfcvsyc}
        Let $C$ be an $(\infty,n)$-category and $f\colon x\to y$ be a $1$-morphism in $C$, i.e., an element $f\in C_{1,[0],0}$. Then the induced map from \cref{notn:Mapfc}
        \[ \Map_C(f,c)\to \Map_C(y,c) \]
        is a weak equivalence in $\MSThnsset$.
    \end{lemma}

    \begin{proof}
    First consider the following commutative diagram in $\MSThnsset$
    \begin{tz}
        \node[](1) {$\Map_C(f,c)$}; 
        \node[right of=1,xshift=2cm](2) {$C_1\times_{C_0} C_1$};
        \node[right of=2,xshift=1.4cm](2') {$C_1$};
        \node[below of=1](3) {$\Delta[0]$}; 
        \node[below of=2](4) {$C_1\times C_0$}; 
        \node[below of=2'](4') {$C_0\times C_0$}; 

        \draw[->] (1) to  (2); 
        \draw[->>] (1) to (3); 
        \draw[->>] (2) to node[left,la]{$\id_{C_1}\times \langle 1\rangle^*$} (4);
        \draw[->>] (2') to node[right,la]{$\langle 0\rangle^*\times \langle 1\rangle^*$} (4');
        \draw[->] (3) to node[below,la]{$(f,c)$} (4);
        \draw[->](2) to node[above,la]{$\mathrm{pr}_2$} (2'); 
        \draw[->](4) to node[below,la]{$\langle 1\rangle^*\times \id_{C_0}$} (4');
        \pullback{2};
    \end{tz}
    Since the right-hand and outer squares are pullbacks, then by the cancellation property of pullbacks, so is the left-hand square. Moreover, we have a commutative diagram in $\MSThnsset$,
    \begin{tz}
        \node[](1) {$C_2$}; 
        \node[below of=1,yshift=.3cm,xshift=1.5cm](2) {$C_1\times C_0$}; 
        \node[above of=2,yshift=-.3cm,xshift=1.5cm](3) {$C_1\times_{C_0} C_1$}; 
        \draw[->>](1) to node[left,la,yshift=-5pt]{$\langle 0,1\rangle^*\times \langle 2\rangle^*$} (2);
        \draw[->>](3) to node[right,la,yshift=-5pt]{$\id_{C_1}\times \langle 1\rangle^*$} (2);
        \draw[->](1) to node[above,la]{$\simeq$} (3);
    \end{tz}
    where the top map is the Segal map. Hence the induced map between fibers at $(f,c)$
    \[ \Map_C(f,c)\to \Map_C(y,c) \]
    is also a weak equivalence in $\MSThnsset$, as desired.
    \end{proof}

\begin{defn}
    Let $C$ be an $(\infty,n)$-category and $S$ a set of $1$-morphisms in $C$, i.e., $S\subseteq C_{1,[0],0}$. Then an object $c\in C$ is \textbf{$S$-local} if, for every $1$-morphism $f\colon x \to y$ in $S$, the induced map from \cref{notn:Mapfc} \[ \Map_C(f,c)\to \Map_C(x,c) \]
    is a weak equivalence in $\MSThnsset$.
\end{defn}

In particular, the two definitions of local objects coincide for $\cM$ and $\Nh(\cM^\mathrm{cf})$.

\begin{lemma} \label{lem:compareSlocal}
    Let $\cM$ be a $\MSThnsset$-enriched model category and $f\colon x\to y$ be a $1$-morphism in $\cM^\mathrm{cf}$ (and equivalently $f\in \Nh(\cM^\mathrm{cf})_{1,[0],0}$), and let $c\in \cM^{\mathrm{cf}}$ be an object. Then the map 
    \[ f^*\colon \Hom_\cM(y,c)\to \Hom_\cM(x,c) \]
    is a weak equivalence in $\MSThnsset$ if and only if the map 
    \[ \Map_{\Nh(\cM^\mathrm{cf})}(f,c)\to \Map_{\Nh(\cM^\mathrm{cf})}(x,c) \]
    is a weak equivalence in $\MSThnsset$.
\end{lemma}

\begin{proof}
    Using \cite[Lemma 3.5.1]{MRR1}, for all objects $x,c\in \cM^\mathrm{cf}$, there are isomorphisms in $\Thnsset$
    \[ \Hom_\cM(x,c)\cong \Hom_{\cM^\mathrm{cf}}(x,c)\cong \Map_{\Nh(\cM^\mathrm{cf})}(x,c). \]
    Hence the map $f^*\colon \Hom_\cM(y,c)\to \Hom_\cM(x,c)$ induces an isomorphic map in $\Thnsset$
    \[ f^*\colon \Map_{\Nh(\cM^\mathrm{cf})}(y,c)\to \Map_{\Nh(\cM^\mathrm{cf})}(x,c) \]
    and one is a weak equivalence in $\MSThnsset$ if and only if the other is. Next, note that we have a diagram in $\MSThnsset$ that commutes up to homotopy
    \begin{tz}
        \node[](1) {$\Map_{\Nh(\cM^\mathrm{cf})}(f,c)$}; 
        \node[below of=1,yshift=.3cm,xshift=-2.3cm](2) {$\Map_{\Nh(\cM^\mathrm{cf})}(y,c)$}; 
        \node[below of=1,yshift=.3cm,xshift=2.3cm](3) {$\Map_{\Nh(\cM^\mathrm{cf})}(x,c)$}; 

        \draw[->](1) to node[above,la]{$\simeq$} (2);
        \draw[->](1) to (3); 
        \draw[->](2) to node[below,la]{$f^*$} (3);
    \end{tz}
    where the left-hand map is a weak equivalence by \cref{rmk:Mapfcvsyc}. Hence, we conclude by $2$-out-of-$3$ that the map $f^*$ is a weak equivalence in $\MSThnsset$ if and only if the map \[ \Map_{\Nh(\cM^\mathrm{cf})}(f,c)\to \Map_{\Nh(\cM^\mathrm{cf})}(x,c) \] is a weak equivalence in $\MSThnsset$.    
\end{proof}

This allows us to deduce (co)completeness of the $(\infty,n)$-localizations of $\Nh(\cM^\mathrm{cf})$. 

\begin{thm} \label{catoflocaliscomplete}
    Let $\cM$ be a left proper and combinatorial $\MSThnsset$-enriched model category and let $S$ be a set of $1$-morphisms in $\cM^\mathrm{cf}$ (and equivalently $S\subseteq \Nh(\cM^\mathrm{cf})_{1,[0],0}$).
    Let $\Nh(\cM^\mathrm{cf})^{S\text{-}\mathrm{loc}}$ denote the full $(\infty,n)$-subcategory of $\Nh(\cM^\mathrm{cf})$ consisting of the $S$-local objects. Then the $(\infty,n)$-category $\Nh(\cM^\mathrm{cf})^{S\text{-}\mathrm{loc}}$ is $(\infty,n)$-complete and $(\infty,n)$-cocomplete.
\end{thm}

\begin{proof}
    Let $L_S\cM$ denote the enriched left Bousfield localization of $\cM$ at $S$ from \cref{Barwick}. By \cref{lem:compareSlocal}, the objects of the homotopy coherent nerve $\Nh((L_S\cM)^\mathrm{cf})$ and the objects of $\Nh(\cM^\mathrm{cf})^{S\text{-}\mathrm{loc}}$ are given by objects in $\Nh((\cM)^\mathrm{cf})$ which satisfy the same universal property, meaning $\Nh((L_S\cM)^\mathrm{cf})$ and $\Nh(\cM^\mathrm{cf})^{S\text{-}\mathrm{loc}}$ are equivalent $(\infty,n)$-categories. Hence we conclude by \cref{prop:limit in precat,thm:htpynervearecomplete}.
\end{proof}

\subsection{(Co)completeness of \texorpdfstring{$(\infty,n)$}{(infinity,n)}-(pre)sheaves} \label{subsec:cocompleteness(pre)sheaves}

We now apply the results of the preceding section to the $\MSThnsset$-enriched model structure $\MSThnsset$ for $(\infty,n-1)$-categories.

\begin{notation}
   We denote by $\cat_{(\infty,n-1)}$ the $(\infty,n)$-category of $(\infty,n-1)$-categories. Recall that it can be modeled by the $(\infty,n)$-category $\Nh((\MSThnsset)^\mathrm{cf})$. 
\end{notation}

By taking $\cM=\MSThnsset$ in \cref{thm:htpynervearecomplete}, we then obtain the following result.

\begin{cor}
    The $(\infty,n)$-category $\cat_{(\infty,n-1)}$ of $(\infty,n-1)$-categories is $(\infty,n)$-complete and $(\infty,n)$-cocomplete.
\end{cor}

We now consider the $(\infty,n)$-category $\mathrm{Fun}(J^{\op},\cat_{(\infty,n-1)})$ of \emph{$(\infty,n)$-presheaves} over an $(\infty,n)$-category $J$. By taking $\cM=\MSThnsset$ in \cref{catoffunctorsarecomplete}, we obtain the following result. 

\begin{cor}
    The $(\infty,n)$-category $\mathrm{Fun}(J^{\op},\cat_{(\infty,n-1)})$ of $(\infty,n)$-presheaves over an $(\infty,n)$-category $J$ is $(\infty,n)$-complete and $(\infty,n)$-cocomplete.
\end{cor}

We now consider full $(\infty,n)$-subcategories of local objects in the $(\infty,n)$-category of $(\infty,n)$-presheaves. By taking $\cM=[\Ch J,\MSThnsset]_\mathrm{proj}$ in \cref{catoflocaliscomplete}, for some $J\in \pcatThn$, we obtain the following result.

\begin{cor} \label{cor:localization}
    Every full $(\infty,n)$-subcategory of local objects of an $(\infty,n)$-category of $(\infty,n)$-presheaves over an $(\infty,n)$-category is $(\infty,n)$-complete and $(\infty,n)$-cocomplete.
\end{cor}

We can in particular apply this result to $(\infty,n)$-sheaves. Recall from \cite[Definition 6.2.2.1 and Remark 6.2.2.3]{LurieHTT} that an \emph{$(\infty,1)$-site $(C,T)$} is given by an $(\infty,1)$-category $C$ and a site $T$ on the homotopy category of~$C$. Moreover, by \cite[Definition 6.2.2.6]{LurieHTT}, the \emph{$(\infty,1)$-category $\cS hv_{(\infty,1)}(C,T)$ of $(\infty,1)$-sheaves over $(C,T)$} is defined as the full $(\infty,1)$-subcategory of $T$-local $(\infty,1)$-presheaves.

\begin{defn}
    Let $(C,T)$ be an $(\infty,1)$-site. The \textbf{$(\infty,n)$-category of sheaves} is the full $(\infty,n)$-subcategory $\cS hv_{(\infty,n)}(C,T)$ of $\mathrm{Fun}(C^{\op},\cat_{(\infty,n-1)})$ consisting of the $T$-local $(\infty,n)$-presheaves.
\end{defn}

\begin{cor}
    Let $(C,T)$ be an $(\infty,1)$-site. The $(\infty,n)$-category $\cS hv_{(\infty,n)}(C,T)$ of $(\infty,n)$-sheaves over $(C,T)$ is $(\infty,n)$-complete and $(\infty,n)$-cocomplete.
\end{cor}

\section{An explicit Yoneda map -- Proof of \texorpdfstring{\cref{prop:existYoneda}}{Proposition 7.5.5}} \label{sec:Yonedamap}
\label{SectionYonedamap}

Recall that, using the fibrational version of the Yoneda lemma from \cref{prop:Yonedamap}, we can build from any element in the fiber of a double $(\infty,n-1)$-right fibration a \emph{Yoneda map}. In \cref{sec:Yonedamapconstruction}, we construct a specific implementation of such a Yoneda map in the case where the double $(\infty,n-1)$-right fibration is the canonical projection $\joinslice{\Nh\cC}{f}\to \Nh\cC$ from a neat cone construction, where $\cC$ is a $\Thnsset$-enriched category. We then prove in \cref{sec:Yonedamapproperty} that, in \cref{setting:limit} and the case where $f=F^\flat\circ p\colon A\to \Nh\cC$, this Yoneda map is appropriately compatible with its enriched version and hence satisfies the conditions of \cref{prop:existYoneda}.

\subsection{Construction of the Yoneda map}
\label{sec:Yonedamapconstruction}

Given a $\Thnsset$-enriched category $\cC$, an object $\ell\in\cC$, a map $f\colon A\to \Nh\cC$ in $\sThnsset$, and an element $\lambda^\flat\in ((\slice{\Nh\cC}{f})_\ell)_{[0],0}$ of the fiber of $\slice{\Nh\cC}{f}\to \Nh\cC$ at $\ell$, in order to construct a map of the form 
\[ \lambda^\flat_*\colon \joinslice{\Nh\cC}{\ell}\to \joinslice{\Nh\cC}{f}, \] we need to associate to every simplex $\sigma\colon F[m,X]\to \joinslice{\Nh\cC}{\ell}$  a simplex $\lambda^\flat_*(\sigma)\colon F[m,X]\to \joinslice{\Nh\cC}{f}$ in $\sThnssetslice{\Nh\cC}$. By passing to adjoints, this amounts to associating to every $\Thnsset$-enriched functor $\sigma^\sharp\colon \Ch L (F[m,X]\star F[0]) \to \cC$ a $\Thnsset$-enriched functor $(\lambda^\flat_*(\sigma))^\sharp\colon \Ch L(F[m,X]\star A)\to \cC$, which are both appropriately compatible with the given data. As the element $\lambda^\flat\in ((\slice{\Nh\cC}{f})_\ell)_{[0],0}$ corresponds to a $\Thnsset$-enriched functor $\lambda\colon \Ch L(F[0]\star A)\to \cC$, using that $\cC$ admits strictly associative compositions, we aim to define $(\lambda^\flat_*(\sigma))^\sharp$ to be the composite of $\sigma^\sharp$ and $\lambda$ along their common summit $\ell$ . 

To model this, we construct for all objects $A,B\in \sThnsset$ a $\Thnsset$-enriched functor \[ \Phi_{A,B}\colon \Ch L (A\star B)\to \Ch L(A\star \repD[0])\amalg_{[0]} \Ch L (\repD[0]\star B) \]
which acts in the desired way. This functor factors through the following $\Thnsset$-enriched category.

\begin{notation}
    Let $A$ and $B$ be objects in $\sThnsset$. We consider the following pullback in $\Thncat$,
    \begin{tz}
        \node[](1) {$\Ch L(A\star F[0])\times_{[1]} \Ch L(F[0]\star B)$}; 
        \node[below of=1](2) {$\Ch L(A\star F[0])$}; 
        \node[right of=1,xshift=3.9cm](3) {$\Ch L(F[0]\star B)$}; 
        \node[below of=3](4) {$\Ch L(F[0]\star F[0])$}; 
        \pullback{1}; 

        \draw[->](1) to (2); 
        \draw[->](1) to (3); 
        \draw[->](2) to node[below,la]{$\Ch L(!\star F[0])$} (4);  \draw[->](3) to node[right,la]{$\Ch L(F[0]\star !)$} (4); 
    \end{tz}
    where we recall that $\Ch L(F[0]\star F[0])\cong \Ch L F[1]\cong [1]$.  
\end{notation}

It is straightforward to see that the above pullback admits the following description. 

\begin{lemma}
    Let $A$ and $B$ be objects in $\sThnsset$. Then the $\Thnsset$-enriched category $\Ch L(A\star F[0])\times_{[1]} \Ch L(F[0]\star B)$ is isomorphic to the $\Thnsset$-enriched category $\cP$ with 
    \begin{itemize}[leftmargin=0.6cm]
        \item object set $\Ob\cP= A_{0,[0],0}\amalg B_{0,[0],0}$, 
        \item given objects $x,y\in A_{0,[0],0}\amalg B_{0,[0],0}$, hom $\Thn$-space 
        \[ \Hom_\cP(x,y)= \begin{cases}
            \Hom_{\Ch L(A\star F[0])}(x,y) & \text{if } x,y\in A_{0,[0],0} \\
            \Hom_{\Ch L(F[0]\star B)}(x,y) & \text{if } x,y\in B_{0,[0],0} \\
            \Hom_{\Ch L(A\star F[0])}(x,\top)\times \Hom_{\Ch L(F[0]\star B)}(\bot,y) & \text{if } x\in A_{0,[0],0}, \; y\in B_{0,[0],0} \\
            \emptyset & \text{if } x\in B_{0,[0],0}, \; y\in A_{0,[0],0}
        \end{cases}\]
        where $\top$ (resp.~$\bot$) denotes the object of $\Ch L(A\star F[0])$ (resp.~$\Ch L(F[0]\star B)$) coming from $F[0]$,
        \item composition is induced by the compositions of $\Ch L(A\star F[0])$ and $\Ch L(F[0]\star B)$.
    \end{itemize}
\end{lemma}

We now construct a map from the above pullback into the desired pushout, which is described as follows. 

\begin{notation}
    Let $A$ and $B$ be objects in $\sThnsset$. We consider the following pushout in $\Thncat$, 
    \begin{tz}
        \node[](1) {$[0]$}; 
        \node[below of=1](2) {$\Ch L(A\star F[0])$}; 
        \node[right of=1,xshift=3.9cm](3) {$\Ch L(F[0]\star B)$}; 
        \node[below of=3](4) {$\Ch L(A\star F[0])\amalg_{[0]} \Ch L(F[0]\star B)$}; 
        \pushout{4}; 

        \draw[->](1) to node[left,la]{$\top$} (2); 
        \draw[->](1) to node[above,la]{$\bot$} (3); 
        \draw[->](2) to (4);  \draw[->](3) to (4); 
    \end{tz}
    where $\top$ (resp.~$\bot$) denotes the object of $\Ch L(A\star F[0])$ (resp.~$\Ch L(F[0]\star B)$) coming from $F[0]$.
\end{notation}

\begin{constr}
    Let $A$ and $B$ be objects in $\sThnsset$. We construct a $\Thnsset$-enriched functor 
    \[ \Psi_{A,B}\colon \Ch L(A\star F[0])\times_{[1]} \Ch L(F[0]\star B)\to \Ch L(A\star F[0])\amalg_{[0]} \Ch L(F[0]\star B)\]
    given by
    \begin{itemize}[leftmargin=0.6cm]
        \item sending an object in $A_{0,[0],0}$ to the corresponding object in $\Ch L(A\star F[0])$ and an object in $B_{0,[0],0}$ to the corresponding object in $\Ch L(F[0]\star B)$ seen as objects of $\cC\coloneqq \Ch L(A\star F[0])\amalg_{[0]} \Ch L(F[0]\star B)$ through the canonical inclusions, 
        \item for $x,y\in A_{0,[0],0}$ (resp.~$x,y\in B_{0,[0],0}$), the map on hom $\Thn$-spaces induced by the canonical inclusion of $\Ch L(A\star F[0])$ (resp.~$\Ch L(F[0]\star B)$) into $\cC$, 
        \item for $x\in A_{0,[0],0}$ and $y\in B_{0,[0],0}$, the map on hom $\Thn$-spaces given by the composite
        \[ \Hom_{\Ch L(A\star F[0])}(x,\top)\times \Hom_{\Ch L(F[0]\star B)}(\bot,y)\to \Hom_\cC(x,\top)\times \Hom_{\cC}(\bot,y)\xrightarrow{\circ} \Hom_\cC(x,y), \]
        where the first map is induced by the canonical inclusions and where we use that $\top=\bot$ in~$\cC$. 
    \end{itemize}
    It is straightforward to see that $\Psi_{A,B}$ is a well-defined $\Thnsset$-enriched functor and that this construction is natural in $A$ and $B$. 
\end{constr}

We then get the desired $\Thnsset$-enriched functor as follows. 

\begin{constr} \label{constr:PhimX}
    Let $A$ and $B$ be objects in $\sThnsset$. We define the $\Thnsset$-enriched functor 
    \[ \Phi_{A,B}\colon \Ch L (A\star B)\to \Ch L(A\star \repD[0])\amalg_{[0]} \Ch L (\repD[0]\star B) \]
    to be the following composite 
    \[ \Ch L (A\star B)\xrightarrow{\Ch L (A\star !)\times_{[1]} \Ch L (!\star B)} \Ch L(A\star F[0])\times_{[1]} \Ch L(F[0]\star B)\xrightarrow{\Psi_{A,B}} \Ch L(A\star F[0])\amalg_{[0]} \Ch L(F[0]\star B). \]
    It is straightforward to see that this is natural in $A$ and $B$, as each part of the composite is so.   
\end{constr}

Using this $\Thnsset$-enriched functor, we can define the desired map $\lambda^\flat_*$.

\begin{constr} \label{constr:Yonedamap}
Let $\cC$ be a $\Thnsset$-enriched category and $\ell\in \cC$ be an object. Consider a map $f\colon A\to \Nh\cC$ in $\sThnsset$, and an element $\lambda^\flat\in ((\joinslice{\Nh\cC}{f})_\ell)_{[0],0}$ of the fiber of $\joinslice{\Nh\cC}{f}\to \Nh\cC$ at $\ell$. We construct a map in $\sThnssetslice{\Nh\cC}$
\[ \lambda^\flat_*\colon \joinslice{\Nh\cC}{\ell}\to \joinslice{\Nh\cC}
{f} \]
by giving its action on simplices. Let $m\geq 0$ and $X$ be a connected $\Thn$-space. A simplex 
    \begin{tz}
    \node[](1) {$F[m,X]$}; 
    \node[below right of=1,xshift=.4cm](2) {$\Nh\cC$}; 
    \node[above right of=2,xshift=.4cm](3) {$\joinslice{\Nh\cC}{\ell}$}; 

    \draw[->] (1) to node[above,la]{$\sigma$} (3); 
    \draw[->] (1) to node[left,la,yshift=-5pt]{$\tau$} (2);
    \draw[->] (3) to (2);
\end{tz}
corresponds, by definition, to a commutative diagram in $\sThnsset$ as below left, which by the adjunction $\Ch L \dashv \Nh$ corresponds to a commutative diagram in $\Thncat$ as below right.
\begin{tz}
        \node[](1) {$F[m,X]\star \repD[0]$}; 
        \node[right of=1,xshift=1.5cm](2) {$\Nh \cC$}; 
        \node[above of=1](3) {$F[m,X]\amalg \repD[0]$}; 
        \draw[->](1) to node[below,la]{$\sigma$} (2); 
        \draw[right hook->](3) to (1);
        \draw[->](3) to  node[right,la,yshift=5pt]{$\tau +\ell$} (2); 

        \node[right of=1,xshift=5.5cm](1) {$\Ch L(F[m,X]\star \repD[0])$}; 
        \node[right of=1,xshift=1.7cm](2) {$\cC$}; 
        \node[above of=1](3) {$\Ch LF[m,X]\amalg [0]$}; 
        \draw[->](1) to node[below,la]{$\sigma^\sharp$} (2); 
        \draw[right hook->](3) to (1);
        \draw[->](3) to  node[right,la,yshift=5pt]{$\tau^\sharp +\ell$} (2); 
    \end{tz}
    Moreover, the element $\lambda^\flat$ corresponds to a commutative diagram in $\sThnsset$ as below left, which by adjunction $\Ch L\dashv \Nh$ corresponds to a commutative diagram in $\Thncat$ as below right.
    \begin{tz}
        \node[](1) {$\repD[0]\star A$}; 
        \node[right of=1,xshift=1.2cm](2) {$\Nh \cC$}; 
        \node[above of=1](3) {$\repD[0]\amalg A$}; 
        \draw[->](1) to node[below,la]{$\lambda^\flat$} (2); 
        \draw[right hook->](3) to (1);
        \draw[->](3) to  node[right,la,yshift=5pt]{$\ell+f$} (2); 

        \node[right of=1,xshift=4.7cm](1) {$\Ch L(\repD[0]\star A)$}; 
        \node[right of=1,xshift=1.4cm](2) {$\cC$}; 
        \node[above of=1](3) {$[0]\amalg \Ch L A$}; 
        \draw[->](1) to node[below,la]{$\lambda$} (2); 
        \draw[right hook->](3) to (1);
        \draw[->](3) to  node[right,la,yshift=5pt]{$\ell+f^\sharp$} (2); 
    \end{tz}
    Hence, together, they yield a map 
    \[ \sigma^\sharp+\lambda\colon \Ch L(F[m,X]\star \repD[0])\amalg_{[0]} \Ch L(\repD[0]\star A)\to \cC. \]
    By precomposing with the $\Thnsset$-enriched functor from \cref{constr:PhimX}, we get
    \begin{tz}
        \node[](1) {$\Ch L (F[m,X]\star A)$}; 
        \node[right of=1,xshift=4.8cm](2) {$\Ch L(F[m,X]\star \repD[0])\amalg_{[0]} \Ch L(\repD[0]\star A)$}; 
        \node[right of=2,xshift=3.4cm](2') {$\cC$};
        \node[above of=1](3) {$\Ch LF[m,X]\amalg \Ch L A$};
        \draw[->](1) to node[below,la]{$\Phi_{F[m,X],A}$} (2); 
        \draw[->](2) to node[below,la]{$\sigma^\sharp+\lambda$} (2'); 
        \draw[right hook->](3) to (1);
        \draw[->,bend left=5](3) to  node[right,la,yshift=5pt]{$\tau^\sharp+f^\sharp$} (2'.north west);
    \end{tz}
    and we set $\lambda^\flat_*(\sigma)\colon F[m,\theta,k]\to \Nh\cC$ to be the adjunct of the composite $ (\sigma^\sharp+\lambda)\circ \Phi_{F[m,X],A}$ under the adjunction $\Ch L\dashv \Nh$. In particular, we have a commutative diagram in $\sThnsset$ as below left, which gives a simplex as below right, as desired.
    \begin{tz}
        \node[](1) {$F[m,X]\star A$}; 
        \node[right of=1,xshift=1.6cm](2) {$\Nh \cC$}; 
        \node[above of=1](3) {$F[m,X]\amalg A$}; 
        \draw[->](1) to node[below,la]{$\lambda^\flat_*(\sigma)$} (2); 
        \draw[right hook->](3) to (1);
        \draw[->](3) to  node[right,la,yshift=5pt]{$\tau+f$} (2); 
        
    \node[right of=3,xshift=4.8cm,yshift=-.25cm](1) {$\repDThnS$}; 
    \node[below right of=1,xshift=.5cm](2) {$\Nh\cC$}; 
    \node[above right of=2,xshift=.5cm](3) {$\joinslice{\Nh\cC}{f}$}; 
    \draw[->] (1) to node[above,la]{$\lambda^\flat_*(\sigma)$} (3); 
    \draw[->] (1) to node[left,la,yshift=-5pt]{$\tau$} (2);
    \draw[->] (3) to (2);
\end{tz}

By naturality of the bijections induced by the adjunction $\Ch L\dashv \Nh$ and the naturality of $\Phi_{F[m,X],A}$ in $F[m,X]$, this assignment is natural in $\sigma$.
\end{constr}

\subsection{Comparison between enriched and internal approaches to Yoneda maps}
\label{sec:Yonedamapproperty}

The map from \cref{constr:Yonedamap} in \cref{setting:limit} and the case where $f=F^\flat\circ p$ induces, for every object $c\in \cC$, a map on fibers $(\lambda_*^\flat)_c\colon (\joinslice{\Nh\cC}{\ell})_c\to (\joinslice{\Nh\cC}{F^\flat\circ p})_c$. We now show that this map satisfies the properties of \cref{prop:existYoneda}. For this, we unpack the component  \[ \lambda^*_c\colon \Hom_\cC(c,\ell)\to \Hom_{[\Ch J,\Thnsset]}(\St_J(p),\Hom_\cC(c,F-)) \]
of the $\Thnsset$-enriched natural transformation induced by precomposing with the cone $\lambda$ in the enriched setting. We first study the following map which is used to define the map $\lambda^*_c$.

\begin{constr} \label{constr:composition}
    Let $\cC$ be a $\Thnsset$-enriched category and $c,\ell\in \cC$ be objects. The composition maps of $\cC$ induce a $\Thnsset$-enriched natural transformation in $[\cC^{\op},\Thnsset]$
\[ \circ_{c,\ell,-}\colon \Hom_\cC(c,\ell)\otimes \Hom_\cC(\ell,-)\Rightarrow \Hom_\cC(c,-). \]
Using that $[\cC^{\op},\Thnsset]$ is tensored and cotensored over $\Thnsset$ and so we have an adjunction $\Hom_\cC(c,\ell)\otimes (-)\dashv (-)^{\Hom_\cC(c,\ell)}$, this corresponds to a $\Thnsset$-enriched natural transformation in $[\cC^{\op},\Thnsset]$
\[ \circ^\sharp_{c,\ell,-}\colon \Hom_\cC(\ell,-)\Rightarrow \Hom_\cC(c,-)^{\Hom_\cC(c,\ell)}. \]
By applying the $\Thnsset$-enriched functor $\Un_\cC\colon [\cC^{\op},\Thnsset]\to \sThnssetslice{\Nh\cC}$, this yields a map in $\sThnssetslice{\Nh\cC}$
\[ \Un_\cC \, \circ^\sharp_{c,\ell,-}\colon \Un_\cC \Hom_\cC(\ell,-)\to \Un_\cC(\Hom_\cC(c,-)^{\Hom_\cC(c,\ell)})\cong \{\Hom_\cC(c,\ell),\Un_\cC\Hom_\cC(c,-)\}
\]
where $\{-,-\}$ denotes the cotensor of $\sThnssetslice{\Nh\cC}$ over $\Thnsset$ and the last isomorphism holds since the adjunction $\St_\cC\dashv \Un_\cC$ is enriched over $\Thnsset$ by \cref{st-unst}.
\end{constr}

Recall the isomorphic objects $P(A,Y)$ and $P'(A,Y)$ of $\sThnsset$ from \cref{lem:equalpushoutneat}.

\begin{lemma} \label{lem:joinvsP}
    Let $A$ be an object in $\sThnsset$ and $Y$ be a connected $\Thn$-space. There is a canonical isomorphism in $\pcatThn$
    \[ L(\repD[0,Y]\star A) \cong LP(A,Y)\cong LP'(A,Y). \]
\end{lemma}

\begin{proof}
    Since $LF[0,Y]\cong F[0]$, this follows directly from applying the colimit-preserving functor $L\colon \sThnsset\to \pcatThn$ to the pushout from \cref{lem:equalpushoutneat} defining $P(A,Y)$. 
\end{proof}

\begin{lemma} \label{lem:bijection}
    Let $\cC$ be a $\Thnsset$-enriched category and $\ell\in \cC$ be an object. Let $A$ be an object in $\sThnsset$ and $Y$ be a connected $\Thn$-space. There is a bijection between the set of maps in $\sThnssetslice{\Nh\cC}$ as below left, and the set of $\Thnsset$-enriched functors $\sigma^\sharp\colon \Ch L(F[0,Y]\star A)\to \cC$ making the below right diagram in $\Thncat$ commute,
    \begin{tz}
    \node[](1) {$A$}; 
    \node[below right of=1,xshift=.8cm](2) {$\Nh\cC$}; 
    \node[above right of=2,xshift=.8cm](3) {$\{Y,\Un_\cC \Hom_\cC(\ell,-)\}$}; 

    \draw[->] (1) to node[above,la]{$\sigma$} (3); 
    \draw[->] (1) to node[left,la,yshift=-5pt]{$\tau$} (2);
    \draw[->] (3) to (2);
    
         \node[right of=2,xshift=5.6cm,yshift=-.25cm](1) {$\Ch L(F[0,Y]\star A)$}; 
        \node[right of=1,xshift=1.7cm](2) {$\cC$}; 
        \node[above of=1](3) {$[0]\amalg \Ch L(A)$}; 
        \draw[->](1) to node[below,la]{$\sigma^\sharp$} (2); 
        \draw[right hook->](3) to (1);
        \draw[->](3) to  node[right,la,yshift=5pt]{$c+\tau^\sharp$} (2);
    \end{tz}
    where $\tau^\sharp\colon \Ch L(A)\to \cC$ denotes the adjunct of $\tau$. 
\end{lemma}

\begin{proof}
    By \cref{Unvsjoinslice}, the map $\sigma$ in $\sThnssetslice{\Nh\cC}$ corresponds to a map in $\sThnssetslice{\Nh\cC}$
    \[ \sigma\colon A\to \{Y,\joinslice{\Nh\cC}{\ell}\}. \]
    Using that $\sThnssetslice{\Nh\cC}$ is tensored and cotensored over $\Thnsset$ and so we have an adjunction $Y\otimes (-)\dashv \{Y,-\}$, this corresponds to a map in $\sThnssetslice{\Nh\cC}$
    \[ \sigma\colon A\otimes Y\to \joinslice{\Nh\cC}{\ell}. \]
    By definition of the neat slice, this corresponds to a commutative diagram in $\sThnsset$ as below left, which by the adjunction $\Ch L\dashv \Nh$ corresponds to a commutative diagram in $\Thncat$ as below right.   
    \begin{tz}
        \node[](1) {$P'(A,Y)$}; 
        \node[right of=1,xshift=1.3cm](2) {$\Nh\cC$}; 
        \node[above of=1](3) {$F[0]\amalg A$}; 
        \draw[->](1) to node[below,la]{$\sigma$} (2); 
        \draw[right hook->](3) to (1);
        \draw[->](3) to  node[right,la,yshift=5pt]{$\ell+\tau$} (2); 
        
        \node[right of=2,xshift=2.2cm](1) {$\Ch L(P'(A,Y))$}; 
        \node[right of=1,xshift=1.5cm](2) {$\cC$}; 
        \node[above of=1](3) {$F[0]\amalg \Ch L(A)$}; 
        \draw[->](1) to node[below,la]{$\sigma^\sharp$} (2); 
        \draw[right hook->](3) to (1);
        \draw[->](3) to  node[right,la,yshift=5pt]{$\ell+\tau^\sharp$} (2); 
    \end{tz}
    We then conclude using the isomorphism $LP'(A,Y)\cong L(F[0,Y]\star A)$ from \cref{lem:joinvsP}. 
\end{proof}

We unpack the action of the map from \cref{constr:composition} on simplices. 

\begin{lemma} \label{lem:Unofcomposition}
    Let $\cC$ be a $\Thnsset$-enriched category and $c,\ell\in \cC$ be objects. Let $m\geq 0$, and $X$ and $Y$ be connected $\Thn$-spaces. Consider a map in $\sThnssetslice{\Nh\cC}$
    \begin{tz}
    \node[](1) {$F[m,X]$}; 
    \node[below right of=1,xshift=.8cm](2) {$\Nh\cC$}; 
    \node[above right of=2,xshift=.8cm](3) {$\Un_\cC \Hom_\cC(\ell,-)$}; 

    \draw[->] (1) to node[above,la]{$\sigma$} (3); 
    \draw[->] (1) to node[left,la,yshift=-5pt]{$\tau$} (2);
    \draw[->] (3) to (2);
\end{tz}
and a map $f\colon Y\to \Hom_\cC(c,\ell)$ in $\Thnsset$. Then the composite in $\sThnssetslice{\Nh\cC}$
\[ F[m,X]\xrightarrow{\sigma}\Un_\cC\Hom_\cC(\ell,-)\xrightarrow{\Un_\cC \, \circ^\sharp_{c,\ell,-}} \{\Hom_\cC(c,\ell),\Un_\cC\Hom_\cC(c,-)\}\xrightarrow{f^*} \{Y,\Un_\cC\Hom_\cC(c,-)\} \]
corresponds under the bijection from \cref{lem:bijection} to the composite in $\Thncat$
\[ \Ch L(F[0,Y]\star F[m,X]) \xrightarrow{\Phi_{F[0,Y],F[m,X]}} \Ch L(F[0,Y]\star F[0]) \amalg_{[0]} \Ch L(F[0]\star F[m,X]) \xrightarrow{f+F_{\sigma^\sharp}} \cC,  \]
where $\Phi_{F[0,Y],F[m,Y]}$ is the $\Thnsset$-enriched functor from \cref{constr:PhimX} and $F_{\sigma^\sharp}$ is the one corresponding to the adjunct $\sigma^\sharp\colon \St_\cC(\tau)\Rightarrow \Hom_\cC(\ell,-)$ of $\sigma$ under the bijection from \cref{lem:nattransfoutofStvsfunctors}.
\end{lemma}

\begin{proof}
    The composite 
\[ F[m,X]\xrightarrow{\sigma}\Un_\cC\Hom_\cC(\ell,-)\xrightarrow{\Un_\cC \, \circ^\sharp_{c,\ell,-}} \{\Hom_\cC(c,\ell),\Un_\cC\Hom_\cC(c,-)\}\xrightarrow{f^*} \{Y,\Un_\cC\Hom_\cC(c,-)\} \]
corresponds by the $\Thnsset$-enriched adjunction $\St_\cC\dashv \Un_\cC$ to the composite in $[\cC^{\op},\Thnsset]$
\[ Y\otimes \St_\cC(\tau)\xRightarrow{f\otimes \sigma^\sharp} \Hom_\cC(c,\ell)\otimes \Hom_\cC(\ell,-)\xRightarrow{\circ_{c,\ell,-}} \Hom_\cC(c,-). \]
Moreover, by \cref{lem:nattransfoutofStvsfunctors}, the $\Thnsset$-enriched natural transformation $\sigma^\sharp$ corresponds to a commutative diagram in $\Thncat$ as below left, and by the adjunction $\Sigma\dashv \Hom$, the map~$f$ corresponds to a commutative diagram in $\Thncat$ as below right.
\begin{tz}
       \node[](1) {$\Ch L(F[0]\star F[m,X])$}; 
        \node[right of=1,xshift=1.7cm](2) {$\cC$}; 
        \node[above of=1](3) {$[0]\amalg \Ch L F[m,X]$}; 
        \draw[->](1) to node[below,la]{$F_{\sigma^\sharp}$} (2); 
        \draw[right hook->](3) to (1);
        \draw[->](3) to  node[right,la,yshift=5pt]{$\ell+\tau^\sharp$} (2);

       \node[right of=1,xshift=6.3cm](1) {$\Ch L(Y\star F[0])$}; 
       \node[left of=1,xshift=-.3cm,yshift=1pt]{$\Sigma Y\cong$};
        \node[right of=1,xshift=1.3cm](2) {$\cC$}; 
        \node[above of=1](3) {$[0]\amalg [0]$}; 
        \draw[->](1) to node[below,la]{$f$} (2); 
        \draw[right hook->](3) to (1);
        \draw[->](3) to  node[right,la,yshift=5pt]{$c+\ell$} (2); 
   \end{tz}
   Finally, postcomposition with the composition maps corresponds to the following composite
   \[ \Ch L(F[0,Y]\star F[m,X]) \xrightarrow{\Phi_{F[0,Y],F[m,X]}} \Ch L(F[0,Y]\star F[0]) \amalg_{[0]} \Ch L(F[0]\star F[m,X]) \xrightarrow{f+F_{\sigma^\sharp}} \cC,  \]
   as desired.
\end{proof}

Finally, we can unpack the desired map. 

\begin{prop} \label{unpacklambdastar}
    Let $J$ be an object in $\pcatThn$ and $\cC$ be a $\Thnsset$-enriched category. Let $F\colon \Ch J\to \cC$ be a $\Thnsset$-enriched functor and $p\colon A\to J$ be a map in $\sThnsset$ with straightening $\St_J(p)\colon \Ch J\to \Thnsset$. Let $f\colon Y\to \Hom_\cC(c,\ell)$ be a map in $\Thnsset$ from a connected $\Thn$-space $Y$. Then the composite in $\Thnsset$
    \[ Y\xrightarrow{f}\Hom_\cC(c,\ell)\xrightarrow{\lambda^*_c} \Hom_{[\Ch J,\Thnsset]}(\St_J(p),\Hom_\cC(c,F-)) \]
corresponds to the composite in $\Thncat$
\[ \Ch L(F[0,Y]\star A) \xrightarrow{\Phi_{F[0,Y],A}} \Ch L(F[0,Y]\star F[0]) \amalg_{[0]} \Ch L(F[0]\star A) \xrightarrow{f+\lambda} \cC,  \]
where $\Phi_{F[0,Y],A}$ is the $\Thnsset$-enriched functor from \cref{constr:PhimX}.
\end{prop}

\begin{proof}
    Recall that the map $\lambda^*_c$ is given by the following composite
    \begin{tz}
        \node[](1) {$\Hom_\cC(c,\ell)$}; 
        \node[right of=1,xshift=6.3cm](2) {$\Hom_\cC(c,\ell)\otimes \Hom_{\widehat{\Ch J}}(\St_J(p),\Hom_\cC(\ell,F-))$}; 
        \node[below of=2](3) {$\Hom_{\widehat{\Ch J}}(\St_J(p),\Hom_\cC(c,\ell)\otimes\Hom_\cC(\ell,F-))$}; 
        \node[below of=1](4) {$\Hom_{\widehat{\Ch J}}(\St_J(p),\Hom_\cC(c,F-))$}; 
        \draw[->](1) to node[above,la]{$\id\times\{\lambda\}$} (2); 
        \draw[->](2) to node[right,la]{$\cong$} (3); 
        \draw[->](3) to node[below,la]{$(\circ_{c,\ell,F-})_*$} (4); 
        \draw[->] (1) to node[left,la]{$\lambda^*_c$} (4);
    \end{tz}
    where $\widehat{\Ch J}=[\Ch J,\Thnsset]$. 
   Moreover, by the proof of \cref{lem:webtwfibers}, we have a natural isomorphism
   \[ \Hom_{[\Ch J,\Thnsset]}(\St_J(p),\Hom_\cC(c,F-))\cong \Hom_{\sThnssetslice{\Nh\cC}}(A\xrightarrow{F^\flat\circ p} \Nh\cC,\Un_\cC \Hom_\cC(c,-)\to \Nh\cC)\]
and the corresponding map is given by the following composite
\begin{tz}
        \node[](1) {$F[0]$}; 
        \node[right of=1,xshift=6.9cm](2) {$\Hom_{/\Nh\cC}(A,\Un_\cC \Hom_\cC(\ell,-))$}; 
        \node[below of=2](3) {$\Hom_{/\Nh\cC}(A,\{\Hom_\cC(c,\ell),\Un_\cC \Hom_\cC(c,-)\})$}; 
        \node[below of=1](4) {$\Hom_{/\Nh\cC}(A,\Un_\cC \Hom_\cC(c,-))^{\Hom_\cC(c,\ell)}$}; 
        \draw[->](1) to node[above,la]{$\lambda$} (2); 
        \draw[->](2) to node[right,la]{$(\Un_\cC \, \circ^\sharp_{c,\ell,-})_*$} (3); 
        \draw[->](3) to node[above,la]{$\cong$} (4); 
        \draw[->] (1) to node[left,la]{$\lambda^*_c$} (4);
    \end{tz}
    where $\Hom_{/\Nh\cC}(-,-)=\Hom_{\sThnssetslice{\Nh\cC}}(-,-)$ and we denote the maps over $\Nh\cC$ by just their source. Hence the result follows from \cref{lem:Unofcomposition} and the naturality of $\Phi_{F[0,Y],A}$ in $A$.  
\end{proof}

Finally, we prove that the map from \cref{constr:Yonedamap} associated to the element $\lambda^\flat$ as in \cref{prop:existYoneda} satisfies the property of \cref{prop:existYoneda}.

\begin{lemma} \label{lem:diagramcommutes} 
In \cref{setting:limit}, let $\ell\in \cC$ be an object, $\lambda\colon \St_J(p)\Rightarrow \Hom_\cC(\ell,F-)$ be a $\Thnsset$-enriched natural transformation in $[\Ch J,\Thnsset]$, and $\lambda^\flat \in ((\joinslice{\Nh\cC}{F^\flat \circ p})_\ell)_{[0],0}$ be as in \cref{notn:enrvsfibcones}. Then, for every object $c\in \cC$, the following diagram in $\Thnsset$ commutes,
\begin{tz}
        \node[](1) {$\Hom_\cC(c,\ell)$}; 
        \node[below of=1](2) {$\Hom_{\Nh\cC}(c,\ell)$}; 
        \node[right of=1,xshift=4.3cm,yshift=-3pt](3) {$\Hom_{[\Ch J,\Thnsset]}(\St_J(p),\Hom_\cC(c,F-))$}; 
        \node[below of=3,yshift=3pt](4) {$(\joinslice{\Nh\cC}{F^\flat\circ p})_c$}; 

        \draw[d] (1) to (2); 
        \draw[->] (1) to node[above,la]{$\lambda^*_c$} ($(3.west)+(0,3pt)$); 
        \draw[->] (2) to node[below,la]{$(\lambda^\flat_*)_c$} (4); 
        \draw[->] (3) to node[right,la]{$\cong$} (4); 
    \end{tz}
    where the right-hand map is the isomorphism from \cref{lem:webtwfibers} and the bottom map is the map induced on fibers by the map $\lambda^\flat_*$ from \cref{constr:Yonedamap} associated to the element $\lambda^\flat$.
\end{lemma}

\begin{proof}
    This follows directly from the unpacking of $\lambda^*_c$ in \cref{unpacklambdastar} and the construction of the map $\lambda^\flat_*$ as described in \cref{constr:Yonedamap}. 
\end{proof}

In particular, we have constructed a map $\lambda^\flat_*$ satisfying the properties of \cref{prop:existYoneda} and so this concludes the proof. 

\bibliographystyle{amsalpha}
\bibliography{ref}

\end{document}